\documentclass[11pt,final,a4paper]{article}

\usepackage[english]{babel} 
\usepackage{amsmath} 
\usepackage[utf8]{inputenc} 
\usepackage[T1]{fontenc} 
\usepackage[svgnames]{xcolor}

\usepackage{amssymb}           
\usepackage{bbm}               
\usepackage{mathtools}         
\usepackage[amsmath,thmmarks,hyperref]{ntheorem} 
\usepackage{mathrsfs}          
\usepackage{stmaryrd}          
\usepackage{paralist}          
\usepackage{aliascnt}          
\usepackage[sort]{cite} 
\usepackage{subfig}
\usepackage{eucal}
\newcommand{\orbitscript}{}
\let\orbitscript=\mathcal
\renewcommand{\mathcal}{\CMcal}
\usepackage[a4paper]{geometry}
\usepackage{multirow}
\usepackage{enumitem}
\usepackage{tikz-cd}
\usepackage{tikzsymbols}
\usepackage[labelfont=bf]{caption}
\usepackage[draft=false,colorlinks,hypertexnames=false]{hyperref}
\hypersetup{allcolors={NavyBlue}} 
\usepackage{appendix}
\usetikzlibrary{patterns}
\usetikzlibrary{backgrounds}

\mathchardef\mhyphen="2D

\title{\vspace*{-.75cm}\huge{Strict quantization of coadjoint orbits}}
\author{
	\textbf{Philipp Schmitt}\thanks{Current address: Leibniz Universität Hannover, \texttt{schmitt@math.uni-hannover.de}}\\Department of Mathematical Sciences, University of Copenhagen, Denmark\vspace*{-.75cm}
}
\date{}

\renewcommand{\mathbb}[1]{\mathbbm{#1}}

\newcommand{\refitem}[1] {\textit{\ref{#1}.)}}

\numberwithin{equation}{section}

\allowdisplaybreaks

\renewcommand{\arraystretch}{1.2}

\let\originalleft\left
\let\originalright\right
\renewcommand{\left}{\mathopen{}\mathclose\bgroup\originalleft}
\renewcommand{\right}{\aftergroup\egroup\originalright}

\renewcommand{\cleardoublepage}{\clearpage\ifodd\c@page\else\vspace*{\fill}\thispagestyle{empty}\newpage\fi}

\newtheoremstyle{mytheoremstyle}{\item[\hskip\labelsep\normalfont\bfseries{##1\ ##2}]}{\item[\hskip\labelsep \normalfont{\bfseries ##1\ ##2} (##3)]}
\theoremstyle{mytheoremstyle}

\theoremheaderfont{\normalfont\bfseries}
\theorembodyfont{\itshape}
\newtheorem{lemma}{Lemma}[section]

\newaliascnt{proposition}{lemma}
\newtheorem{proposition}[proposition]{Proposition}
\aliascntresetthe{proposition}

\newaliascnt{thm}{lemma}
\newtheorem{theorem}[thm]{Theorem}
\aliascntresetthe{thm}

\newaliascnt{corollary}{lemma}
\newtheorem{corollary}[corollary]{Corollary}
\aliascntresetthe{corollary}

\newaliascnt{definition}{lemma}
\newtheorem{definition}[definition]{Definition}
\aliascntresetthe{definition}

\newaliascnt{claim}{lemma}

\aliascntresetthe{claim}

\theorembodyfont{\rmfamily}

\newaliascnt{example}{lemma}
\newtheorem{example}[example]{Example}
\aliascntresetthe{example}

\newaliascnt{remark}{lemma}
\newtheorem{remark}[remark]{Remark}
\aliascntresetthe{remark}

\newaliascnt{convention}{lemma}

\aliascntresetthe{convention}

\theorembodyfont{\itshape}
\theoremnumbering{Roman}
\newtheorem{maintheorem}{Main Theorem}

\makeatletter
\def\theorem@checkbold{}
\makeatother

\theoremheaderfont{\scshape}
\theorembodyfont{\normalfont}
\theoremstyle{nonumberplain}
\theoremseparator{:}
\theoremsymbol{\hbox{$\boxempty$}}
\newtheorem{proof}{Proof}
\theoremsymbol{\hbox{$\triangledown$}}

\newenvironment{lemmalist}{\begin{compactenum}[\itshape i.)]}{\end{compactenum}}

\newenvironment{propositionlist}{\begin{compactenum}[\itshape 
i.)]}{\end{compactenum}}
\newenvironment{definitionlist}{\begin{compactenum}[\itshape 
i.)]}{\end{compactenum}}

\newcommand{\I}              {\mathrm{i}}
\newcommand{\E}              {\mathrm{e}}
\newcommand{\D}              {\mathop{}\!\mathrm{d}}

\newcommand{\group}[1]        {\mathrm{#1}}
\newcommand{\algebra}[1]      {\mathscr{#1}}
\newcommand{\lie}[1]          {\mathfrak{#1}}
\DeclareMathOperator{\Pol}    {\mathrm{Pol}}
\DeclareMathOperator{\ad}     {\mathrm{ad}}
\DeclareMathOperator{\Ad}     {\mathrm{Ad}}
\newcommand{\Fun}[1][k]      {\mathscr{C}^{#1}}
\newcommand{\Cinfty}         {\Fun[\infty]}
\newcommand{\acts}            {\mathbin{\triangleright}}
\newcommand{\Sec}[1][k]      {\Gamma^{#1}}
\newcommand{\Secinfty}       {\Sec[\infty]}
\DeclareMathOperator{\Diffop}         {\mathrm{DiffOp}}

\newcommand{\at}[1]          {\big|_{#1}}
\newcommand{\At}[1]          {\Big|_{#1}}

\newcommand{\tensor}[1][{}]           
{\mathbin{\otimes_{\scriptscriptstyle{#1}}}}

\DeclarePairedDelimiter{\abs}{\lvert}{\rvert}
\DeclarePairedDelimiter{\norm}{\lVert}{\rVert}
\newcommand{\Sym}                     {\mathrm{S}}
\DeclareMathOperator{\diag}  {\mathrm{diag}}
\newcommand{\ev}             {\mathrm{ev}}
\newcommand{\id}             {\mathsf{id}}
\newcommand{\opp}            {\mathrm{opp}}
\newcommand{\racts}          {\mathbin{\triangleleft}}
\newcommand{\RE}             {\mathsf{Re}}
\newcommand{\Tensor}         {\mathrm{T}}
\DeclareMathOperator{\spann} {\mathrm{span}}

\theoremnumbering{arabic}
\theoremheaderfont{\scshape}
\theoremsymbol{\hbox{$\boxempty$}}
\theoremstyle{empty}
\newtheorem{proofof}{}

\newcommand{\anfa}{``}

\newcommand{\anfel}{''\ }

\newcommand{\textueber}[2]{\overset{\mathclap{\strut{#2}}}#1}
\newcommand{\textuebersign}[2]{\mathrel{\smash{\textueber{#1}{#2}}}}

\newcommand{\komma}{\, ,}

\newcommand{\punkt}{\, .}

\newcommand{\univ}{\algebra U}

\newcommand{\orb}[1]{\orbitscript O_{#1}}
\newcommand{\orbext}[1]{\hat{\orbitscript O}_{#1}}
\newcommand{\vanideal}[1]{\mathcal I(#1)}
\newcommand{\formParam}{\hbar}

\newcommand{\plus}[1]{{#1}^+}
\newcommand{\minus}[1]{{#1}^-}
\newcommand{\plusminus}[1]{{#1}^\pm}
\newcommand{\wordset}{W}
\DeclareMathOperator{\tensorhat}{\hat{\otimes}}
\newcommand{\formalstarhat}{\mathbin{\hat\star}}
\newcommand{\starhat}{\mathbin{\hat *_\hbar}}
\newcommand{\leftinv}[2][]{#2^{\mathrm{left}#1}}
\newcommand{\leftinvholo}[1]{#1^{\mathrm{left},(1,0)}}
\newcommand{\rightinv}[2][]{#2^{\mathrm{right}#1}}
\newcommand{\rightinvholo}[1]{#1^{\mathrm{right},(1,0)}}
\newcommand{\leftact}{\mathrm L}
\newcommand{\rightact}{\mathrm R}
\newcommand{\nDiffop}{\text{$k$-$\Diffop$}}
\newcommand{\analytics}{\mathcal A}
\newcommand{\Hol}{\mathrm{Hol}}

\newcommand{\holo}{\mathcal H}
\newcommand{\finites}[1]{\mathrm{Fin}^{#1}}
\newcommand{\argumentdot}{{}\cdot{}}
\newcommand{\invpi}{\pi_*}
\newcommand{\invpihat}{\hat\pi_*}
\newcommand{\formal}[1]{[\!\!\;[#1]\!\!\;]}
\newcommand{\Lzwei}{\mathrm{L}^{2}}

\newcommand{\SL}[1][n]{\group{SL}_{#1}(\mathbb C)}
\newcommand{\GL}[1][n]{\group{GL}_{#1}(\mathbb C)}
\newcommand{\liesl}[1][n]{\lie{sl}_{#1}(\mathbb C)}
\newcommand{\liegl}[1][n]{\lie{gl}_{#1}(\mathbb C)}
\newcommand{\SLR}[1][n]{\group{SL}_{#1}(\mathbb R)}
\newcommand{\GLR}[1][n]{\group{GL}_{#1}(\mathbb R)}

\newcommand{\Tangent}{\mathrm{T}}

\newcommand{\starhbar}{\mathbin{*_\hbar}}

\newcommand{\vertiii}[1]{{\left\vert\kern-0.25ex\left\vert\kern-0.25ex\left \vert #1\right\vert\kern-0.25ex\right\vert\kern-0.25ex\right\vert}}
\newcommand{\mynorm}[2]{\vertiii{#2}_{#1}}
\DeclareMathOperator{\trace}{tr}

\newcommand{\vermaacts}[1]{\mathbin{\acts_\lambda^{#1}}}
\newcommand{\vermaracts}{\mathbin{\racts_\lambda^*}}
\newcommand{\tildevermaacts}[1]{\mathbin{\tilde\acts_\lambda^{#1}}}
\newcommand{\tildevermaracts}{\mathbin{\tilde\racts_\lambda^*}}

\newcommand{\doi}[1]{\href{http://dx.doi.org/#1}{doi: #1}}
\newcommand{\arxiv}[1]{\href{http://arxiv.org/#1}{arxiv: #1}}

\makeatletter
\addto\extrasenglish{\@ifundefined{Appendixautorefname}{}{}
}	
\makeatother

\newcommand{\appendixref}[1]{\hyperref[#1]{Appendix~\ref*{#1}}}

\newsavebox{\mysecondimage}

\newcommand{\figurewithtwoimages}[4]{
 \begin{figure}
 	\centering
 	\savebox{\mysecondimage}{\hbox{#2}}
 	\hspace*{\fill}
 	\subfloat{\raisebox{\dimexpr.5\ht\mysecondimage-.5\height\relax}{#1}}
 	\hfill
 	\subfloat{\usebox{\mysecondimage}}
 	\hspace*{\fill}
 	\caption{#3}
 	\label{#4}
 \end{figure}
} 
\begin{document}
	\maketitle
	\begin{abstract}
		We obtain a family of strict $\hat G$-invariant products $\starhat$
		on the space of holomorphic functions on a semisimple coadjoint orbit $\orbext{}$ 
		of a complex connected semisimple Lie group $\hat G$.
		By restriction, we also obtain strict $G$-invariant products $\starhbar$
		on a space $\mathcal A(\orb{})$	of certain analytic functions on a semisimple coadjoint orbit $\orb{}$ 
		of a real connected semisimple Lie group $G$.
		The space $\mathcal A(\orb{})$ endowed with one of the products $\starhbar$
		is a Fr\'echet algebra,
		and the formal expansion of the products around $\hbar = 0$
		determines a formal deformation quantization of $\orb{}$,
		which is of Wick type if $G$ is compact.
		We study a generalization of a Wick rotation,
		which provides isomorphisms between the quantizations
		obtained for different real orbits with the same complexification. 
		Our construction relies on an explicit computation 
		of the canonical element of the Shapovalov pairing 
		between generalized Verma modules, 
		and complex analytic results on the extension of holomorphic functions.
	\end{abstract}
\noindent
{\small\textbf{Mathematics Subject Classification (2020)}: 53D55, 17B08, 32Q28, 22E46.\\[0.2cm]
}

\tableofcontents

\section*{Introduction}
\addcontentsline{toc}{section}{Introduction}

The quantization problem in physics asks how to associate a quantum 
system to a classical mechanical system, 
such that the classical system can be 
recovered from the quantum system in a classical limit.
Since both systems can be studied by their observable algebras,
a first step is to quantize the classical observable algebra.
This algebra is usually the Poisson algebra $\Cinfty(M)$
of smooth functions on a Poisson manifold $M$.
The observable algebra of a quantum mechanical system 
is some non-commutative $^*$-algebra $\mathcal A$,
which in many cases is obtained from a C$^*$-algebra.
In a second step, 
the states of the quantum mechanical system can be obtained as 
normalized positive linear functionals on $\mathcal A$.
To define their superposition,
one has to represent $\mathcal A$ on a (pre) Hilbert space,
so that the superposition of two vector states can be defined as 
the vector state corresponding to the sum of the two vectors.

\emph{Formal deformation quantization}, as introduced in  
\cite{bayen.et.al:1978a}, has proven to be a  
fruitful theory for studying some aspects of the quantization problem.
One views Planck's constant $\hbar$
as a formal parameter $\formParam$ and tries to find so-called
\emph{formal star products} $\star$ on $\mathcal A = \Cinfty(M)\formal\formParam$,
which may be thought of as the infinite jet of a full solution to
the quantization problem at $\hbar = 0$. 
These star products are associative $\mathbb C \formal\formParam$-bilinear
products for which $1 \in \Cinfty(M)$ is a unit and 
which satisfy the correct classical limit.
To be more precise, if $f, g \in \Cinfty(M)$ and 
$f \star g = \sum_{r=0}^\infty \formParam^r C_r(f,g)$
with operators $C_r \colon \Cinfty(M) \times \Cinfty(M) \to \Cinfty(M)$,
then one requires $C_0$ to be the pointwise multiplication, $C_0(f,g) = f g$, 
and the quantization to be in the direction of the Poisson bracket, 
$C_1(f,g) - C_1(g,f)= \I \lbrace f , g \rbrace$.
Usually one also requires the $C_r$ to be  bidifferential operators,
so that $\star$ is local and can be restricted to open subsets of $M$.
Using formal power series means on the one hand that we 
cannot substitute $\formParam$ with the real value of Planck's constant
as required for direct physical applications, but on the other hand
that we can transfer the quantization problem to algebra by neglecting analytic aspects,
such as convergence of the power series.
Consequently, many powerful tools become available for its study,
and existence and classification results were obtained in
\cite{bertelson.cahen.gutt:1997a, dewilde.lecomte:1983b, fedosov:1994a, nest.tsygan:1995a}
for symplectic manifolds,
whereas in the more general case of Poisson manifolds
they follow from Kontsevich's formality theorem \cite{kontsevich:2003a}.
One can also study formal star products that are 
equivariant with respect to the action of a Lie group, where the classification 
follows for example from \cite{dolgushev:2005a}.

A complete solution of the quantization problem consists of a Hilbert space $H$
together with a quantization map that associates a quantum observable, 
usually a self-adjoint operator on $H$, to any classical observable.
This motivates the definition of a \emph{strict quantization}
\cite{landsman:1998a, natsume.nest:1999a, natsume.nest.peter:2003a, rieffel:1993a},
which is some field of \anfa nice\anfel $^*$-algebras $\mathcal A_\hbar$ 
(over $\mathbb C$) depending \anfa nicely\anfel on a parameter $\hbar$ 
ranging over some subset of $\mathbb C$, with $\mathcal A_0$ being a completion
of the classical observable algebra and the deformation being in the direction
of the Poisson bracket.
However, strict quantizations are much harder to understand than formal
deformation quantizations.
There are many examples of strict quantizations in different contexts,
and therefore there are several ways to formalize the above definition, 
i.e.\ specifying the parameter set and what \anfa nice\anfel actually means.
No general existence results are known,
and a classification seems completely hopeless due to the increased complexity.

There are two prominent constructions of strict quantizations. The first is due 
to Rieffel \cite{rieffel:1993a} who, using oscillatory integrals, deforms the 
product on a Fr\'echet algebra 
endowed 
with an isometric action of $\mathbb R^d$. If the 
original algebra is a C$^*$-algebra, then Rieffel constructs a 
C$^*$-algebraic quantization. A generalization to negatively curved K\"ahlerian Lie groups
can be found in \cite{bieliavsky.gayral:2015a}. 
The second construction, 
due to Natsume, Nest, and Peter \cite{natsume.nest.peter:2003a}, essentially 
glues convergent versions of the Weyl product 
on charts to obtain a C$^*$-algebraic quantization. However, both methods work 
only for some symplectic manifolds and fail for 
example for the 2-sphere with its $\group{SO}(3)$-invariant symplectic 
structure \cite{rieffel:1998a}. They also make crucial use of the finite 
dimensionality of the classical mechanical system, so it remains unclear 
how to 
apply them to quantum 
field theories, 
despite such field theories fitting into the framework of formal deformation 
quantization.

Another approach to strict quantization was proposed by Beiser and Waldmann in 
\cite{beiser.roemer.waldmann:2007a, beiser.waldmann:2014a, waldmann:2014a}. 
They start with formal deformation quantizations, which are well-understood,
and try to find subalgebras on which the formal power series converge.
Such subalgebras are usually defined using additional geometric structures,
and can be completed with respect to a topology in which the product is continuous.
This approach was carried out explicitly for star products of exponential
type on possibly infinite-dimensional vector spaces \cite{schoetz.waldmann:2018a}, 
for the linear Poisson structure on the dual of a Lie algebra 
\cite{esposito.stapor.waldmann:2017a}, and 
for the hyperbolic disc $\mathbb D^n$ using an invariant star product 
obtained via phase space reduction 
\cite{kraus.roth.schoetz.waldmann:2019a}.
See also \cite{waldmann:2019a} for a survey.
In this paper, we extend this approach to semisimple coadjoint orbits
of connected semisimple Lie groups,
which gives a much larger class of geometrically interesting examples.

Coadjoint orbits play an important role in different areas of mathematics.
In the representation theory of unitary Lie groups 
they appear e.g.\ in the Kirillov orbit method \cite{kirillov:1962a},
while in symplectic geometry they are related to momentum maps. 
Basic examples of coadjoint orbits are hyperbolic discs
and complex projective spaces, including the 2-sphere. 
Any coadjoint orbit $\orb{}$ of a Lie group $G$ has a canonical $G$-invariant symplectic form,
and if $\orb{}$ is semisimple and $G$ is compact, connected, and semisimple then there is a 
unique compatible $G$-invariant complex structure that makes $\orb{}$ a K\"ahler manifold.

Constructions of star products on coadjoint orbits are due to many authors 
\cite{alekseev.lachowska:2005a, 
	cahen.gutt.rawnsley:1990a, cahen.gutt.rawnsley:1993a, 
	cahen.gutt.rawnsley:1994a, cahen.gutt.rawnsley:1995a, 
	fioresi.lledo:2001a,
	karabegov:1998c, karabegov:1999a}. 
In this paper, we focus on semisimple coadjoint orbits of connected semisimple Lie groups,
and the algebraic construction of Alekseev--Lachowska \cite{alekseev.lachowska:2005a}.
The canonical element $F_\lambda$ of the Shapovalov pairing 
between certain generalized Verma modules 
satisfies an associativity equation generalizing that of a Drinfel'd twist.
This twist induces a formal product for holomorphic functions on a complex orbit
and a formal star product for smooth functions on a real orbit, 
and those products are compatible by restriction.
It is very convenient that we can pass from one setting to the other:
We will mainly work in the complex setting, 
which is more convenient for obtaining continuity estimates,
and restrict to the real setting only in the very end.

Our first result uses methods developed by Ostapenko \cite{ostapenko:1992a} 
to obtain an explicit formula for the canonical element of the Shapovalov pairing
for a semisimple Lie algebra $\lie g$:

\begin{maintheorem} \label{maintheo:coadj:i}
	The Shapovalov pairing
	$\langle\argumentdot,\argumentdot\rangle^\sim_{\lambda} \colon			
	\univ(\plus{\tilde{\lie n}}) \times \univ(\minus{\tilde{\lie n}})
	\to \mathbb C$
	is non-degenerate if $\lambda \in \tilde \Lambda$, and in this case its canonical element
	$F_\lambda \in 
	\univ(\plus{\tilde{\lie n}}) \tensorhat \univ(\minus{\tilde{\lie n}})$
	is given by
	\begin{equation} \label{eq:formula}
	F_\lambda 
	= \sum_{w\in\tilde\wordset} p^w_{\lambda}(\alpha_w)^{-1} \tilde\pi_\lambda^+(X_w) 
	\tensor \tilde\pi_\lambda^-(Y_w)
	\punkt
	\end{equation}
\end{maintheorem}
The notation is explained in detail in \autoref{sec:starProduct}.
For now, it suffices to mention that the Shapovalov pairing
is a pairing between the universal enveloping algebras of 
two nilpotent Lie subalgebras $\plusminus{\tilde{\lie n}}$ of $\lie g$,
depending on a parameter $\lambda \in \lie g^*$.
The sum is over a set of words $\tilde W$ related to the root system of $\lie g$, 
the $p_\lambda^w(\alpha_w)$ are non-zero coefficients which are defined by an explicit formula, 
$X_w$ and $Y_w$ are elements of $\univ \lie g$ and $\tilde \pi_\lambda^\pm$ maps these
elements to $\univ(\plusminus{\tilde{\lie n}})$. 
The element $F_\hbar$, which induces the star product, is obtained by rescaling $\lambda$,
and doing so the coefficients $p_{\I\lambda/\hbar}^w(\alpha_w)^{-1}$
will depend rationally on $\hbar$,
with a countable set of poles $P$ that accumulate only at $0$.
It seems as if explicit formulas for deformation quantizations received
special attention by various authors, and \eqref{eq:formula} provides such
a formula that works in great generality.

As mentioned above, the formal expansion of $F_\hbar$ induces formal products
in complex and real settings.
Furthermore, we also obtain a family of actual (non-formal) products
for holomorphic polynomial functions in the complex setting 
and for polynomial functions in the real setting,
parametrized by $\mathbb C \setminus P$, 
since only finitely many elements of the infinite sum defining $F_\hbar$ 
are non-zero on polynomials. 
All these products are $G$-invariant,
and under some conditions on the Cartan subalgebra
used in the construction they are also Hermitian, meaning that
$\overline{ f *_\hbar g } = \overline g *_{\overline \hbar} \overline f$.
In the real setting and for a compact semisimple connected Lie group $G$,
the formal star product is of Wick type \cite{karabegov:1996a} 
with respect to the K\"ahler complex structure on the coadjoint orbit, 
meaning that it derives the first argument only in holomorphic directions
and the second argument only in antiholomorphic directions.

The next major step after constructing the star product
is to use the explicit formulas to prove its continuity in the complex setting
with respect to the topology of locally uniform convergence.
This topology is locally convex and we can extend the product to a continuous 
product on the completion of the holomorphic polynomials. 
Using methods from analytic geometry 
we identify this completion with the space of holomorphic functions.

\begin{maintheorem} \label{maintheo:coadj:ii}
	For any semisimple coadjoint orbit $\hat{\orbitscript O}$ of a 
	connected semisimple complex Lie group $G$,
	there is a family of products
	$\starhat \colon \Hol(\orbext{}) \times \Hol(\orbext{}) \to \Hol(\orbext{})$
	for $\hbar \in \mathbb C\setminus P$, 
	where every product $\starhat$ is $G$-invariant and continuous
	with respect to the topology of locally uniform convergence.
	The dependence of $\starhat$ on $\hbar$ is holomorphic.
\end{maintheorem}
This result is certainly interesting in its own right.
However, as mentioned above, we can also restrict it to real coadjoint orbits 
$\orb{}\subseteq \smash{\orbext{}}$. 
Denote by $\analytics(\orb{})$ the class of functions on $\orb{}$
that extend to holomorphic functions on $\smash{\orbext{}}$
(if a function extends, its extension is unique),
which contains the polynomials. We define the topology of 
extended locally uniform convergence on $\analytics(\orb{})$ by saying that a sequence 
of functions in $\analytics(\orb{})$ converges if the corresponding sequence of 
extensions converges locally uniformly,
so that $\analytics(\orb{})$ is homeomorphic to $\smash{\Hol(\orbext{})}$. 

\begin{maintheorem} \label{maintheo:coadj:iii}
	For any semisimple coadjoint orbit $\orb{}$ of a connected semisimple 
	real Lie group $G$, there is a family of products
	$\starhbar \colon \analytics(\orb{}) \times \analytics(\orb{}) 
												\to \analytics(\orb{})$
	for $\hbar \in \mathbb C\setminus P$,
	where every product $\starhbar$ is $G$-invariant and continuous 
	with respect to the topology of extended locally uniform convergence.
	The dependence of $\starhbar$ on $\hbar$ is holomorphic.
	The formal expansion of $\starhbar$ around $0$ is a formal star product 
	deforming the $G$-invariant symplectic form of $\orb{}$.
\end{maintheorem}
For the hyperbolic disc the quantum algebra $(\analytics(\mathbb D^n), *_\hbar)$ 
agrees with the algebra obtained in \cite{kraus.roth.schoetz.waldmann:2019a} 
while for the 2-sphere, $(\analytics(\mathbb S^2), *_\hbar)$ is the algebra considered 
in \cite{me}.

Since we constructed a quantization of the holomorphic functions on a complex 
coadjoint orbit and the restriction $\smash{\Hol(\orbext{})} \to \analytics(\orb{})$
is an isomorphism,
the quantizations of different real orbits with the same complexification are related:

\begin{maintheorem} \label{maintheo:coadj:iv}
	If $\orb{}$ and $\orb{}'$ are coadjoint orbits of real semisimple connected
	Lie groups with the same complexification and through one common semisimple element,
	then the algebras $(\analytics(\orb{}), *_\hbar)$ and 
	$(\analytics(\orb{}'), *'_\hbar)$ are isomorphic.
\end{maintheorem}
This isomorphism generalizes the classical Wick rotation,
which can be interpreted as an isomorphism 
between the polynomial algebras $\Pol(\mathbb{CP}^n)$ and $\Pol(\mathbb D^n)$.
However, this isomorphism does not necessarily respect the star involutions
with which the algebras $\analytics(\orb{})$ are equipped.
In other words, the algebras $\analytics(\orb{})$ and $\analytics(\orb{}')$
are isomorphic as algebras, but not necessarily as $^*$-algebras.

In order to apply our quantization to physics, we should represent the 
Fr\'echet algebras $(\analytics(\orb{}), *_\hbar)$ on a Hilbert space. Given a 
positive linear functional we can use the GNS representation to do so. 
For a formal star product of Wick type all point evaluation functionals are 
formally positive. 
However, formal positivity means only that the first non-vanishing 
order is positive and therefore, as in this case, might not survive the passage 
to strict 
products (where the contribution of higher orders can dominate the contribution 
of the first order). For certain coadjoint 
orbits we will prove that point evaluations stay positive.

One aspect that we do not discuss in this work is the relation to geometric
or Berezin--Toeplitz quantization 
\cite{cahen.gutt.rawnsley:1990a, 
	cahen.gutt.rawnsley:1993a, 
	cahen.gutt.rawnsley:1994a, 
	cahen.gutt.rawnsley:1995a}.
These theories construct a quantization by studying holomorphic sections 
of a quantizing line bundle over the manifold $M$.
This line bundle needs to satisfy some integrality condition,
which for compact $M$ means that only countably many values of $\hbar$,
accumulating at $0$, are allowed.
The algebra $\Cinfty(M)$ is, in the limit $\hbar \to 0$, 
approximated by finite dimensional matrix algebras.
The construction of Alekseev--Lachowska coincides with
another more geometric construction of star products 
on semisimple coadjoint orbits by Karabegov \cite{me, karabegov:1999a},
if $\hbar$ is not a pole.
However, Karabegov's construction still makes sense at the poles,
where it coincides with (a variant of) the Berezin--Toeplitz quantization \cite{karabegov:1999a}.
In this sense our infinite dimensional Fr\'echet 
algebras $(\analytics(\orb{}), *_\hbar)$ interpolate between the finite 
dimensional Berezin--Toeplitz algebras.
It could be very interesting to study this in greater detail.

\subsubsection*{Contents}
In \autoref{sec:preliminaries} we recall some well-known facts about coadjoint 
orbits. This includes the realizability of coadjoint orbits as orbits of matrix Lie groups, 
and a characterization of invariant multidifferential operators on homogeneous spaces.
In \autoref{sec:starProduct} we introduce the Shapovalov pairing 
of (generalized) Verma modules and derive an explicit formula for its canonical 
element. From this, we obtain a product for holomorphic
polynomials on complex coadjoint orbits. In \autoref{sec:continuity} we show 
that this product is continuous with respect to the topology of locally uniform 
convergence, so that we can extend it to the completion, which consists of all 
holomorphic functions on the orbit. Finally, we restrict our results to real
coadjoint orbits in \autoref{sec:properties}. We will determine additional
properties of the star products obtained in this way (e.g.\ being of Wick type or
of standard ordered type), study positive linear functionals,
and investigate isomorphisms of the algebras obtained for 
different real forms of the same complex coadjoint orbit.
In \appendixref{appendices} we 
give some remaining proofs and more 
details on complex structures.

\subsubsection*{Notation}
In the whole paper $G$ is either a real or complex Lie group,
$\lie g$ denotes the Lie algebra of $G$, and $\univ \lie g$ denotes the universal enveloping
algebra of $\lie g$.
In \autoref{sec:starProduct} and \autoref{sec:continuity}, $G$ is always complex.
In \autoref{sec:properties}, $G$ refers to a real Lie group 
and $\hat G$ to a complexification of $G$.
$K$ denotes a compact real Lie group. 
Coadjoint orbits through $\lambda \in \lie g^*$ are denoted by $\orb\lambda$.

We write $\Cinfty(M)$ for the smooth complex-valued functions on a manifold $M$. 
If $M$ is a real manifold, $\Tangent M$ denotes its (real) tangent bundle 
(so sections of $\Tangent M$ are derivations of 
the algebra of real-valued smooth functions on $M$).
The complexification of $\Tangent M$ is denoted by $\Tangent^{\mathbb C} M$
(so sections of $\Tangent^{\mathbb C} M$ are derivations of $\Cinfty(M)$).
If $M$ is a complex manifold, then the holomorphic tangent
bundle is denoted by $\Tangent^{(1,0)} M$. 
\section{Preliminaries}
\label{sec:preliminaries}
In this section we summarize some results that are needed in the rest of this article:
We review the definition of coadjoint orbits and their realizability 
as orbits of matrix Lie groups in \autoref{subsec:generalities}.
In \autoref{subsec:diffopsOnHomogeneousSpaces} we introduce
invariant multidifferential operators on homogeneous spaces.

\subsection{Coadjoint orbits}
\label{subsec:generalities}

Let $G$ be a real or complex Lie group with Lie algebra $\lie g$.
We denote the \emph{adjoint action} of $G$ on $\lie g$
by $\Ad \colon G \to \group{End}(\lie g)$.
For any $g \in G$, $\Ad_g \coloneqq \Ad(g)$ is the tangent map of the conjugation 
$G \ni x \mapsto g x g^{-1} \in G$ by $g$.
Its differential $\ad \colon \lie g \to \lie{end}(\lie g)$ is given by the Lie bracket,
$\ad_X(Y) = [X,Y]$.
The \emph{coadjoint action} $\Ad^* \colon G \to \group{End}(\lie g^*)$
of $G$ on the dual $\lie g^*$ of $\lie g$
is defined by $\Ad^*_g \xi = \xi \circ \Ad_{g^{-1}}$ for $\xi \in \lie g^*$.

The \emph{coadjoint orbit} $\orb\lambda$ of $G$ 
through an element $\lambda \in \lie g^*$ is defined as 
\begin{equation}
\orb\lambda = \lbrace \xi \in \lie g^* \mid  \xi = \Ad^*_g \lambda 
\text{ for some $g \in G$}\rbrace \punkt
\end{equation}
It is well-known that $\orb\lambda \cong G/G_\xi$ where
$\xi \in \orb\lambda$ is any point on the coadjoint orbit and 
$G_\xi = \lbrace g \in G \mid \Ad_g^* \xi = \xi \rbrace$
is the \emph{stabilizer subgroup} of $\xi$.
If $G$ is a real (complex) Lie group, there is a unique 
smooth (complex) manifold structure on $G / G_\xi$
that makes the projection $\pi \colon G \to G / G_\xi$
a smooth (holomorphic) submersion, and we use it to define the structure
of a smooth (complex) manifold on $\orb\lambda$.
It does not depend on the choice of $\xi \in \orb\lambda$.

Fix a basis $e_1, \dots, e_n$ of $\lie g$ and let $\smash{C_{ij}^k}$ be the structure
constants with respect to this basis, i.e.\ $[e_i, e_j] = \sum_{k=1}^n C_{ij}^k e_k$.
Then 
$\lbrace f,g \rbrace (\xi) 
=
\sum_{i,j,k=1}^n C_{ij}^k \xi(e_k) 
	\frac{\partial f}{\partial e_i} 
	\frac{\partial g}{\partial e_j}$
defines a linear Poisson structure on $\lie g^*$,
where $f, g \in \Cinfty(\lie g^*)$ and the $e_i$ are viewed as global linear coordinates on $\lie g^*$.
The following proposition is well-known, see e.g.\ \cite[Example 1.1.3]{chari.pressley:1994a}.

\begin{proposition} \label{theo:coadjointOrbitsAreSymplecticLeaves}
	If the Lie group $G$ is connected, 
	then the coadjoint orbits of $G$ are precisely
	the symplectic leaves of this linear Poisson structure.
	In particular, all connected Lie groups with the same Lie algebra
	have the same coadjoint orbits.
\end{proposition}

\begin{corollary}
	If the Lie group $G$ is semisimple and connected,
	then $G$ and its image under $\Ad \colon G \to \group{End}(\lie g)$
	have the same coadjoint orbits.
\end{corollary}

\begin{proof}
	Since $\lie g$ is semisimple, it has trivial center and therefore
	$\ad \colon \lie g \to \lie{end}(\lie g)$ is injective.
	Consequently, $G$ and its image in $\group{End}(\lie g)$ have the
	same Lie algebra. Since both are connected, the result follows by
	applying the previous proposition.
\end{proof}
It is easy to show that not only $G$ and its image under $\Ad$ have the same coadjoint orbits,
but also $\Ad \colon G \to \group{End}(\lie g)$ intertwines the actions of $G$
and its image on the coadjoint orbits.
Since the image of $G$ under $\Ad$ is a matrix Lie group,
we can therefore, when studying coadjoint orbits of connected semisimple Lie groups,
assume without loss of generality that such a Lie group is a matrix Lie group.
Using the argument provided in \cite[Theorem 9]{goto:1950a}
we can even assume that $G$ is a closed matrix Lie group.

For $X \in \lie g$, denote the \emph{fundamental vector field} of $X$
for the coadjoint action by 
$\smash{X_{\orb\lambda} \at \xi} \coloneqq \frac{\D}{\D t}\at{t=0}\Ad^*_{\exp(-t X)} \xi$,
where $\xi \in \orb\lambda$.
Note that the map $\lie g / \lie g_\xi \to \Tangent_\xi \orb\lambda$,
$X \mapsto X_{\orb\lambda} \at \xi$ is an isomorphism, where $\lie g_\xi$
denotes the Lie algebra of $G_\xi$.
Consequently,\begin{equation} \label{eq:KKSform}
\omega_{\mathrm{KKS}} (X_{\orb\lambda}, Y_{\orb\lambda})\at\xi
= \xi([X, Y]) 
\end{equation}
determines a well-defined $2$-form on $\orb\lambda$,
which is called the Kirillov-Kostant-Souriau form.
One can show that $\omega_{\mathrm{KKS}}$ is symplectic
and $G$-invariant.
By symplectic we mean that $\omega_{\mathrm{KKS}}$ is closed and that
$\omega_{\mathrm{KKS}} \at \xi \colon 
\Tangent_\xi \orb\lambda \times \Tangent_\xi \orb\lambda \to \mathbb k$ 
is $\mathbb k$-bilinear, antisymmetric, and non-degenerate for all $\xi \in \orb\lambda$,
where $\mathbb k$ is either $\mathbb R$ or $\mathbb C$,
depending on whether $G$ is real or complex.

For a semisimple Lie algebra $\lie g$, the Killing form 
$B \colon \lie g \times \lie g \to \mathbb k$ is non-degenerate,
giving an isomorphism
${}^\flat \colon \lie g \to \lie g^*$, 
$X \mapsto X^\flat \coloneqq B(X, \argumentdot)$.
We denote its inverse by ${}^\sharp \colon \lie g^* \to \lie g$.
In the complex case we say that $\lambda \in \lie g^*$ is semisimple
if $\ad_{\lambda^\sharp} \in \lie{end}(\lie g)$ is diagonalisable
and in the real case $\lambda \in \lie g^*$ is semisimple
if the complex linear extension of $\lambda$
to the complexification of $\lie g$ is semisimple.
A coadjoint orbit $\orb\lambda$ is semisimple if $\lambda$ is semisimple.

\begin{proposition} \label{theo:stabilizerConnected}
	Let $G$ be a complex connected semisimple Lie group
	and $\lambda \in \lie g^*$ be semisimple.
Then $G_\lambda$ is connected.
\end{proposition}

\begin{proof}
	The Lie algebra spanned by $\lambda^\sharp$ integrates
	to a connected commutative Lie subgroup $T'$ of $G$,
	and since $\lambda^\sharp$ is semisimple, all elements of $T'$
	are diagonalisable in the adjoint representation.
There is a smallest closed complex Lie group $T$ containing $T'$,
	that can be obtained as follows:
	Take the closure of $T'$ (which is a real Lie group),
	take the Lie algebra of this closure (which is a real Lie subalgebra of $\lie g$),
	take the complex Lie algebra spanned by it,
	integrate this Lie algebra to a connected Lie subgroup of $G$,
	and possibly repeat these steps.
$T$ is still connected and commutative, and all its elements are
	diagonalisable in the adjoint representation, so $T$ is a complex torus in $G$.
	Its centralizer is exactly $G_\lambda$, and centralizers of tori are connected.
\end{proof}
Note that the statement is also true for a real compact connected semisimple 
Lie group $K$, but might fail if the 
compactness assumption is dropped.

We denote the smooth functions on $G$ that are invariant under the action of 
$G_\lambda$ from the right by $\Cinfty(G)^{G_\lambda}$. That is, $f 
\in\Cinfty(G)^{G_\lambda}$ if and only if $f \in \Cinfty(G)$ and $f(g g') = 
f(g)$ for all $g \in G$ 
and $g'\in G_\lambda$. There is an algebra isomorphism 
\begin{equation}
\pi^* \colon \Cinfty(G / G_\lambda) \to \Cinfty(G)^{G_\lambda} \komma\quad 
f \mapsto \pi^* f \coloneqq f \circ \pi
\end{equation}
and for a complex Lie group, this isomorphism restricts
to an isomorphism on holomorphic functions.
We denote the inverse by
$\invpi \colon \Cinfty(G)^{G_\lambda} \to \Cinfty(G / G_\lambda)$.

\begin{remark}\label{theo:coadjointOrbitsAreClosed}
	This article is written mainly from a differential geometric perspective.
	Note however, that any complex connected semisimple Lie group $G$
	has a unique structure of an algebraic group,
	see Theorem 6.3 and the preceding corollary 
	in Chapter 1 of \cite{onishchik.vinberg:1994a}.
	Any holomorphic representation of $G$ is polynomial.
	Consequently, if $G$ is realized as a subgroup
	of $\GL[N]$ it is automatically closed.
The coadjoint action $G \times \lie g^* \to \lie g^*$
	is a morphism of algebraic varieties, 
and coadjoint orbits of $G$ are smooth subvarieties of $\lie g^*$.
A coadjoint orbit of $G$ is closed in the Zariski topology
	if and only if it is semisimple, see \cite[Theorem 5.4]{crooks:2018a}.
In particular, semisimple coadjoint orbits of complex
	connected semisimple Lie groups are affine algebraic varieties.
	
	Note however, that this is not necessarily true for real connected 
	semisimple Lie groups (not even if they are linear).
It is still true that real connected semisimple linear Lie groups
	and their coadjoint orbits are connected components
	(with respect to the usual topology) of affine algebraic varieties. 
\end{remark} 
\subsection[\texorpdfstring{Invariant holomorphic $k$-differential operators on homogeneous spaces}{Invariant holomorphic k-differential operators on homogeneous spaces}]{\texorpdfstring{Invariant holomorphic \boldmath$k$-differential operators on homogeneous spaces}{Invariant holomorphic k-differential operators on homogeneous spaces}}
\label{subsec:diffopsOnHomogeneousSpaces}

In the whole subsection $G$ is a complex Lie group, 
$H$ is a closed complex Lie subgroup of $G$,
and $k \geq 1$ is an integer.
We present some results on holomorphic
$G$-invariant $k$-differential operators 
on the homogeneous space $G / H$,
in particular we construct a bijection between the set
$((\univ \lie g/ \univ \lie g \cdot \lie h)^{\tensor k})^H$
and the set of such operators. 
The results seem to be well-known,
but proofs are hard to find in the literature. 

A $k$-differential operator $D$ 
(see \appendixref{appendix:diffops} for a short review of the definition)
on a manifold $M$ endowed with an action of a Lie group $G$
is said to be invariant under $G$ if  
$\smash{\phi_g^*} (D \smash{\vec f}) = D ((\phi_g^*)^{\times k} \smash{\vec f})$ 
for all $\vec f \in \Cinfty(M)^k$ and all $g \in G$.
Here $\phi_g \colon M \to M$ is the diffeomorphism of $M$ given by the action
of a fixed element $g \in G$,
and the upper star denotes the pullback.
We write $\smash{\nDiffop_\holo^{G}(M)}$ for the space of holomorphic
$G$-invariant $k$-differential operators on a complex manifold $M$.
A $k$-differential operator on $G$ is said to be left-invariant if it is 
invariant with respect to the left action $\leftact \colon G \times G \to G$, $(g,g') 
\mapsto g g' \eqqcolon \leftact_g(g')$. 

Let $M$ be a complex manifold with complex structure $I \colon \Tangent M \to \Tangent M$.
For a vector field $V \in \Secinfty(\Tangent M)$ its holomorphic part is
$V^{(1,0)} = \frac 1 2 (V - \I I V) \in \Secinfty(\Tangent^{(1,0)} M)$.
Let $\lie g$ be the Lie algebra of $G$. For any $X \in \lie g$ define the
\emph{left-invariant vector field}
\begin{equation}\label{eq:defHoloLeftInvVF}
\leftinv X \at{g} 
\coloneqq 
\frac{\D}{\D t}\at{t=0} g \exp(t X) 
\in \Secinfty(\Tangent G) \punkt
\end{equation}
Its holomorphic part
$\leftinvholo X 
= \frac 1 2 (\leftinv X - \I \leftinv{(\I X)}) 
\in \Secinfty(\Tangent^{(1,0)} G)$
induces a holomorphic left-invariant 
$1$-differential operator $f \mapsto \leftinvholo X f$ on $G$.
The map $\leftinvholo{(\argumentdot)} \colon \lie g \to \Secinfty(\Tangent^{(1,0)} G)$ 
is a Lie algebra homomorphism, inducing an algebra homomorphism 
$\leftinvholo{(\argumentdot)} \colon \univ \lie g \to \Diffop_\holo^{G}(G)$.
 
In the following we extend various maps to $k$-fold  
products and still denote them by the same symbol,
\begin{subequations}
\begin{align}
\label{eq:Ad:extended}
\Ad_g \colon (\univ \lie g)^{\tensor k} &\to (\univ\lie g)^{\tensor k} \komma &
u_1 \tensor \dots \tensor u_k &\mapsto \Ad_g u_1 \tensor \dots \tensor \Ad_g 
u_k \komma
\\ 
\pi^* \colon \Cinfty(G/H)^{k} &\to (\Cinfty(G)^{H})^{k} \komma &
(f_1, \dots, f_k) &\mapsto (\pi^* f_1, \dots, \pi^* f_k) \komma
\\ \notag
\leftinvholo{(\argumentdot)} \colon (\univ \lie g)^{\tensor k} &\to \nDiffop_\holo^{G}(G) 
\komma & & \\
&&&\!\!\!\!\!\!\!\!\!\!\!\!\!\!\!\!\!\!\!\!\!\!\!\!\!\!\!\!\!\!\!\!\!\!\!\!\!\!\!\!
\!\!\!\!\!\!\!\!\!\!\!\!\!\!\!\!\!\!\!\!\!\!\!\!\!\!\!\!\!\!\!\!\!\!\!\!\!\!\!\!
\!\!\!\!\!\!\!\!\!\!\!\!\!\!\!\!\!\!\!\!\!\!\!\!\!\!\!
u_1 \tensor \dots \tensor u_k \mapsto
\left((f_1, \dots, f_k) \mapsto
\leftinvholo{u_1} f_1 \cdot \ldots \cdot \leftinvholo{u_k} f_k \right) \punkt
\label{eq:extendingToTensors}
\end{align}
\end{subequations}

\begin{proposition}\label{theo:leftInvDiffOps:complex}
	$\!$The map
	$\leftinvholo{(\argumentdot)} \colon
	(\univ \lie g)^{\tensor k} \to \nDiffop_{\holo}^{G}(G)$
	is an isomorphism.
\end{proposition}
\begin{proof}
	See \appendixref{appendix:diffops}.
\end{proof}
Next, we want to describe holomorphic $G$-invariant
$k$-differential operators on the homogeneous space $G / H$.
Let $H$ be a closed Lie subgroup of $G$ with Lie algebra $\lie h$,
and let $\univ \lie g \cdot \lie h \subseteq \univ \lie g$ be the left ideal
generated by $\lie h$.
Note that $(\univ \lie g / \univ\lie g \cdot \lie h)^{\tensor k}$ is isomorphic to 
$(\univ\lie g)^{\tensor k}/I$ where $I = I_1 + \dots + I_k$ and
$I_i = (\univ \lie g)^{\tensor(i-1)} \tensor \univ \lie g 
\cdot \lie h \tensor (\univ \lie g)^{\tensor(k-i)}$ 
is a left ideal in $(\univ \lie g)^{\tensor k}$. 
Introduce the set
\begin{align} \notag
U_{\mathrm{inv}} &= \lbrace \vec u \in (\univ \lie g)^{\tensor k} \mid 
[\vec u] \in (\univ \lie g / \univ \lie g \cdot \lie h)^{\tensor k}
\text{ is $H$-invariant}\rbrace \\
&= \lbrace \vec u \in (\univ \lie g)^{\tensor k} \mid 
\Ad_h \vec u - \vec u \in I \text{ for all $h \in H$}\rbrace \punkt
\end{align}
Here the action of $H$ on $(\univ \lie g)^{\tensor k}$ is the diagonal action
defined in \eqref{eq:Ad:extended}.

\begin{lemma} \label{theo:diffops:Hinvariance}
	Let $\vec u \in U_{\mathrm{inv}}$, $\vec v \in I$,
	and $\vec f \in (\Cinfty(G)^{H})^{k}$.
	Then 
	\begin{equation}
	\leftinvholo{\vec v} \vec f = 0 
	\quad\text{ and }\quad
	\leftinvholo{\vec u}\vec f \in \Cinfty(G)^{H} \punkt
	\end{equation}
\end{lemma}

\begin{proof}
	Let $Y \in \lie h$ and $f\in \Cinfty(G)^H$. Then we compute
	\begin{equation*}
	(\leftinv{Y} f)(g) 
	=\frac{\D}{\D t}\At{t=0} f(g \exp(t Y)) 
	=\frac{\D}{\D t}\At{t=0} f(g) 
	= 0 \punkt
	\end{equation*}
	By using that
	$\leftinvholo Y = \frac 1 2 (\leftinv Y - \I \leftinv{(\I Y)})$
	this implies that $\leftinvholo{Y} f = 0$,
	and therefore also $\leftinvholo{\vec v} \vec f = 0$ for all $\vec v \in I$ and
	$\vec f \in (\Cinfty(G)^H)^k$.
	If $X \in \lie g$, then
	\begin{equation*}
	(\leftinv{X} f)(g h) 
	=\frac{\D}{\D t}\At{t=0} f(g h \exp(t X)) 
	=\frac{\D}{\D t}\At{t=0} f(g \exp(t \Ad_h X))
	= (\leftinv{(\Ad_h X)} f) (g)
	\end{equation*}
	for all $f \in \Cinfty(G)^H$, $g \in G$, and $h \in H$. Consequently, we obtain 
	$(\leftinvholo{X} f)(g h) = (\leftinvholo{(\Ad_h X)} f) (g)$,
	and extending to the universal enveloping algebra and to tensor products yields
	$(\leftinvholo{\vec u} \vec f)(g h) = (\leftinvholo{(\Ad_h \vec u)} \vec f) (g)$
	for all $\vec u \in (\univ\lie g)^{\tensor k}$ and $\vec f \in (\Cinfty(G)^H)^k$.
	If $\vec u \in U_{\mathrm{inv}}$, then together with the first part we obtain
	\begin{multline*}
	(\leftinvholo{\vec u}\vec f)(gh) 
	= (\leftinvholo{(\Ad_h \vec u)} \vec f)(g)
	= \\
	= (\leftinvholo{\vec u}\vec f)(g) + (\leftinvholo{(\Ad_h\vec u - \vec u)}\vec f)(g)
	= (\leftinvholo{\vec u}\vec f)(g) \punkt
	\end{multline*} 	
\end{proof}
Because of this lemma we can define
\begin{equation*}
\tilde \Psi \colon U_{\mathrm{inv}} \to 
\mathrm{Map}(\Cinfty(G/H)^k, \Cinfty(G/H)) \komma 
\quad
\tilde\Psi(\vec u) \vec f = \pi_* (\leftinvholo{\vec u} (\pi^*\vec f)) \punkt
\end{equation*}
Since $\pi^*$ and $\pi_*$ are algebra homomorphisms, it follows
that $\tilde \Psi(\vec u)$ and $\leftinvholo{\vec u}$ satisfy essentially the same commutation
relations with the operator that multiplies a component by a smooth function.
Consequently $\tilde \Psi(\vec u)$ is
$k$-differential and of the same order than $\leftinvholo{\vec u}$
(see the definition of $k$-differential operators given in
\autoref{def:nDifferentialOperators}).
Moreover, $\tilde \Psi(\vec u)$ is $G$-invariant,
because $\pi^*$ and $\pi_*$ are $G$-equivariant and 
$\leftinvholo{\vec u}$ is $G$-invariant.
Since $\pi \colon G \to G / H$ is a holomorphic map, 
it follows that $\tilde \Psi(\vec u)$ is holomorphic,
and $\tilde \Psi$ really maps into $\nDiffop^{G}_\holo(G/H)$.
The map $\tilde\Psi$ descends to a map
\begin{equation} \label{eq:diffops:bijection:real}
\Psi \colon 
((\univ \lie g / \univ \lie g \cdot \lie h)^{\tensor k})^{H} 
\to 
\nDiffop_\holo^{G}(G/H)
\end{equation}
because $\tilde \Psi(I) = 0$ according to the previous lemma.

\begin{proposition} \label{theo:diffops:bijection:real}
	The map $\Psi$ defined in \eqref{eq:diffops:bijection:real} is an 
	isomorphism.
\end{proposition}
\begin{proof}
	The proof is given in \appendixref{appendix:diffops}.
\end{proof}
The last result of this subsection gives a description of the $k$-differential operator
$\Psi([\vec u])$ on the coadjoint orbit without using extensions to $G$.
Let $S$ be the antipode of $\univ \lie g$ and extend the Lie algebra homomorphism
$\lie g 
\ni 
X \mapsto X_{\orb\lambda} 
\in 
\Secinfty(\Tangent \orb\lambda)$
defined just before \eqref{eq:KKSform} to an algebra homomorphism
$\univ\lie g \to \Diffop(\orb\lambda)$.

\begin{proposition}\label{theo:calculationOfLeftInvDiffOpsOnOrbit} Let 
	$\orb\lambda \cong G / G_\lambda$ be a coadjoint orbit.
	If $\vec u = u_1 \tensor \dots \tensor u_k \in U_\mathrm{inv}$ 
	and $\vec f = (f_1, \dots, f_k) \in \Cinfty(\orb\lambda)^{k}$, then
	\begin{equation}
	\Psi([\vec u]) \vec f(\Ad_g^* \lambda) 
	= 
	(S(\Ad_g u_1))^{(1,0)}_{\orb\lambda} f_1 
	(\Ad_g^* \lambda)
	\cdot\ldots\cdot 
	(S(\Ad_g u_k))^{(1,0)}_{\orb\lambda} f_k (\Ad_g^* \lambda) \punkt
	\end{equation}
\end{proposition}

\begin{proof}
	Defining the Lie algebra homomorphism
	$\rightinv{(\argumentdot)} \colon \lie g \to \Secinfty(\Tangent G)$,
	$X \mapsto \rightinv X$ 
	with 
	$\rightinv X \at g \coloneqq \frac{\D}{\D t}\at{t=0} \exp(-t X) g$
	and extending to $\univ \lie g$ as before, one checks that
	\begin{align*}
	\leftinv u f(g) &= \leftinv{X_1} \dots \leftinv{X_j} f(g)
	\\&= \frac{\D}{\D t_1}\At{t_1=0} \dots \frac{\D}{\D t_j}\At{t_j=0} f(g 
	\exp(t_1 X_1) \dots 
	\exp(t_j X_j))
	\\&= \frac{\D}{\D t_1}\At{t_1=0} \dots \frac{\D}{\D t_j}\At{t_j=0} 
	f(\exp(t_1\Ad_g X_1) \dots 
	\exp(t_j\Ad_g X_j) g)
	\\&= \rightinv{(- \Ad_g X_j)} \dots \rightinv{(-\Ad_g X_1)} f(g)
	\\&= \rightinv{(S(\Ad_g u))} f(g)
	\end{align*}
	for $u = X_1 \dots X_j \in \univ \lie g$
	and similarly $\leftinvholo u f (g) = \rightinvholo{(S(\Ad_g u))} f(g)$.
	Furthermore, we have
	\begin{multline*}
	\rightinv X (\pi^* f) (g) 
	= \frac{\D}{\D t}\At{t=0} \pi^* f (\exp(-t X) g) = \\
	= \frac{\D}{\D t}\At{t=0} f(\Ad^*_{\exp(-tX)} \Ad^*_g \lambda)
	= X_{\orb\lambda} f(\Ad_g^* \lambda)
	= \pi^* (X_{\orb\lambda} f)(g)
	\end{multline*}
	for all $X \in \lie g$,
	implying that $\rightinvholo X \circ \pi^* = \pi^* \circ X^{(1,0)}_{\orb\lambda}$,
	and therefore that $\rightinvholo u \circ \pi^* = \pi^* \circ u^{(1,0)}_{\orb\lambda}$
	for all $u \in \univ \lie g$.
	Finally,
	\begin{align*}
	\Psi([\vec u]) \vec f(\Ad_g^* \lambda) 
	&= (\leftinvholo{\vec u} \pi^* \vec f) (g) \\
	&= \leftinvholo{u_1} (\pi^* f_1)(g) \cdot\ldots\cdot 
	\leftinvholo{u_k} (\pi^* f_k)(g) \\
	&= \rightinvholo{(S(\Ad_g u_1))} (\pi^* f_1)(g) \cdot\ldots\cdot
	\rightinvholo{(S(\Ad_g u_k))} (\pi^* f_k)(g) \\
	&= (S(\Ad_g u_1))^{(1,0)}_{\orb\lambda} f_1 (\Ad^*_g \lambda) \cdot\ldots\cdot 
	(S(\Ad_g u_k))^{(1,0)}_{\orb\lambda} f_k (\Ad^*_g \lambda) \punkt
	\end{align*}
\end{proof} 
\section{Quantizing complex coadjoint orbits}
\label{sec:starProduct}
In this section we construct a formal associative product
for holomorphic functions on a semisimple coadjoint orbit
of a complex connected semisimple Lie group,
and a strict associative product for polynomials.
These products are induced by a twist,
which is constructed using the Shapovalov pairing
between generalized Verma modules.
For the convenience of the reader we first consider the special case
of regular semisimple orbits in \autoref{subsec:twist:restricted},
where we introduce the Shapovalov pairing between Verma modules and
compute its canonical element.
In \autoref{subsec:twist:general} we generalize these results
to non-regular semisimple orbits.
In \autoref{subsec:starProductFromTwist} we describe
the induced formal and strict products in detail.
We consider an example in \autoref{subsec:examples}.

Later, in \autoref{sec:properties},
we will use the results of this section to obtain star products
on semisimple coadjoint orbits of real connected semisimple Lie groups.
From the example considered in this section, we will then obtain strict
quantizations of the hyperbolic disc and the complex projective space.
 
\subsection{Verma modules and the Shapovalov pairing}
\label{subsec:twist:restricted}

In this subsection we introduce the Shapovalov pairing
between Verma modules.
In case this pairing is non-degenerate,
we derive an explicit formula for its canonical element,
following \cite{ostapenko:1992a}.
A similar formula in the more general setting of quantum groups
was obtained recently in \cite{mudrov:2014a}.
The results allow us to quantize regular orbits.

Let $\lie g$ be a complex semisimple Lie algebra
with Cartan subalgebra $\lie h$.
Recall that a root is a non-zero element $\alpha \in \lie h^*$ 
such that 
$\lie g^\alpha 
\coloneqq 
\lbrace 
	X \in \lie g 
	\mid 
	\ad_H X = \alpha(H) X \text{ for all $H \in \lie h$}
\rbrace$
contains a non-zero element.
Denote the set of roots by $\Delta$
and choose an ordering (i.e.\ a subset $\Delta^+$ of positive roots
such that, setting $\Delta^- \coloneqq - \Delta^+$, we have
$\Delta^+ \cup \Delta^- = \Delta$,
$\Delta^+ \cap \Delta^- = \emptyset$, and such that
if the sum of positive roots is a root, then it is positive).
Denote the simple roots (i.e.\ elements of $\Delta^+$ that cannot be 
written as a sum of two elements of $\Delta^+$) by $\Sigma$.
Let $\plus{\lie n}$ and $\minus{\lie n}$ be
the nilpotent Lie subalgebras of $\lie g$ spanned by 
the positive and negative root spaces, respectively, and define
$ \plus{\lie b} \coloneqq \lie h \oplus  \plus{\lie n}$ and  
$\minus{\lie b} \coloneqq \lie h \oplus \minus{\lie n}$ 
(the direct sum is as vector spaces, the Lie algebra structure 
on $\plusminus{\lie b} \subseteq \lie g$ is obtained by restriction from $\lie g$).

Note that $0$ is not a root. 
However, it is convenient to introduce the notation $\lie g^0 \coloneqq \lie h$.
Then $\lie g$ is $(\Delta \cup \lbrace 0 \rbrace)$-graded, in the sense
that $\lie g = \bigoplus_{\alpha \in \Delta \cup \lbrace 0 \rbrace} \lie g^\alpha$
and $[\lie g^\alpha, \lie g^\beta] \subseteq \lie g^{\alpha+\beta}$ for any
$\alpha, \beta \in \Delta \cup \lbrace 0 \rbrace$.
Consequently the tensor algebra $\Tensor \lie g$ is $\mathbb Z \Delta$-graded,
where the so-called root lattice $\mathbb Z \Delta$ is the set of linear
combinations of roots.
The two-sided ideal generated by elements of the form
$X \tensor Y - Y \tensor X -[X,Y]$ with $X, Y \in \lie g$
is homogeneous and therefore the universal enveloping algebra
$\univ\lie g = \Tensor \lie g / \langle X \tensor Y - Y \tensor X -[X,Y] \rangle$
is also $\mathbb Z \Delta$-graded.
Denote the degree of a homogeneous element $w \in \univ\lie g$
by $d(w) \in \mathbb Z\Delta$.

Given a linear functional $\lambda \in \lie h^*$, 
the formula $H \acts z = \lambda(H) z$ makes $\mathbb C$
a left $\lie h$-module, and since $\lie h$ is commutative
also a right $\lie h$-module.
We can extend this to a left or right $\plusminus{\lie b}$-module
by noting that $\plusminus{\lie b} = \lie h \oplus \plusminus{\lie n}$
and letting $\plusminus{\lie n}$ act trivially.
Denote the corresponding left $\univ(\plusminus{\lie b})$-module
by $\mathbb C^\pm_\lambda$ and the right $\univ(\minus{\lie b})$-module
by $\mathbb C^*_\lambda$.
Define the \emph{Verma modules}
\begin{equation}
M_\lambda \coloneqq \univ\lie g \tensor_{\univ(\plus{\lie b})} \mathbb C^+_\lambda \komma 
\quad
M^-_\lambda \coloneqq \univ\lie g \tensor_{\univ(\minus{\lie b})} \mathbb C^-_{-\lambda}  
\quad\text{and}\quad
M^*_\lambda \coloneqq \mathbb C^*_\lambda \tensor_{\univ(\minus{\lie b})} \univ\lie g \punkt
\end{equation}
Note that $M_\lambda$ and $M_\lambda^-$ are left $\univ \lie g$-modules, whereas $M_\lambda^*$ is a 
right $\univ \lie g$-module.
$M_\lambda$ is the most general left $\univ\lie g$-module of highest weight $\lambda$,
meaning that any other left $\univ\lie g$-module of highest weight $\lambda$ 
can be obtained as a quotient of $M_\lambda$. $M_\lambda^-$ is the most general left $\univ\lie 
g$-module of lowest weight $-\lambda$.

There are canonical isomorphisms
$M_\lambda^* \tensor_{\univ\lie g} M_\lambda 
\cong
	\mathbb C^*_\lambda \tensor_{\univ(\minus {\lie b})} \univ \lie g 
	\tensor_{\univ(\plus{\lie b})} \mathbb C_\lambda 
\cong 
	\mathbb C^*_\lambda \tensor_{\univ \lie h} \mathbb C_\lambda 
\cong
	\mathbb C$
since the left and right $\lie h$-module structures on $\mathbb C$ coincide.

\begin{definition}
	The pairing 
	$\langle \argumentdot , \argumentdot \rangle'_\lambda 
	\colon
	M^*_\lambda \times M_\lambda \to \mathbb C$ 
	defined by 
	$(x,y) \mapsto x \tensor_{\univ\lie g} y$ 
	is called the \emph{Shapovalov pairing} between $M^*_\lambda$ and 
	$M_\lambda$.
\end{definition}
In the following it will be convenient to have alternative descriptions of 
$M_\lambda$, $M_\lambda^-$ and $M^*_\lambda$.
Let $\lbrace X_1, \dots, X_k \rbrace$ be a basis of $\plus{\lie n}$, 
$\lbrace Y_1, \dots, Y_k \rbrace$ be a basis of $\minus{\lie n}$, and 
$\lbrace H_1, \dots, H_r \rbrace$ be a basis of $\lie h$.
Since $\lie g = \plus{\lie n} \oplus \lie h \oplus \minus{\lie n}$ (as vector spaces)
the Poincar\'e--Birkhoff--Witt theorem implies that
\begin{equation}
\lbrace Y^I H^J X^K \mid I,K \in \mathbb N_0^k, J \in \mathbb N_0^r \rbrace
\quad\text{and}\quad
\lbrace X^K H^J Y^I \mid I,K \in \mathbb N_0^k, J \in \mathbb N_0^r \rbrace
\end{equation}
are bases for $\univ \lie g$.
Here we use the multiindex notation $Y^I \coloneqq Y_1^{I_1} \dots Y_k^{I_k}$
(and similarly for $H$ and $X$).
Define maps 
\begin{subequations}
\begin{align}
\pi_\lambda^- &\colon \univ \lie g \to \univ(\minus{\lie n}) \komma &
\pi_\lambda^-(Y^I H^J X^K) &\coloneqq \lambda(H_1)^{J_1} \dots \lambda(H_r)^{J_r} Y^I \delta_{K,0} \komma
\\
\pi_\lambda^+ &\colon \univ \lie g \to \univ(\plus{\lie n}) \komma &
\pi_\lambda^+(X^K H^J Y^I) &\coloneqq (-\lambda(H_1))^{J_1} \dots (-\lambda(H_r))^{J_r} X^K \delta_{I,0} 
\komma
\\
\pi_\lambda^* &\colon \univ \lie g \to \univ(\plus{\lie n}) \komma &
\pi_\lambda^*(Y^I H^J X^K) &\coloneqq \lambda(H_1)^{J_1} \dots \lambda(H_r)^{J_r} X^K \delta_{I,0} \komma
\end{align}
\end{subequations}
where $\delta_{K,0}$ is $1$ if $K = (0, \dots, 0)$ and is $0$ otherwise.
Note that $\pi_\lambda^\pm$ and $\pi_\lambda^*$
are independent of the choice of bases.
Fix non-zero vectors $1 \in\mathbb C^\pm_\lambda$ and $1 \in \mathbb C^*_\lambda$
(thinking of $\mathbb C$ as a vector space, this choice is not canonical).

\begin{lemma} \label{lemma:vermaModule}
The maps
$\argumentdot \tensor 1 \colon \univ (\minus{\lie n}) \to M_\lambda$, $v \mapsto v \tensor 1$ and
$\argumentdot \tensor 1 \colon \univ (\plus{\lie n}) \to M^-_\lambda$, $u \mapsto u \tensor 1$
define isomorphisms of left $\univ(\minus{\lie n})$-modules and $\univ(\plus{\lie n})$-modules, 
respectively. The map
$1 \tensor \argumentdot \colon \univ (\plus{\lie n}) \to M^*_\lambda$, $u \mapsto 1 \tensor u$
defines an isomorphism of right $\univ(\plus{\lie n})$-modules.
The $\univ\lie g$-module structures on $\univ(\plusminus{\lie n})$
obtained by transferring the module structures on the Verma modules
with these isomorphisms are given explicitly by
\begin{subequations}
\begin{align} \label{eq:moduleStructure:i}
\vermaacts - &\colon \univ\lie g \times \univ(\minus{\lie n}) \to \univ(\minus{\lie n}) \komma&
(w, v) &\mapsto w \vermaacts - v \coloneqq \pi^-_\lambda (w v) \komma
\\ \label{eq:moduleStructure:ii}
\vermaacts + &\colon \univ\lie g \times \univ(\plus{\lie n}) \to \univ(\plus{\lie n}) \komma&
(w, u) &\mapsto w \vermaacts + u \coloneqq \pi^+_\lambda (w u) \komma
\\ \label{eq:moduleStructure:iii}
\vermaracts &\colon \univ(\plus{\lie n}) \times \univ\lie g \to \univ(\plus{\lie n}) \komma&
(u, w) &\mapsto u \vermaracts w \coloneqq \pi^*_\lambda (u w) \punkt
\end{align}
\end{subequations}
Furthermore, $S (w \vermaacts + u) = S(u) \vermaracts S(w)$,
where $S$ denotes the antipode of $\univ \lie g$.
Or, in other words, $S \colon \univ(\plus{\lie n}) \to \univ(\plus{\lie n})$ is an isomorphism
from the left $\univ \lie g$-module $(\univ(\plus{\lie n}), \acts_\lambda^+)$ to the right
$\univ \lie g$-module $(\univ(\plus{\lie n}), \racts_\lambda^*)$ 
over the map $S \colon \univ \lie g \to \univ \lie g$.
\end{lemma}

\begin{proof}
	One checks easily that the maps
	$M_\lambda \to \univ(\minus{\lie n})$, $w \tensor z 1 \mapsto z \cdot \pi^-_\lambda(w)$ and
	$M^-_\lambda \to \univ(\plus{\lie n})$, $w \tensor z 1 \mapsto z \cdot \pi^+_\lambda(w)$
	as well as
	$M^*_\lambda \to \univ(\plus{\lie n})$, $z 1 \tensor w \mapsto z \cdot \pi^*_\lambda(w)$
	are all well-defined and inverses of the maps in the statement of the lemma.
	Consequently, $w \vermaacts - v 
	= 
	(\argumentdot \tensor 1)^{-1} (w v \tensor 1) 
	=
	\pi^-_\lambda(w v)$,
	and \eqref{eq:moduleStructure:ii} and \eqref{eq:moduleStructure:iii}
	follow similarly.
	Finally, $\pi_\lambda^* \circ S = S \circ \pi_\lambda^+$, so 
	$S( w \vermaacts + u) 
	= 
	S \circ \pi_\lambda^+ (w u) 
	= 
	\pi_\lambda^* \circ S (w u) 
	= 
	\pi_\lambda^* (S(u) S(w)) 
	=
	S(u) \vermaracts S(w)$.
\end{proof}
The pairing of the left $\univ \lie g$-modules $(\univ(\plusminus{\lie n}), \acts_\lambda^\pm)$ 
obtained from the Shapovalov pairing by composing with the isomorphisms 
$(\univ(\minus{\lie n}), \acts_\lambda^-) 
\xrightarrow{\argumentdot \tensor 1}
M_\lambda$ 
and 
$(\univ(\plus{\lie n}), \acts_\lambda^+) 
\xrightarrow{S} 
(\univ(\plus{\lie n}), \racts_\lambda^*) 
\xrightarrow{1 \tensor \argumentdot}
 M_\lambda^*$
of the previous lemma, is
\begin{equation}\label{eq:AL:pairing}
\langle \argumentdot, \argumentdot \rangle_\lambda 
\colon
\univ(\plus{\lie n}) \times \univ(\minus{\lie n}) \to \mathbb C \komma
\quad
(u,v) \mapsto \langle u, v \rangle_\lambda 
\coloneqq 
\langle 1 \tensor S(u), v \tensor 1 \rangle'_\lambda \punkt
\end{equation} 
In order to compute 
$\langle u , v \rangle_\lambda$ 
for $u \in \univ(\plus{\lie n})$ and
$v \in \univ(\minus{\lie n})$ one needs to write
$S(u) v \in \univ\lie g$ in the form $\sum_i v'_i h'_i u'_i$ with 
$u'_i \in \univ(\plus{\lie n})$, 
$v'_i \in \univ(\minus{\lie n})$ and
$h'_i \in \univ \lie h$.
The pairing is then given by summing $\lambda(h'_i)$
for those summands that have $v'_i = u'_i = 1$.
This is made more precise in the next lemma.
Define 
$\pi_\lambda 
\coloneqq
\pi_\lambda^- \circ \pi_\lambda^* 
=
\pi_\lambda^* \circ \pi_\lambda^- 
\colon
\univ \lie g \to \mathbb C$,
where $\mathbb C$ is identified with $\mathbb C 1 \subseteq \univ (\plusminus{\lie n})$
and we have implicitly used the inclusion $\univ(\plusminus{\lie n}) \to \univ \lie g$
when composing the maps.

\begin{lemma} \label{theo:shapovalovPairing:computations}
	For $u \in \univ(\plus{\lie n})$ and $v \in \univ(\minus{\lie n})$ 
	the pairing $\langle \argumentdot, \argumentdot \rangle_\lambda$ 
	defined in \eqref{eq:AL:pairing} can be 
	computed as
	\begin{equation}\label{eq:shapovalovPairing:computation:i}
	\langle u, v \rangle_\lambda 
	= 
	\pi_\lambda(S(u) v) \punkt
	\end{equation}
	It is $\univ\lie g$-invariant,
	in the sense that
	$\langle w \vermaacts + u, v \rangle_\lambda 
	= \langle u, S(w) \vermaacts - v \rangle_\lambda$
	for $u \in \univ(\plus{\lie n})$, $v \in \univ(\minus{\lie n})$ and $w \in \univ\lie g$.
The pairing respects the degree $d$ defined in the beginning of this section,
	meaning that  
	$\langle u, v \rangle_\lambda = 0$ 
	for homogeneous elements $u \in \univ(\plus{\lie n})$ and $v \in \univ(\minus{\lie n})$
	with $d(u) \neq -d(v)$. Furthermore, if $d(u) = -d(v)$, then
	\begin{equation}\label{eq:shapovalovPairing:computation:ii}
	\langle u, v \rangle_\lambda 1_{\univ(\minus{\lie n})} = S(u) \vermaacts - v 
	\quad\text{and}\quad
	\langle u, v \rangle_\lambda 1_{\univ(\plus{\lie n})} = S(v) \vermaacts + u \punkt
	\end{equation}
\end{lemma}

\begin{proof}
	By definition
	$\langle u, v \rangle_\lambda 
	=
	1 \tensor_{\univ(\minus{\lie b})} S(u) v \tensor_{\univ(\plus{\lie b})} 1$.
	So to prove \eqref{eq:shapovalovPairing:computation:i}
	it suffices to check that
	$1 \tensor_{\univ(\minus{\lie b})} w \tensor_{\univ(\plus{\lie b})} 1 
	= 
	\pi_\lambda(w)$ for all $w \in \univ \lie g$,
	which one can easily verify on the basis 
	$\lbrace Y^I H^J X^K \mid I,K \in \mathbb N_0^k, J \in \mathbb N_0^r \rbrace$.
The $\univ \lie g$-invariance follows by noting that 
	$\langle \argumentdot, \argumentdot \rangle'_\lambda$ 
	is $\univ \lie g$-invariant, meaning
	$\langle x w, y \rangle'_\lambda = \langle x, w y \rangle'_\lambda$
	for $x \in M_\lambda^*$ and $y \in M_\lambda$,
	and using the isomorphisms of the previous lemma.
For homogeneous $u \in \univ(\plus{\lie n})$ and $v \in \univ(\minus{\lie n})$
	with $d(u) \neq -d(v)$ it follows that $S(u) v$ is also homogeneous
	of degree $d(u) + d(v) \neq 0$ and therefore $\pi_\lambda (S(u) v) = 0$.
Finally, if $d(u) = - d(v)$, then $d(S(u)v) = 0$ and
	$\langle u, v \rangle_\lambda 1_{\univ(\minus{\lie n})}
	=
	\pi_\lambda(S(u) v) 1_{\univ(\minus{\lie n})} 
	=
	\pi_\lambda^-(S(u) v)
	= 
	S(u) \acts_\lambda^- v$,
	implying the first equality of \eqref{eq:shapovalovPairing:computation:ii}.
	The second one follows from applying $S$ on both sides of
	$\langle u, v \rangle_\lambda 1_{\univ(\plus{\lie n})}
	=
	\pi_\lambda(S(u)v) 1_{\univ(\plus{\lie n})} 
	= 
	\pi_\lambda^*(S(u)v) 
	= 
	S(\pi_\lambda^+ (S(v) u)) 
	= 
	S(S(v) \acts_\lambda^+ u)$.
\end{proof}
If the pairing $\langle \argumentdot, \argumentdot \rangle_\lambda$
is non-degenerate, we can pick bases
$\lbrace u_i \rbrace_{i \in \mathbb N}$ of $\univ(\plus{\lie n})$ and
$\lbrace v_j \rbrace_{j \in \mathbb N}$ of $\univ(\minus{\lie n})$
consisting of homogeneous elements with respect to $d$ and
satisfying $\langle u_i, v_j \rangle_\lambda = \delta_{ij}$.
Then the element
$F_\lambda \coloneqq \sum_{i=1}^\infty u_i \tensor v_i
\in \univ(\plus{\lie n}) \tensorhat \univ(\minus{\lie n})$
is called the \emph{canonical element} of the pairing.
It is independent of the choice of bases.
By $\univ(\plus{\lie n}) \tensorhat \univ(\minus{\lie n})$
we mean the completion of the tensor product 
with respect to the $\mathbb Z \Delta$-grading $d$
defined in the beginning of this subsection,
which is needed to make sense of the infinite sum.
The following lemma is a standard statement when working with canonical elements.

\begin{lemma}\label{theo:canonicalElement:FormulaToCheck}
	Assume that $\langle\argumentdot,\argumentdot\rangle_\lambda$
	is non-degenerate, and let 
	$F_\lambda = \sum_{i=1}^\infty u_i \tensor v_i 
	\in \univ(\plus{\lie n})\tensorhat \univ(\minus{\lie n})$
	be its canonical element. 
	Then 
	\begin{equation} \label{eq:canonicalElement:lemma}
	\sum_{i=1}^\infty u_i \langle u, v_i \rangle_\lambda = u   
	\quad\text{and}\quad
	\sum_{i=1}^\infty v_i \langle u_i, v \rangle_\lambda = v
	\end{equation} 
	hold for all $u \in \univ(\plus{\lie n})$
	and all $v \in \univ(\minus{\lie n})$,
	and $F_\lambda$ is uniquely determined by this property.
\end{lemma}
Note that $\langle u, v_i \rangle$ and $\langle u_i, v \rangle$
are non-zero for only finitely many indices $i$, so that the sums in
\eqref{eq:canonicalElement:lemma} are both finite. 
The pairing $\langle \argumentdot,\argumentdot \rangle_\lambda$ is 
non-degenerate precisely when the Verma modules are irreducible,
but we will not need this below.
In order to determine $F_\lambda$ explicitly,
we need to introduce some more notation.

Denote the Killing form of $\lie g$ by $B$.
Since $\lie g$ is semisimple, $B$ is non-degenerate on $\lie g$.
Extending linear functionals on $\lie h$ by $0$ on the root spaces
$\lie g^\alpha$, we may view $\lie h^*$ as a subspace of $\lie g^*$.
Since $B$ restricts to zero on $\lie h \times \lie g^\alpha$
for any $\alpha \in \Delta$,
it follows that $B$ is non-degenerate on $\lie h$ and
that the maps 
$^\flat\colon \lie g \to \lie g^*$ and
$^\sharp\colon \lie g^* \to \lie g$ 
defined in \autoref{subsec:generalities}
restrict to mutually inverse isomorphisms 
$^\flat\colon \lie h \to \lie h^*$ and
$^\sharp\colon \lie h^* \to \lie h$. 
For $\alpha, \beta \in \lie h^*$, let
$(\alpha, \beta) \coloneqq B(\alpha^\sharp, \beta^\sharp)$.

Denote the positive roots by $\alpha_1, \dots, \alpha_k$.
For every positive root $\alpha_i \in \Delta^+$ choose elements
$X_i \coloneqq X_{\alpha_i} \in \lie g^{\alpha_i}$ and
$Y_i \coloneqq Y_{\alpha_i} = X_{-\alpha_i} \in \lie g^{-\alpha_i}$
such that $B(X_i, Y_i) = 1$.
Then we have $[X_i, Y_i] = \alpha_i^\sharp$ since for all $H \in \lie h$, 
\begin{equation*}
B([X_i, Y_i], H) 
= B(X_i, [Y_i, H]) 
= \alpha_i(H) B(X_i, Y_i) 
= \alpha_i(H) 
= B(\alpha_i^\sharp, H)
\end{equation*}
and the Killing form is non-degenerate on $\lie h$.
Note that 
$[\alpha_i^\sharp, X_i] 
= \alpha_i(\alpha_i^\sharp) X_i 
= (\alpha_i, \alpha_i) X_i$
and similarly
$[\alpha_i^\sharp, Y_i] = -(\alpha_i, \alpha_i) Y_i$,
so $X'_i = 2(\alpha_i, \alpha_i)^{-1} X_i$,
$Y'_i = Y_i$ and
$H'_i = 2(\alpha_i, \alpha_i)^{-1} \alpha_i^{\sharp}$
satisfy the commutation relations 
$[X'_i, Y'_i] =    H'_i$,
$[H'_i, X'_i] =  2 X'_i$ and 
$[H'_i, Y'_i] = -2 Y'_i$
of the usual generators of $\liesl[2]$, the special linear Lie algebra in 2 dimensions.

Let $\rho = \frac 1 2 \sum_{\alpha\in \Delta^+} \alpha$ be 
the \emph{half-sum} of all positive roots.
Denote non-negative integral linear combinations 
of positive roots by $\mathbb N_0 \Delta^+$.
For $\lambda \in \lie h^*$ fixed,
and $\mu \in \lie h^*$ define the number
\begin{equation} \label{eq:def:plambda}
p_{\lambda}(\mu) \coloneqq \frac 1 2 (\mu, \mu) - (\rho, \mu) - (\lambda, \mu) \punkt
\end{equation}
Recall that for a representation $\varrho \colon \lie g \to V$
and $\mu \in \lie h^*$ we define 
$V^\mu 
\coloneqq 
\lbrace 
	v \in V 
	\mid 
	\varrho(H) v = \mu(H) v \text{ for all $H \in \lie h$}
\rbrace$.
If $V^\mu \neq \lbrace 0 \rbrace$, then we call $\mu$ a \emph{weight}
and any $v \in V^\mu$ is called a \emph{weight vector} of weight $\mu$.
$V$ is called a weight module if $V = \smash{\bigoplus_{\mu \in \lie h^*} V^\mu}$.
A highest weight module is a weight module 
generated by a vector $v \in V$ satisfying $X_\alpha v = 0$
for all $\alpha \in \Delta^+$.
It is said to be of highest weight $\mu$ if $v \in V^\mu$.

\begin{lemma}[{Ostapenko, {\cite[Lemma 2]{ostapenko:1992a}}}] 
	\label{theo:casimirActing}
	Let $V$ be a highest weight module of highest weight 
	$\lambda$, assume $\mu \in \mathbb N_0 \Delta^+$,
	and let $v \in V^{\lambda - \mu}$.
	Then
	\begin{equation}
	-p_\lambda(\mu) v = \sum_{\alpha \in \Delta^+} Y_{\alpha} X_{\alpha} v 
	\punkt
	\end{equation}
\end{lemma}

\begin{proof}
	Choose an orthonormal basis $\lbrace H_1, \dots, H_r\rbrace$ of $\lie h$
	with respect to the Killing form.
	The Casimir element 
	\begin{equation*}
	c
	=\sum_{\alpha \in \Delta^+} (X_{\alpha} Y_{\alpha} + 
	Y_{\alpha} X_{\alpha}) + \sum_{i=1}^r H_i H_i 
	= \sum_{\alpha \in \Delta^+} (2 Y_{\alpha} X_{\alpha} + \alpha^\sharp)
	+ \sum_{i=1}^r H_i H_i 
	\end{equation*}
	acts as a scalar on $V$ because $V$ is generated 
	by a highest weight vector and $c$ is central in $\univ\lie g$.
Evaluating it on a highest weight vector the $Y_\alpha X_\alpha$-part vanishes 
	and we obtain that $c$ acts as multiplication by 
	$\sum_{\alpha \in \Delta^+} (\alpha, \lambda) 
	+ \sum_{i=1}^r \lambda(H_i) \lambda(H_i) 
	= (2\rho, \lambda) + (\lambda, \lambda)$. 
Therefore 
	\begin{equation*}
	(2\rho,\lambda)v+(\lambda, \lambda) v 
	= 2 \sum_{\alpha \in \Delta^+} Y_{\alpha} X_{\alpha}v 
	+ (2\rho, \lambda - \mu)v  + (\lambda - \mu, \lambda - \mu)v
	\end{equation*}
	holds for any $v \in V^{\lambda - \mu}$,
	and rearranging this equation proves the lemma.
\end{proof}
\newcommand{\tikzcircleradius}{0.06}
\newcommand{\tikzvertspace}{0.8660254}
\newcommand{\tikzseparator}{0.1}
\newcommand{\tikzoffset}{0.08}
\newcommand{\tikznodeoffset}{0.2}
\newcommand{\tikzshorten}{0.95}
\newcommand{\tikzlinewidth}{1.5pt}
\begin{figure}
	\definecolor{qqzzqq}{rgb}{0,0.6,0}
	\definecolor{qqqqff}{rgb}{0,0,1}
	\definecolor{wwzzff}{rgb}{0.3,0.5,1}
	\definecolor{ffqqqq}{rgb}{1,0,0}
	\definecolor{uuuuuu}{rgb}{0.27,0.27,0.27}
	\definecolor{qqzzqq}{rgb}{0,0.6,0}
	\definecolor{qqqqff}{rgb}{0,0,1}
	\definecolor{wwzzff}{rgb}{0.3,0.5,1}
	\definecolor{ffqqqq}{rgb}{1,0,0}
	\definecolor{uuuuuu}{rgb}{0.27,0.27,0.27}
	\definecolor{qqzzqq}{rgb}{0,0.6,0}
	\definecolor{qqqqff}{rgb}{0,0,1}
	\definecolor{wwzzff}{rgb}{0.3,0.5,1}
	\definecolor{ffqqqq}{rgb}{1,0,0}
	\definecolor{uuuuuu}{rgb}{0.27,0.27,0.27}
	\definecolor{qqzzqq}{rgb}{0,0.6,0}
	\definecolor{qqqqff}{rgb}{0,0,1}
	\definecolor{wwzzff}{rgb}{0.3,0.5,1}
	\definecolor{ffqqqq}{rgb}{1,0,0}
	\definecolor{uuuuuu}{rgb}{0.27,0.27,0.27}
	\begin{tikzpicture}[line cap=round,line join=round,>=triangle 45, scale=.865]
\draw (-1.9,0) -- (1.9,0);
\draw [->,line width=1.2pt,color=ffqqqq] (0,0) -- (1.5,-0.87);
	\draw [->,line width=1.2pt,color=ffqqqq] (0,0) -- (0,-1.73);
	\draw [->,line width=1.2pt,color=ffqqqq] (0,0) -- (-1.5,-0.87);
\draw [->,line width=1.2pt,color=qqzzqq] (0,0) -- (1.5,0.87);
	\draw [->,line width=1.2pt,color=qqzzqq] (0,0) -- (-1.5,0.87);
	\draw [->,line width=1.2pt,color=qqzzqq] (0,0) -- (0,1.73);
\draw (-1.5,0.98) node[anchor=south] {$\alpha_1$};
	\draw (0,1.84) node[anchor=south] {$\alpha_2$};	
	\draw (1.5,0.98) node[anchor=south] {$\alpha_3$};
\begin{scriptsize}
	\fill [color=uuuuuu] (0,0) circle (1.5pt);
	\fill [color=uuuuuu] (0,-1.73) circle (1.5pt);
	\fill [color=uuuuuu] (-1.5,-0.87) circle (1.5pt);
	\fill [color=uuuuuu] (-1.5,0.87) circle (1.5pt);
	\fill [color=uuuuuu] (0,1.73) circle (1.5pt);
	\fill [color=uuuuuu] (1.5,0.87) circle (1.5pt);
	\fill [color=uuuuuu] (1.5,-0.87) circle (1.5pt);
	\end{scriptsize}
	\end{tikzpicture}
	\hfill
	\renewcommand{\tikzcircleradius}{0.1}
	\begin{tikzpicture}[>=triangle 45, scale=0.6]
\foreach \n in {0,1,2,3} {
		\fill ({(-\n+0)*\tikzvertspace},-0.0-0.5*\n) circle (\tikzcircleradius);
		\fill ({(-\n+1)*\tikzvertspace},-0.5-0.5*\n) circle (\tikzcircleradius);
		\fill ({(-\n+2)*\tikzvertspace},-1.0-0.5*\n) circle (\tikzcircleradius);
		\fill ({(-\n+3)*\tikzvertspace},-1.5-0.5*\n) circle (\tikzcircleradius);
	}
	\node at (0,-3.7) {$\vdots$};
	\node at (2.46,-2.42) {$\ddots$};
	\node at (-2.46,-2.42) {\reflectbox{$\ddots$}};
\fill[color=red] (0,0) circle (\tikzcircleradius);
	\node[color=red] at (0,-0.4) {$\lambda$};
\fill[color=red] (\tikzvertspace,-2.5) circle (\tikzcircleradius);
	\node[color=red] at (\tikzvertspace,-2.5-0.4) {$\lambda-\mu$};
\draw [->,line width=1pt,color=qqzzqq] 
	(-\tikzvertspace*\tikzseparator,0.5*\tikzseparator) 
	-- (-\tikzvertspace+\tikzvertspace*\tikzseparator,0.5-0.5*\tikzseparator);
	\draw [->,line width=1pt,color=qqzzqq] 
	(\tikzvertspace*\tikzseparator,0.5*\tikzseparator) 
	-- (\tikzvertspace-\tikzvertspace*\tikzseparator,0.5-0.5*\tikzseparator);
	\draw [->,line width=1pt,color=qqzzqq] 
	(0,\tikzseparator) -- (0,1-\tikzseparator);
\node[color=qqzzqq] at (0,1.2) {$\lambda+\alpha_2$};
	\node[color=qqzzqq] at (-\tikzvertspace-0.8,0.5) {$\lambda+\alpha_1$};
	\node[color=qqzzqq] at (\tikzvertspace+0.8,0.5) {$\lambda+\alpha_3$};
\end{tikzpicture}
	\hfill
	\renewcommand{\tikzcircleradius}{0.05}
	\begin{tikzpicture}[>=triangle 45, xscale=1.28, yscale=1.52]
	\draw[line width=\tikzlinewidth,color=red] (0,0) -- (-\tikzvertspace-2*\tikzoffset, 0.5);
	\draw[line width=\tikzlinewidth,color=red,->] (-\tikzvertspace-2*\tikzoffset, 0.5) -- 
	({-\tikzvertspace-2*\tikzoffset+\tikzshorten*(-\tikzvertspace-2*\tikzoffset)}, 
	0.5+\tikzshorten*0.5);
	\draw[line width=\tikzlinewidth,color=red, ->] (-\tikzvertspace-2*\tikzoffset, 0.5) -- 
	(-\tikzvertspace-2*\tikzoffset-\tikzshorten*2*\tikzoffset, 0.5+\tikzshorten*1.0);	
	\fill (-2*\tikzvertspace-4*\tikzoffset, 1) circle (\tikzcircleradius);
	\node at (-2*\tikzvertspace-4*\tikzoffset, 1+\tikznodeoffset){$11$};
	\fill (-\tikzvertspace-4*\tikzoffset, 1.5) circle (\tikzcircleradius);
	\node at (-\tikzvertspace-7*\tikzoffset, 1.5+\tikznodeoffset-\tikzoffset){$12$};
	
	\draw[line width=\tikzlinewidth,color=red] (-\tikzvertspace-2*\tikzoffset, 0.5) -- 
	(-2*\tikzoffset, 1);
	\draw[line width=\tikzlinewidth,color=red,->] (-1.35*\tikzoffset,1.02) -- 
	(-1.35*\tikzoffset,1+\tikzshorten*1.0);
	\draw[line width=\tikzlinewidth,color=red,->] (-2*\tikzoffset,1) -- 
	(-2*\tikzoffset+\tikzshorten*\tikzvertspace, 1+\tikzshorten*0.5);
	\draw[line width=\tikzlinewidth,color=red,->] (-2*\tikzoffset,1) -- 
	(-2*\tikzoffset-\tikzshorten*\tikzvertspace, 1+\tikzshorten*0.5);
	\fill (-1.5*\tikzoffset, 2) circle (\tikzcircleradius);
	\node at (-4.35*\tikzoffset, 2+\tikznodeoffset){$132$};
	\fill (-2*\tikzoffset-\tikzvertspace, 1.5) circle (\tikzcircleradius);
	\node at (-\tikzvertspace-2*\tikzoffset, 1.5+\tikznodeoffset){$131$};
	\fill (-2*\tikzoffset+\tikzvertspace, 1.5) circle (\tikzcircleradius);
	\node at (\tikzvertspace-4*\tikzoffset, 1.5+\tikznodeoffset){$133$};
	
	\draw[line width=\tikzlinewidth,color=blue] (0, 0) -- (0, 1);
	\draw[line width=\tikzlinewidth,color=blue,->] (0, 1) -- (0, 1+\tikzshorten*1.0);
	\draw[line width=\tikzlinewidth,color=blue,->] (0, 1) -- (-\tikzshorten*\tikzvertspace, 
	1+\tikzshorten*0.5);
	\draw[line width=\tikzlinewidth,color=blue,->] (0, 1) -- (\tikzshorten*\tikzvertspace, 
	1+\tikzshorten*0.5);
	\fill (0, 2) circle (\tikzcircleradius);
	\node at (2*\tikzoffset, 2+\tikznodeoffset){$22$};
	\fill (-\tikzvertspace, 1.5) circle (\tikzcircleradius);
	\node at (-\tikzvertspace+3*\tikzoffset, 1.5+\tikznodeoffset-\tikzoffset){$21$};
	\fill (+\tikzvertspace, 1.5) circle (\tikzcircleradius);
	\node at (\tikzvertspace+2*\tikzoffset, 1.5+\tikznodeoffset){$23$};
	
	\draw[line width=\tikzlinewidth,color=qqzzqq,->] (0,0) -- (\tikzvertspace*\tikzshorten, 
	0.5*\tikzshorten);
	\fill (\tikzvertspace, 0.5) circle (\tikzcircleradius);
	\node at (\tikzvertspace, 0.5+\tikznodeoffset){$3$};
	\end{tikzpicture}
	\caption{\emph{Left:} The roots of $\liesl[3]$.
		The Cartan subalgebra $\lie h$ of $\liesl[3]$ is 2-dimensional
		and there are six 1-dimensional root spaces.
		The picture shows the real subspace of $\lie h^*$ spanned by the roots.
		The positive roots are denoted by $\alpha_1$,
		$\alpha_2$ and $\alpha_3$ and drawn in green,
		negative roots are drawn in red.
\emph{Middle:} The weights in a highest weight module of highest weight $\lambda$.
		The picture shows again the real subspace of $\lie h^*$ spanned by the roots.
		Weights are indicated by black dots, 
		and $\mu = 3 \alpha_1 + 2 \alpha_3$.
		Since $\lambda$ is a highest weight,
		the spaces $V^{\lambda + \alpha_1}$,
		$V^{\lambda + \alpha_2}$ and $V^{\lambda + \alpha_3}$ must all be trivial.
\emph{Right:} Visualization of the tree 
		$T = \lbrace \emptyset,1,2,3,11,12,13,21,22,23,131,132,133 \rbrace$.
		The elements of $\max T = \lbrace 3, 11, 12, 21, 22, 23, 131, 132, 133 \rbrace$
		are indicated by black dots. 
		Words starting with $1$ are coloured red, 
		words starting with $2$ blue,
		and words starting with $3$ green.}
	\label{figure:roots:verma:tree}
\end{figure}
Let $\wordset$ be the set of words with letters from
$\lbrace 1, \dots, k\rbrace$.
For any $w = (w_1, \dots, w_{\abs w}) \in \wordset$,
we define $w^{\opp} \coloneqq (w_{\abs w}, \dots, w_1)$,
$w_{i \dots j} \coloneqq (w_i, \dots, w_j)$,
$X_w \coloneqq \smash{X_{w_1} \dots X_{w_{\abs w}}} \in \univ(\plus{\lie n})$,
$Y_w \coloneqq \smash{Y_{w_1} \dots Y_{w_{\abs w}}} \in \univ(\minus{\lie n})$ and
$\alpha_w \coloneqq \alpha_{w_1} + \dots + \alpha_{w_{\abs w}}$.
We use $w_{i \dots j} \coloneqq \emptyset$ if $j<i$,
$X_\emptyset \coloneqq 1$, $Y_\emptyset \coloneqq 1$ and $\alpha_\emptyset \coloneqq 0$.
Furthermore let 
\begin{equation}
p^w_\lambda(\mu) \coloneqq \prod_{i = 0}^{\abs w-1} 
p_{\lambda}(\mu-\alpha_{w_{1\dots i}}) 
\punkt
\end{equation}
We call a set $T$ of words a \emph{tree} 
if $w = (w_1, \dots, w_{\abs w}) \in T$ 
implies that $w_{1\dots i} \in T$ for all $i = 0, \dots, \abs w-1$
and $(w_1, w_2, \dots, w_{\abs w-1}, x) \in T$
for all $x \in \lbrace 1, \dots, k\rbrace$.
See \autoref{figure:roots:verma:tree} for a visualization of a tree.
For a tree $T$ we denote by $\max T$ the set of elements $w \in T$
such that $w \neq w'_{1\dots i}$ for any $w' \in T$
and any $i \in \lbrace 0, \dots, \abs{w'}-1\rbrace$.
Finally a tree is said to be \emph{$\mu$-admissible}
if $p_\lambda (\mu - \alpha_w) \neq 0$ for all $w \in T \setminus \max T$,
or equivalently if $p^w_\lambda(\mu) \neq 0$ for all $w \in T$.

\begin{lemma}[Ostapenko, {\cite[Theorem 3]{ostapenko:1992a}}] 
	\label{theo:casimirIterated}
	Let $V$ be a highest weight module of highest weight $\lambda$,
	assume $\mu \in \mathbb N_0 \Delta^+$, 
	and let $v \in V^{\lambda - \mu}$.
Then
	\begin{equation}\label{eq:casimirIterated}
	v = \sum_{w \in \max T} (-1)^{\abs w} p^w_\lambda(\mu)^{-1} 
	Y_{w} X_{w^{\opp}} v
	\end{equation}
	holds for every $\mu$-admissible tree $T$.
\end{lemma}

\begin{proof}
	Apply the previous lemma repeatedly.
\end{proof}

\begin{lemma}\label{theo:casimirIterated:lowestWeight}
	Let $V$ be a lowest weight module of lowest weight $-\lambda$,
	assume $\mu \in \mathbb N_0 \Delta^+$,
	and let $v \in V^{-\lambda + \mu}$.
Then
	$\sum_{\alpha \in \Delta^+} X_\alpha Y_{\alpha} v = -p_\lambda(\mu) v$,
	and
	\begin{equation}
	v = \sum_{w \in \max T} (-1)^{\abs w} p^w_\lambda(\mu)^{-1}  
	X_{w} Y_{w^{\opp}} v
	\end{equation}
    holds for every $\mu$-admissible tree $T$.
\end{lemma}
\begin{proof}
	Similar to the proof of \autoref{theo:casimirActing} and 
	\autoref{theo:casimirIterated}.
\end{proof}
Define the set 
\begin{equation}
\Lambda \coloneqq \lbrace \lambda \in \lie h^* \mid p_\lambda(\mu) \neq 0 \,\,\,
\forall \mu \in \mathbb N_0 \Delta^+ \setminus\lbrace 0 \rbrace \rbrace \punkt
\end{equation}

\begin{proposition}\label{theo:canonicalelement:restricted}
	The Shapovalov pairing
	$\langle \argumentdot, \argumentdot \rangle_\lambda 
	\colon
	\univ (\plus{\lie n}) \times \univ (\minus{\lie n}) \to \mathbb C$
	is non-degenerate for $\lambda \in \Lambda$, 
	and in this case its canonical element
	$F_\lambda \in \univ(\plus{\lie n}) \tensorhat \univ(\minus{\lie n})$
	is given by
	\begin{equation} \label{eq:canonicalelement:restricted}
	F_\lambda 
	= \sum_{w\in\wordset} p^w_{\lambda}(\alpha_w)^{-1} X_w \tensor Y_w
	= \sum_{w \in\wordset} 
	\prod_{i = 1}^{\abs w} p_{\lambda}(\alpha_{w_{i \dots \abs w}})^{-1} 
	X_w \tensor Y_w\punkt
	\end{equation}
\end{proposition}

\begin{proof}
	\renewcommand{\tikzcircleradius}{0.04}
	\renewcommand{\tikzlinewidth}{1pt}
	\newcommand{\tikzlinesep}{0.045}
	\newcommand{\tikzlinesepx}{\tikzlinesep*\tikzvertspace}
	\newcommand{\tikzlinesepy}{\tikzlinesep*0.5}
	\definecolor{qqzzqq}{rgb}{0,0.6,0}
	\definecolor{uuuuuu}{rgb}{0.27,0.27,0.27}
	\definecolor{ffqqqq}{rgb}{1,0,0}
	\figurewithtwoimages{
		\begin{tikzpicture}[line cap=round,line join=round,>=triangle 45, scale=.95]
\draw (-1.9,0) -- (1.9,0);
\draw [->,line width=1.2pt,color=ffqqqq] (0,0) -- (1.5,-0.87);
		\draw [->,line width=1.2pt,color=ffqqqq] (0,0) -- (0,-1.73);
		\draw [->,line width=1.2pt,color=ffqqqq] (0,0) -- (-1.5,-0.87);
\draw [->,line width=1.2pt,color=qqzzqq] (0,0) -- (1.5,0.87);
		\draw [->,line width=1.2pt,color=qqzzqq] (0,0) -- (-1.5,0.87);
		\draw [->,line width=1.2pt,color=qqzzqq] (0,0) -- (0,1.73);
\draw (-1.5,0.98) node[anchor=south] {$\alpha_1$};
		\draw (0,1.84) node[anchor=south] {$\alpha_2$};	
		\draw (1.5,0.98) node[anchor=south] {$\alpha_3$};
\begin{scriptsize}
		\fill [color=uuuuuu] (0,0) circle (1.5pt);
		\fill [color=uuuuuu] (0,-1.73) circle (1.5pt);
		\fill [color=uuuuuu] (-1.5,-0.87) circle (1.5pt);
		\fill [color=uuuuuu] (-1.5,0.87) circle (1.5pt);
		\fill [color=uuuuuu] (0,1.73) circle (1.5pt);
		\fill [color=uuuuuu] (1.5,0.87) circle (1.5pt);
		\fill [color=uuuuuu] (1.5,-0.87) circle (1.5pt);
		\end{scriptsize}
		\end{tikzpicture}
	}{
		\begin{tikzpicture}[>=triangle 45, scale=1.5]
		\draw[color=red, line width=\tikzlinewidth] (\tikzvertspace, -1.5) -- (0, -1);
		\draw[color=red, line width=\tikzlinewidth,->] (0, -1) -- (0,-1+\tikzshorten*1.0);
		
		\draw[color=red, line width=\tikzlinewidth] (0, -1) -- (-\tikzvertspace, -0.5);
		\draw[color=red, line width=\tikzlinewidth,->] (-\tikzvertspace, -0.5) -- 
		(-\tikzvertspace-\tikzshorten*\tikzvertspace, -0.5+\tikzshorten*0.5);
		\draw[color=red, line width=\tikzlinewidth,->] (-\tikzvertspace, -0.5) -- (-\tikzvertspace, 
		-0.5+\tikzshorten*1.0);
		\draw[color=red, line width=\tikzlinewidth,->] (-\tikzvertspace, -0.5) -- 
		(-\tikzvertspace+\tikzshorten*\tikzvertspace, -0.5+0.5*\tikzshorten);
		
		\draw[color=red, line width=\tikzlinewidth] (0, -1) -- (\tikzvertspace, -0.5);
		\draw[color=red, line width=\tikzlinewidth,->] (\tikzvertspace-\tikzlinesepx, 
		-0.5+\tikzlinesepy) -- 
		(\tikzvertspace-\tikzlinesepx+\tikzshorten*\tikzvertspace, 
		\tikzlinesepy-0.5+\tikzshorten*0.5);
		\draw[color=red, line width=\tikzlinewidth,->] (\tikzvertspace-\tikzlinesep, -0.5) -- 
		(\tikzvertspace-\tikzlinesep, -0.5+\tikzshorten*1.0);
		\draw[color=red, line width=\tikzlinewidth,->] (\tikzvertspace-\tikzlinesepx, 
		-0.5-\tikzlinesepy) -- 
		(-\tikzlinesepx+\tikzvertspace-\tikzshorten*\tikzvertspace, 
		-\tikzlinesepy-0.5+\tikzshorten*0.5);
		
		\draw[color=blue, line width=\tikzlinewidth] (\tikzvertspace, -1.5) -- (\tikzvertspace, 
		-0.5);
		\draw[color=blue, line width=\tikzlinewidth,->] (\tikzvertspace+\tikzlinesepx, 
		-0.5-\tikzlinesepy) -- 
		(\tikzvertspace+\tikzlinesepx+\tikzshorten*\tikzvertspace, 
		-\tikzlinesepy-0.5+0.5*\tikzshorten);
		\draw[color=blue, line width=\tikzlinewidth,->] (\tikzvertspace+\tikzlinesep, -0.5) -- 
		(\tikzvertspace+\tikzlinesep, -0.5+\tikzshorten*1.0);
		\draw[color=blue, line width=\tikzlinewidth,->] (\tikzvertspace+\tikzlinesepx, 
		-0.5+\tikzlinesepy) -- 
		(\tikzlinesepx+\tikzvertspace-\tikzshorten*\tikzvertspace, 
		\tikzlinesepy-0.5+\tikzshorten*0.5);
		
		\draw[color=qqzzqq, line width=\tikzlinewidth] (\tikzvertspace, -1.5) -- (2*\tikzvertspace, 
		-1);
		\draw[color=qqzzqq, line width=\tikzlinewidth] (2*\tikzvertspace, -1) -- (\tikzvertspace, 
		-0.5);
		\draw[color=qqzzqq, line width=\tikzlinewidth,->] (2*\tikzvertspace, -1) -- 
		(2*\tikzvertspace+\tikzshorten*\tikzvertspace, -1+\tikzshorten*0.5);
		\draw[color=qqzzqq, line width=\tikzlinewidth,->] (2*\tikzvertspace, -1) -- 
		(2*\tikzvertspace+\tikzshorten*2*\tikzlinesepx, {-1+\tikzshorten*(1-2*\tikzlinesepy)});
		\draw[color=qqzzqq, line width=\tikzlinewidth,->] (\tikzvertspace, -0.5) -- 
		(\tikzvertspace+\tikzshorten*\tikzvertspace, -0.5+\tikzshorten*0.5);
		\draw[color=qqzzqq, line width=\tikzlinewidth,->] (\tikzvertspace, -0.5) -- 
		(\tikzvertspace, 
		-0.5+\tikzshorten*1.0);
		\draw[color=qqzzqq, line width=\tikzlinewidth,->] (\tikzvertspace, -0.5) -- 
		(\tikzvertspace-\tikzshorten*\tikzvertspace, -0.5+\tikzshorten*0.5);
		
\foreach \n in {0,1,2} {
			\fill ({(-\n+0)*\tikzvertspace},-0.0-0.5*\n) circle (\tikzcircleradius);
			\fill ({(-\n+1)*\tikzvertspace},-0.5-0.5*\n) circle (\tikzcircleradius);
			\fill ({(-\n+2)*\tikzvertspace},-1.0-0.5*\n) circle (\tikzcircleradius);
			\fill ({(-\n+3)*\tikzvertspace},-1.5-0.5*\n) circle (\tikzcircleradius);
		}
		\node at (\tikzvertspace,-2.9) {$\vdots$};
		\node at (3.5*\tikzvertspace,-1.75) {$\ddots$};
		\node at (-2.5*\tikzvertspace,-1.25) {\reflectbox{$\ddots$}};
\node at (0,0.2) {$\lambda$};
\node at (\tikzvertspace,-1.5-0.2) {$\lambda-\mu$};
		\end{tikzpicture}
	}{The tree $T$ used in the proof of \autoref{theo:canonicalelement:restricted} for $\lie 
		g = \lie{sl}_3(\mathbb C)$ and $\mu = 2\alpha_1+\alpha_3$.
		Elements of the tree starting with $1$, $2$ and $3$
		are coloured red, blue and green, respectively.
		Note that all weight spaces of maximal elements of this tree are trivial,
		except for $V^\lambda$. All non-maximal weight spaces are non-trivial.
	}{figure:tree:proof:special}
	We check that $F_\lambda$ satisfies the property given in 
	\autoref{theo:canonicalElement:FormulaToCheck}.
We decompose $v \in \univ(\minus{\lie n})$ as
	$v = \sum_{\mu \in \mathbb N_0 \Delta^+} v_{-\mu}$
	where $v_{-\mu}$ is homogeneous of degree $-\mu$
	with respect to the $\mathbb Z \Delta$-grading.
For $\mu \in \mathbb N_0 \Delta^+$ let $\wordset_\mu$ be 
	the set of words $w \in W$ satisfying $\alpha_w = \mu$. Then
	\begin{multline*}
	\sum_{w \in\wordset}  
	p^w_{\lambda}(\alpha_w)^{-1} Y_w  
	\langle X_w, v \rangle_\lambda
	=
	\sum_{w\in\wordset}  
	p^ w_{\lambda}(\alpha_w)^{-1} Y_w \vermaacts - S(X_w) \vermaacts - v_{-\alpha_w}
	= \\ =
	\sum_{\mu \in \mathbb N_0 \Delta^+}\sum_{w\in\wordset_\mu}  
	(-1)^{\abs w} p^w_{\lambda}(\alpha_w)^{-1} Y_w \vermaacts - 
		X_{w^{\mathrm{opp}}} \vermaacts - v_{-\mu} 
	=
	\sum_{\mu \in \mathbb N_0\Delta^+} 
	v_{-\mu} = v 
	\punkt
	\end{multline*}
	The first equality holds because 
	$Y_w \langle X_w, v \rangle_\lambda  
	=
	Y_w \vermaacts - (
		\langle X_w, v_{-\alpha_w} \rangle_\lambda 1_{\univ(\minus{\lie n})}
	)
	= Y_w \vermaacts - S(X_w) \vermaacts - v_{-\alpha_w}$
	by \autoref{theo:shapovalovPairing:computations}.
The second equality is true by basic manipulations.
The third equality follows from \autoref{theo:casimirIterated}
	because we can rewrite the sum over all $w \in \wordset_\mu$
	as a sum over $\max T$ for a $\mu$-admissible tree $T$ as follows:
	Define
	\begin{equation*}
	T \coloneqq \lbrace \emptyset \rbrace 
	\cup
	\lbrace w \in W \mid 
	\exists w' \in \wordset_\mu \text{ and $0 \leq i \leq \abs{w'}-1$
		such that $w_{1 \dots \abs w -1} = w'_{1 \dots i}$} 
	\rbrace \komma
	\end{equation*}
	which is the smallest tree containing $\wordset_\mu$.
Since $\lambda \in \Lambda$ this tree is $\mu$-admissible, 
	and clearly $\wordset_\mu \subseteq \max T$.
Let $w \in \max T$. 
	Then either	$\alpha_w = \mu$, so that $w \in \wordset_\mu$,
	or there do not exist $w' \in \wordset_\mu$
	and $i \in \lbrace 0, \dots, \abs{w'}\rbrace$ 
	with $w = w'_{1 \dots i}$,
	so that $\mu - \alpha_w \notin \mathbb N_0 \Delta^+$ 
	and therefore $X_{w^{\mathrm{opp}}} v_{-\mu} = 0$. 
	
	Similarly, for $u=\sum_{\mu\in\mathbb N_0 \Delta^+} 
	u_\mu \in \univ(\plus{\lie n})$
	with $d(u_\mu) = \mu$ we compute that
	\begin{multline*}
	\sum_{w \in\wordset}  
	p^w_{\lambda}(\mu)^{-1} X_w 
	\langle u,Y_w \rangle_\lambda
	=
	\sum_{w \in\wordset}  
	p^w_{\lambda}(\mu)^{-1} 
	X_w \acts^+_\lambda S(Y_w) \acts^+_\lambda u_{\alpha_w} 
	= \\ =
	\sum_{\mu\in\mathbb N_0 \Delta^+}
	\sum_{w\in\wordset_\mu}  
	(-1)^{\abs w} p^w_{\lambda}(\mu)^{-1} 
	X_w \acts^+_\lambda Y_{w^{\mathrm{opp}}} \acts^+_\lambda u_{\mu} 
	= \sum_{\mu \in \mathbb N_0\Delta^+} u_\mu
	= u \komma
	\end{multline*}
	using 
	$X_w \langle u, Y_w \rangle_\lambda 
	= X_w \acts_\lambda^+ (\langle 
	 u_{\alpha_w}, Y_w \rangle_\lambda 1_{\univ(\plus{\lie n})})
	= X_w \acts^+_\lambda S(Y_w) \acts^+_\lambda u_{\alpha_w}$,
	and that the sum over $w \in \wordset_\mu$ can be rewritten as a sum
	over maximal elements of a tree $T$ in a similar way as before.$\,$
\end{proof}
Using the inclusion 
$\univ(\plus{\lie n})\tensorhat\univ(\minus{\lie n}) 
\to (\univ\lie g)^{\hat\tensor 2}$
and passing to the quotient,
we can map the element $F_\lambda$ from \eqref{eq:canonicalelement:restricted}
to $(\univ\lie g / \univ\lie g \cdot \lie h)^{\smash{\hat\tensor 2}}$.
Note that $\univ \lie g \cdot \lie h$ is a homogeneous ideal in $\univ \lie g$
with respect to the degree $d$,
so the quotient $\univ \lie g / \univ \lie g \cdot \lie h$ is still graded. 
The completed tensor product is defined
with respect to this grading. 
The action of $\lie h$ on $(\univ \lie g)^{\tensor 2}$ given by 
$H \acts (w \tensor w') = \ad_H w \tensor w' + w \tensor \ad_H w'$
with $H \in \lie h$ and $w, w' \in \univ \lie g$ stays well-defined
on the quotient and preserves the degree,
so it extends uniquely to a continuous action on the completed tensor product.
Denote the coproduct of the Hopf algebra $\univ \lie g$ by $\Delta$.
It is defined by extending the assignment 
$\lie g \ni X \mapsto X \tensor 1 + 1 \tensor X \in \univ \lie g \tensor \univ \lie g$ 
to an algebra homomorphism 
$\Delta \colon \univ \lie g \to \univ \lie g \tensor \univ \lie g$.

\begin{proposition}[{Alekseev--Lachowska \cite{alekseev.lachowska:2005a}}]
	Let $\lambda \in \Lambda$.
Then $F_\lambda \in (\univ\lie g/ \univ\lie g \cdot \lie h)^{\hat\tensor 2}$
	is $\lie h$-invariant and satisfies
	\begin{equation}\label{eq:associativity}
	((\id \tensor \Delta) F_\lambda) 1 \tensor F_\lambda
	= 
	((\Delta \tensor \id) F_\lambda) F_\lambda \tensor 1 
	\end{equation}
	in $(\univ \lie g / \univ \lie g \cdot \lie h)^{\hat\tensor 3}$.
\end{proposition}
\begin{proof}
	See the proof of \autoref{theo:starproduct:alekseev}.
\end{proof}
Using the results of \autoref{subsec:diffopsOnHomogeneousSpaces},
elements of $((\univ \lie g / \univ \lie g \cdot \lie h)^{\tensor 2})^H$
determine bidifferential operators on a complex coadjoint orbit
for which $\lie g_\lambda = \lie h$.
Such orbits are of maximal dimension among all coadjoint orbits
and called \emph{regular}.
Note that $H$ is automatically connected by \autoref{theo:stabilizerConnected},
so $\lie h$-invariance of $F_\lambda$ implies $H$-invariance,
but $F_\lambda$ is only an element of the completed tensor product.
So applying the construction from \autoref{subsec:diffopsOnHomogeneousSpaces} naively
gives a sum of bidifferential operators of increasing orders.
To make sense of this sum, we can either introduce a formal parameter $\formParam$
in the construction in such a way that we obtain a formal power series of bidifferential operators, 
or we can restrict ourselves to applying these operators to some class of polynomials,
for which only finitely many of the bidifferential operators appearing in the sum give a non-zero 
contribution.

We will now proceed as follows:
In \autoref{subsec:twist:general}, we generalize the construction of $F_\lambda$
to work for arbitrary stabilizers $\lie g_\lambda$ (and not just $\lie h$).
In \autoref{subsec:starProductFromTwist} we will give details on how to construct
bidifferential operators out of $F_\lambda$, both in the formal and polynomial settings
mentioned above.
 
\subsection{Generalization to non-regular orbits}
\label{subsec:twist:general}

The aim of this subsection is to generalize the results of the last subsection 
to non-regular semisimple coadjoint orbits.
To achieve this, we need to replace $\lie h$ by a possibly larger stabilizer
$\lie g_\lambda$ and define a generalization of the Shapovalov pairing.
When this pairing is non-degenerate, we derive an explicit formula for its 
canonical element, which satisfies \eqref{eq:associativity}.

Let $\lie g$ be a complex semisimple Lie algebra
acting under the coadjoint action, i.e.\ the action dual
to the adjoint action, on its dual $\lie g^*$.
We assume that $\lambda \in \lie g^*$ is semisimple
(as defined in \autoref{subsec:generalities})
with stabilizer
$\lie g_\lambda \coloneqq \lbrace X \in \lie g \mid \ad_X^*\lambda = 0 \rbrace$. 
We fix a Cartan subalgebra $\lie h$ containing $\lambda^\sharp$
(which is possible since $\lambda$ is semisimple)
and denote the corresponding root system by $\Delta$.
Since any $H \in \lie h$ commutes with $\lambda^\sharp$,
it follows that 
$\ad_H^* \lambda 
= \lambda([-H, \argumentdot]) 
= -B(\lambda^\sharp, [H, \argumentdot]) 
= -B([\lambda^\sharp, H], \argumentdot) 
= 0$,
so $\lie h \subseteq \lie g_\lambda$.
We let 
\begin{equation}
\Delta'
\coloneqq \lbrace \alpha \in \Delta \mid (\alpha, \lambda) = 0 \rbrace
\quad\text{and}\quad
\hat\Delta
\coloneqq \lbrace \alpha \in \Delta \mid (\alpha, \lambda) \neq 0 \rbrace
= \Delta \setminus \Delta' \punkt
\end{equation}
One checks easily that
$\lie g_\lambda = \lie h \oplus \bigoplus_{\alpha \in \Delta'} \lie g^\alpha$.
Given an ordering on $\Delta$ with $\Delta^\pm$ being the set 
of positive respectively negative roots,
define $\smash{\hat\Delta^\pm} = \Delta^\pm \cap \smash{\hat \Delta}$ and 
$(\Delta')^\pm = \Delta^\pm \cap \Delta'$.
Furthermore, let 
$\plusminus{\tilde{\lie n}} \coloneqq \bigoplus_{\alpha \in \hat\Delta^\pm} \lie g^\alpha$ and
$\plusminus{\tilde{\lie b}} \coloneqq \lie g_\lambda \oplus \plusminus{\tilde{\lie n}}$.

\begin{definition}\label{def:invariantOrdering}
	An ordering of $\Delta$ is called \emph{invariant} if for any $\alpha \in \hat 
	\Delta^+$ and $\beta \in \Delta'$ such that $\alpha + \beta$ is again a 
	root, this root $\alpha + \beta$ is in $\hat \Delta^+$. 
\end{definition}
Note that since the sum of two roots in $\Delta'$ is again in $\Delta'$ (if it 
is a root), it is automatic that $\alpha + \beta \in \hat \Delta$. The 
important part of the previous definition is that $\alpha + \beta$ should again 
be positive.
See \autoref{figure:orderings} for an example 
of invariant and non-invariant orderings.

\begin{lemma}
	An ordering of $\Delta$ is invariant if and only if 
	$\alpha+\beta \in \hat\Delta^+$ holds
	for any $\alpha, \beta \in \hat\Delta^+$ 
	with $\alpha+\beta \in \Delta$.
\end{lemma}
In the condition of the lemma it is automatic that $\alpha + \beta$ is positive
and the important part is that it lies in $\hat\Delta$.

\begin{proof}
	Assume the condition of the lemma is false,
	i.e.\ $\alpha,\beta \in \hat\Delta^+$
	and $\alpha + \beta \in \Delta \setminus \hat\Delta^+$.
	Since $\alpha + \beta$ is positive
	we must then have $\alpha + \beta \in \Delta'$.
	Consequently $\alpha + (-(\alpha+\beta)) = -\beta \notin \hat\Delta^+$,
	so the ordering is not invariant.
	
	Conversely, if the ordering is not invariant, then we can find
	$\alpha \in \hat\Delta^+$ and $\beta \in \Delta'$
	such that $\alpha+\beta \in \Delta \setminus \hat\Delta^+$.
	Then we must have $\alpha + \beta \in \hat\Delta^-$ and therefore
	$\alpha + (-(\alpha+\beta)) = - \beta \notin \hat\Delta^+$,
	so the condition of the lemma is not fulfilled.
\end{proof}
Intuitively the invariance of an ordering means that roots in $\Delta'$
are close to being simple, 
or more precisely that they are linear combinations of simple roots in $\Delta'$.
Indeed, if $\alpha \in (\Delta')^+$, then $\alpha$ is a non-negative linear
combination of simple roots.
By the lemma at least one of those simple roots, 
say $\sigma$, must be in $\Delta'$, so $\alpha = \sigma$ or $\alpha-\sigma \in (\Delta')^+$
and we can apply induction.

\begin{corollary}
	If the ordering of $\Delta$ is invariant,
	then $\plusminus{\tilde {\lie n}}$ and $\plusminus{\tilde {\lie b}}$
	are both Lie subalgebras of $\lie g$.
	Moreover, 
	$[\lie g_\lambda, \plusminus{\tilde{\lie n}}] \subseteq \plusminus{\tilde{\lie n}}$
	and 
	$[\lie g_\lambda, \plusminus{\tilde{\lie b}}] \subseteq \plusminus{\tilde{\lie b}}$.
\end{corollary}

\begin{proof}
	The condition in the previous lemma says precisely that
	$[\plusminus{\tilde{\lie n}}, \plusminus{\tilde{\lie n}}] 
	\subseteq
	\plusminus{\tilde{\lie n}}$,
	i.e.\ that $\plusminus{\tilde{\lie n}}$ is a Lie subalgebra of $\lie g$.
	The defining property of an invariant ordering means that
	$[\lie g_\lambda, \plusminus{\tilde{\lie n}}] \subseteq \plusminus{\tilde{\lie n}}$.
	The statements for $\plusminus{\tilde{\lie b}}$ are then clear.
\end{proof}

\begin{definition}
We say an ordering is \emph{standard} if there is a set 
$S \subseteq \mathbb C \setminus \lbrace 0 \rbrace$,
closed under addition and satisfying $S \cap (-S) = \emptyset$,
$S \cup (-S) = \mathbb C \setminus \lbrace 0 \rbrace$
such that $\alpha \in \hat\Delta$ is positive if and only if $(\alpha, \lambda) \in S$.
\end{definition}
Standard invariant orderings exist always
since we can construct them as follows.
First, take any ordering on the set
$\Delta'$ 
(meaning a subset $(\Delta')^+$ such that
if the sum of two elements of $(\Delta')^+$ is in $\Delta'$,
then it is in $(\Delta')^+$ and such that
for $(\Delta')^- \coloneqq - (\Delta')^+$ we have 
$(\Delta')^+ \cup (\Delta')^- = \Delta'$ and 
$(\Delta')^+ \cap (\Delta')^- = \emptyset$).
Then choose a set $S$ that is closed under addition and satisfies 
$S \cap (-S) = \emptyset$ and $S \cup (-S) = \mathbb C \setminus \lbrace 0 \rbrace$,
e.g.\ $S = \lbrace z \in \mathbb C \setminus \lbrace 0 \rbrace \mid \RE(z) > 0 \text{ or $z \in \I 
\mathbb R^+$} \rbrace$.
Let $\alpha \in \Delta$ be positive if $\alpha \in (\Delta')^+$ or 
$(\alpha, \lambda) \in S$.

For real coadjoint orbits standard invariant orderings are the ones which induce 
star products of pseudo Wick type (under some further assumptions,
see \autoref{theo:starProduct:types}),
and therefore the orderings we are mainly interested in.
However, the construction below works also for other
(possibly non-standard) invariant orderings.

\begin{figure}
	\definecolor{qqzzqq}{rgb}{0,0.6,0}
	\definecolor{qqqqff}{rgb}{0,0,1}
	\definecolor{wwzzff}{rgb}{0.3,0.5,1}
	\definecolor{ffqqqq}{rgb}{1,0,0}
	\definecolor{uuuuuu}{rgb}{0.27,0.27,0.27}
	\begin{tikzpicture}[line cap=round,line join=round,>=triangle 
	45,x=1.0cm,y=1.0cm, scale=.95]
	\clip(-2.2,-2.2) rectangle (2.2,2.2);
\fill[color=green!10] (-4.47,7.74) -- (0,0) -- (4.42,7.66) -- cycle; 
\draw [domain=-9.1:9.99] plot(\x,{(-0-0*\x)/1});
	\draw [domain=-9.1:9.99] plot(\x,{(-0-0.87*\x)/-0.5});
	\draw [domain=-9.1:9.99] plot(\x,{(-0--0.87*\x)/-0.5});
\fill [color=green!30] (-1.5,0.87) circle (5pt);
	\fill [color=green!30] (1.5,0.87) circle (5pt);
	\draw [color=uuuuuu] (-1.5,0.87) circle (5pt);
	\draw [color=uuuuuu] (1.5,0.87) circle (5pt);
\draw [->,line width=1.2pt,color=ffqqqq] (0,0) -- (1.5,-0.87);
	\draw [->,line width=1.2pt,color=ffqqqq] (0,0) -- (0,-1.73);
	\draw [->,line width=1.2pt,color=ffqqqq] (0,0) -- (-1.5,-0.87);
\draw [->,line width=1.2pt,color=qqzzqq] (0,0) -- (1.5,0.87);
	\draw [->,line width=1.2pt,color=qqzzqq] (0,0) -- (-1.5,0.87);
	\draw [->,line width=1.2pt,color=qqzzqq] (0,0) -- (0,1.73);
\fill [color=blue] (1.4,1.4) circle (1.5pt);
	\draw [color=blue](1.43,1.65) node[anchor=north west] {$\lambda$};
\begin{scriptsize}
	\fill [color=uuuuuu] (0,0) circle (1.5pt);
	\fill [color=uuuuuu] (0,-1.73) circle (1.5pt);
	\fill [color=uuuuuu] (-1.5,-0.87) circle (1.5pt);
	\fill [color=uuuuuu] (-1.5,0.87) circle (1.5pt);
	\fill [color=uuuuuu] (0,1.73) circle (1.5pt);
	\fill [color=uuuuuu] (1.5,0.87) circle (1.5pt);
	\fill [color=uuuuuu] (1.5,-0.87) circle (1.5pt);
	\end{scriptsize}
	\end{tikzpicture}
	\hfill
	\begin{tikzpicture}[line cap=round,line join=round,>=triangle 
	45,x=1.0cm,y=1.0cm, scale=.95]
	\clip(-2.2,-2.2) rectangle (2.2,2.2);
\fill[color=green!10] 
	(-4.47,7.74) -- (0,0) -- 
	(4.42,7.66) -- cycle; 
\draw [domain=-9.1:9.99] plot(\x,{(-0-0*\x)/1});
	\draw [domain=-9.1:9.99] plot(\x,{(-0-0.87*\x)/-0.5});
	\draw [domain=-9.1:9.99] plot(\x,{(-0--0.87*\x)/-0.5});
\fill [color=blue!20] (1.5,0.87) circle (5pt);
	\fill [color=green!30] (-1.5,0.87) circle (5pt);
	\draw [color=uuuuuu] (-1.5,0.87) circle (5pt);
	\draw [color=uuuuuu] (1.5,0.87) circle (5pt);
\draw [line width=1pt,color=blue,dashed] (0,0) -- (-1.5,-0.87);
	\draw [->,line width=1pt,color=ffqqqq] (0,0) -- (0,-1.73);
	\draw [->,line width=1pt,color=ffqqqq] (0,0) -- (1.5,-0.87);
\draw [->,line width=1pt,color=qqzzqq] (0,0) -- (-1.5,0.87);
	\draw [line width=1pt,color=blue,dashed] (0,0) -- (1.5,0.87);
	\draw [->,line width=1pt,color=qqzzqq] (0,0) -- (0,1.73);
\fill [color=blue] (-0.805,1.4) circle (1.5pt);
	\draw [color=blue](-0.835,1.65) node[anchor=north east] {$\lambda$};
	\draw [color=blue](-1.6,-0.4) node[anchor=west]{$\Delta'$};
\begin{scriptsize}
	\fill [color=uuuuuu] (0,0) circle (1.5pt);
	\fill [color=uuuuuu] (0,-1.73) circle (1.5pt);
	\fill [color=uuuuuu] (-1.5,-0.87) circle (1.5pt);
	\fill [color=uuuuuu] (-1.5,0.87) circle (1.5pt);
	\fill [color=uuuuuu] (0,1.73) circle (1.5pt);
	\fill [color=uuuuuu] (1.5,0.87) circle (1.5pt);
	\fill [color=uuuuuu] (1.5,-0.87) circle (1.5pt);
	\end{scriptsize}
	\end{tikzpicture}
	\hfill
	\begin{tikzpicture}[line cap=round,line join=round,>=triangle 
	45,x=1.0cm,y=1.0cm, scale=.95]
	\clip(-2.2,-2.2) rectangle (2.2,2.2);
\fill[color=green!10] 
	(-4.47,7.74) -- (0,0) -- 
	(4.42,7.66) -- cycle; 
\draw [domain=-9.1:9.99] plot(\x,{(-0-0*\x)/1});
	\draw [domain=-9.1:9.99] plot(\x,{(-0-0.87*\x)/-0.5});
	\draw [domain=-9.1:9.99] plot(\x,{(-0--0.87*\x)/-0.5});
\fill [color=green!30] (-1.5,0.87) circle (5pt);
	\fill [color=green!30] (1.5,0.87) circle (5pt);
	\draw [color=uuuuuu] (-1.5,0.87) circle (5pt);
	\draw [color=uuuuuu] (1.5,0.87) circle (5pt);
\draw [line width=1pt,color=blue,dashed] (0,0) -- (0,-1.73);
	\draw [->,line width=1pt,color=ffqqqq] (0,0) -- (1.5,-0.87);
	\draw [->,line width=1pt,color=ffqqqq] (0,0) -- (-1.5,-0.87);
\draw [->,line width=1pt,color=qqzzqq] (0,0) -- (1.5,0.87);
	\draw [line width=1pt,color=blue,dashed] (0,0) -- (0,1.73);
	\draw [->,line width=1pt,color=qqzzqq] (0,0) -- (-1.5,0.87);
\fill [color=blue] (1.4,0) circle (1.5pt);
	\draw [color=blue](1.4,0) node[anchor=south] {$\lambda$};
	\draw [color=blue](0,-1.2) node[anchor=west]{$\Delta'$};
\begin{scriptsize}
	\fill [color=uuuuuu] (0,0) circle (1.5pt);
	\fill [color=uuuuuu] (0,-1.73) circle (1.5pt);
	\fill [color=uuuuuu] (-1.5,-0.87) circle (1.5pt);
	\fill [color=uuuuuu] (-1.5,0.87) circle (1.5pt);
	\fill [color=uuuuuu] (0,1.73) circle (1.5pt);
	\fill [color=uuuuuu] (1.5,0.87) circle (1.5pt);
	\fill [color=uuuuuu] (1.5,-0.87) circle (1.5pt);
	\end{scriptsize}
	\end{tikzpicture}
	\caption{Invariant and non-invariant orderings.
		As in the left picture of \autoref{figure:roots:verma:tree} the roots
		of $\liesl[3]$ are shown.
		Simple roots are encircled.
		Roots in $\Delta'$ are drawn with blue dashed lines.
		Roots in $\hat\Delta$ are drawn in green if they are positive,
		and in red if they are negative.
		The fundamental Weyl chamber has a light green background.
		A regular orbit of $\SL[3]$ is shown on the left,
		the other two pictures are of non-regular orbits.
		In the right picture the ordering on $\Delta$ is not invariant,
		since adding the negative root in $\Delta'$ (the lower blue dashed line)
		to one of the positive roots (a green arrow)
		gives a negative root (a red arrow).
		The ordering in the middle picture is invariant and standard,
		the ordering in the left picture is invariant, but not standard.
		It would be standard if $\lambda$ was in the fundamental Weyl chamber.}
	\label{figure:orderings}
\end{figure}

Before generalizing the results of the last subsection,
we would like to mention the following technical lemma for later use:
\begin{lemma} \label{theo:sumsInHatCannotHaveManyPrimedElements}
	Let $\lie g$ be a semisimple Lie algebra,
	let $\lambda \in \lie g^*$ be semisimple,
	and let $\lie h$ be a Cartan subalgebra of $\lie g$ containing $\lambda^\sharp$.
	Assume that we have chosen an invariant ordering
	defining sets $\Delta^+$, $\smash{\hat\Delta}$, and $\Delta'$ as above.
	Then there is a constant $M \in \mathbb N$ such that
	for any $m \in \mathbb N$ the sum of $m$ positive roots in $\hat\Delta^+$
	and at least $M m$ positive roots in $(\Delta')^+$ is not in $\mathbb N_0 \hat\Delta^+$.
\end{lemma}
\begin{proof}
	Label the simple roots by $\sigma_1, \dots, \sigma_r$ such that the first $r'$ simple 
	roots $\sigma_1, \dots, \sigma_{r'}$ are in $\Delta'$ and the remaining simple roots
	 are in $\hat \Delta$.
	Label all roots in $\hat\Delta^+$ by $\alpha_1, \dots, \smash{\alpha_{\tilde k}}$.
	Then there are unique non-negative integers $c^i_j \in \mathbb N_0$ such that 
	$\alpha_j = \sum_{i=1}^r c^i_j \sigma_i$. Set
	$M' = \smash{\max_{j \in \lbrace 1, \dots, \tilde k \rbrace} \sum_{i=1}^{r'} c^i_j}$,
	$M'' = \max_{j \in \lbrace 1, \dots, \tilde k \rbrace} \sum_{i=r'+1}^r c^i_j$
	and $M = M' M''+1$.
	
	Since $\alpha_j \in \smash{\hat\Delta^+}$ we have $\sum_{i=r'+1}^r c_j^i \geq 1$,
	and 
	$\sum_{i=1}^{r'} c^i_j \leq M' \leq M' \sum_{i=r'+1}^r c^i_j$ 
	for any $j \in \lbrace 1, \dots, \tilde k \rbrace$.
	Note that any element $\beta \in \mathbb N_0 \hat\Delta^+$ can be 
	written uniquely as 
	$\beta = \sum_{i=1}^r \beta^i \sigma_i$ with $\beta^i \in \mathbb N_0$, 
	and the coefficients satisfy the same inequality 
	$\sum_{i=1}^{r'} \beta^i \leq M' \sum_{i=r'+1}^r \beta^i$.
	 
	Recall that any root in $(\Delta')^+$ is a linear combination
	of simple roots in $(\Delta')^+$.
	So if $\sum_{i=1}^r d^i \sigma_i \in (\Delta')^+$,
	then $d^i = 0$ for all $i = r'+1, \dots, r$.
	Therefore, if $\gamma$ is the sum of $m$ roots from $\hat\Delta^+$
	and at least $Mm$ roots from $(\Delta')^+$,
	and $\gamma = \sum_{i=1}^r \gamma^i \sigma_i$, 
	then
	 $M' \sum_{i=r'+1}^r \gamma^i
	 \leq
	 M' M'' m < M m
	 \leq 
	 \sum_{i=1}^{r'} \gamma^i $, 
	 so $\gamma$ cannot be in $\mathbb N_0 \hat \Delta^+$.
\end{proof}
Note that for a regular coadjoint orbit, we have $\Delta' = \emptyset$.
Consequently $\hat\Delta = \Delta$, $\lie g_\lambda = \lie h$,
$\plus{\tilde{\lie n}} = \plus{\lie n}$ and 
$\minus{\tilde{\lie n}} = \minus{\lie n}$.
In this case every ordering is invariant,
and the generalized Shapovalov pairing, that we will introduce now,
coincides with the Shapovalov pairing introduced in the last subsection.
Since $\lie g_\lambda = \lie h$ when $\Delta' = \emptyset$,
we usually denote an element of $\lie g_\lambda$ by $H$.

Let $\lambda \in \lie g_\lambda^*$ be the restriction
of $\lambda \in \lie g^*$ to $\lie g_\lambda$.
Then $\lambda([H', H]) = \ad^*_H \lambda(H') = 0$
for all $H, H' \in \lie g_\lambda$,
so $H \acts z = \lambda(H) z$ makes $\mathbb C$ a left or right
$\lie g_\lambda$-module. Extending trivially along $\plusminus{\tilde{\lie n}}$
gives a left or right $\plusminus{\tilde{\lie b}}$-module,
and we denote the corresponding left $\univ(\plusminus{\tilde{\lie b}})$-module
by $\tilde{\mathbb C}^\pm_\lambda$ and the right $\univ(\minus{\tilde{\lie b}})$-module
by $\tilde{\mathbb C}^*_\lambda$. 
Define the \emph{generalized Verma modules}  
\begin{equation}
\tilde M_\lambda 
= 
\univ\lie g \tensor_{\univ(\plus{\tilde{\lie b}})} \tilde{\mathbb C}^+_\lambda 
\komma\quad 
\tilde M_\lambda^- 
=
\univ\lie g \tensor_{\univ(\minus{\tilde{\lie b}})} \tilde{\mathbb C}^*_{-\lambda} 
\komma \quad\text{and}\quad
\tilde M^*_\lambda 
=
\tilde{\mathbb C}_\lambda^* \tensor_{\univ(\minus{\tilde{\lie b}})} \univ\lie g \punkt
\end{equation}
$\tilde M_\lambda$ and $\tilde M_\lambda^-$ are left $\univ \lie g$-modules,
$\tilde M_\lambda^*$ is a right $\univ\lie g$-module.
Most of the results of the previous subsection
have obvious analogues in this setting.

Let $\lbrace X_1, \dots, X_{\tilde k} \rbrace$ be a basis of $\plus{\tilde{\lie n}}$, 
$\lbrace Y_1, \dots, Y_{\tilde k} \rbrace$ be a basis of $\minus{\tilde{\lie n}}$, and 
$\lbrace H_1, \dots, H_{\tilde r} \rbrace$ be a basis of $\lie g_\lambda$.
Since $\lie g = \plus{\tilde{\lie n}} \oplus \lie g_\lambda \oplus \minus{\tilde{\lie n}}$
the Poincar\'e--Birkhoff--Witt theorem implies that
\begin{equation}
\lbrace Y^I H^J X^K \mid I,K \in \mathbb N_0^{\tilde k}, J \in \mathbb N_0^{\tilde r} \rbrace
\quad\text{and}\quad
\lbrace X^K H^J Y^I \mid I,K \in \mathbb N_0^{\tilde k}, J \in \mathbb N_0^{\tilde r} \rbrace
\end{equation}
are bases for $\univ \lie g$.
Define maps 
\begin{subequations}
	\label{eq:projPiLambda}
	\begin{align}
	\tilde\pi_\lambda^- &\colon \univ \lie g \to \univ(\minus{\tilde{\lie n}}) \komma &
	\tilde\pi_\lambda^-(Y^I H^J X^K) &\coloneqq \lambda(H_1)^{J_1} \dots \lambda(H_{\tilde r})^{J_{\tilde r}} 
	Y^I 
	\delta_{K,0} \komma
	\\
	\tilde \pi_\lambda^+ &\colon \univ \lie g \to \univ(\plus{\tilde{\lie n}}) \komma &
	\tilde \pi_\lambda^+(X^K H^J Y^I) &\coloneqq (-\lambda(H_1))^{J_1} \dots (-\lambda(H_{\tilde 
	r}))^{J_{\tilde r}} X^K 
	\delta_{I,0} \komma
	\\
	\tilde \pi_\lambda^* &\colon \univ \lie g \to \univ(\plus{\tilde{\lie n}}) \komma &
	\tilde \pi_\lambda^*(Y^I H^J X^K) &\coloneqq \lambda(H_1)^{J_1} \dots \lambda(H_{\tilde r})^{J_{\tilde 
	r}} X^K 
	\delta_{I,0} \punkt
	\end{align}
\end{subequations}
Note that they are compatible with the maps $\pi_\lambda^-$, $\pi_\lambda^+$, and 
$\pi_\lambda^*$ in the sense that $\tilde \pi_\lambda^- \circ \pi_\lambda^- = \tilde 
\pi_\lambda^-$, $\tilde \pi_\lambda^+ \circ \pi_\lambda^+ = \tilde \pi_\lambda^+$, and $\tilde 
\pi_\lambda^* \circ \pi_\lambda^* = \tilde \pi_\lambda^*$. On the left hand sides, we are 
implicitly using the inclusion $\univ(\plusminus{\lie n}) \to \univ \lie g$.
Note that this inclusion is not a $\univ \lie g$-module map.

\begin{lemma} \label{lemma:genVermaModule}
	The maps
	$\argumentdot \tensor 1 \colon \univ (\minus{\tilde{\lie n}}) \to \tilde M_\lambda$,
	$v \mapsto v \tensor 1$ and
	$\argumentdot \tensor 1 \colon \univ (\plus{\tilde{\lie n}}) \to \tilde M^-_\lambda$,
	$u \mapsto u \tensor 1$
	define isomorphisms of left $\univ(\minus{\tilde{\lie n}})$-modules 
	and $\univ(\plus{\tilde{\lie n}})$-modules, 
	respectively. The map
	$1 \tensor \argumentdot \colon \univ (\plus{\tilde{\lie n}}) \to M^*_\lambda$,
	$u \mapsto 1 \tensor u$
	is an isomorphism of right $\univ(\plus{\tilde{\lie n}})$-modules.
	The $\univ\lie g$-module structures on $\univ(\plusminus{\tilde{\lie n}})$
	obtained by transferring the module structures on the generalized Verma modules
	with these isomorphisms are given explicitly by
	\begin{subequations}
		\begin{align} \label{eq:genModuleStructure:i}
		\tildevermaacts -&\colon \univ\lie g \times \univ(\minus{\tilde{\lie n}}) \to 
		\univ(\minus{\tilde{\lie n}}) 
		\komma&
		(w, v) &\mapsto w \tildevermaacts - v \coloneqq \tilde \pi^-_\lambda (w v) \komma
		\\ \label{eq:genModuleStructure:ii}
		\tildevermaacts +&\colon \univ\lie g \times \univ(\plus{\tilde{\lie n}}) \to 
		\univ(\plus{\tilde{\lie n}}) \komma&
		(w, u) &\mapsto w \tildevermaacts + u \coloneqq \tilde \pi^+_\lambda (w u) \komma
		\\ \label{eq:genModuleStructure:iii}
		\tildevermaracts &\colon \univ(\plus{\tilde{\lie n}}) \times \univ\lie g \to 
		\univ(\plus{\tilde{\lie n}}) 
		\komma&
		(u, w) &\mapsto u \tildevermaracts w \coloneqq \tilde \pi^*_\lambda (u w) \punkt
		\end{align}
	\end{subequations}
	Furthermore, $S (w \tildevermaacts + u) = S(u) \tildevermaracts S(w)$,
	where $S$ denotes the antipode of $\univ \lie g$.
\end{lemma}

\begin{proof}
	Similar to the proof of \autoref{lemma:vermaModule}.
\end{proof}
Note that since $\univ(\plusminus{\tilde{\lie n}})$ is a $\univ\lie g$-module,
we must have
\begin{equation} \label{eq:piLambdaIsModuleHom}
\tilde\pi_\lambda^\pm (w \tilde\pi_\lambda^\pm (w'))
=
w \tildevermaacts \pm (w' \tildevermaacts \pm 1)
=
(ww') \tildevermaacts \pm 1
=
\tilde\pi_\lambda^\pm (ww')
\end{equation}
and
\begin{equation}
\tilde\pi_\lambda^* (\tilde\pi_\lambda^* (w)w' ) = \tilde\pi_\lambda^* (ww')
\end{equation}
for all $w, w' \in \univ\lie g$. 
In particular, this implies that the map
$\smash{\tilde\pi_\lambda^\pm \at{\univ(\plusminus{\lie n})}} \colon 
\univ(\plusminus{\lie n}) \to \univ(\plusminus{\tilde{\lie n}})$
is a $\univ \lie g$-module homomorphism 
(with respect to the module structures given by 
$\vermaacts \pm$ and $\tildevermaacts \pm$).
Indeed, for the plus case we have
\begin{equation*}
\tilde \pi_\lambda^+ (w \vermaacts + u)
=
\tilde \pi_\lambda^+ \pi_\lambda^+ (wu)
=
\tilde \pi_\lambda^+ (wu)
=
\tilde \pi_\lambda^+ (w \tilde \pi_\lambda^+ u)
=
w \tildevermaacts + \tilde \pi_\lambda^+ u
\end{equation*}
for all $w \in \univ \lie g$ and $u \in \univ(\plus{\lie n})$ and the minus case is similar.
Define
$\lie g_\lambda^{\pm} \coloneqq \bigoplus_{\alpha \in(\Delta')^\pm} \lie g^\alpha 
= \lie g_\lambda \cap \lie n^\pm$. 
Note that 
$\univ \lie g \cdot \lie g_\lambda^\pm 
= 
\lbrace w \vermaacts \pm X 
	\mid w \in \univ\lie g, X \in \lie g_\lambda^\pm \rbrace$
is a $\univ \lie g$-submodule of $\univ(\plusminus{\lie n})$.
Since $\tilde\pi_\lambda^\pm$ is a map of $\univ\lie g$-modules and vanishes on $\lie 
g_\lambda^\pm$, $\univ \lie g \cdot \lie g_\lambda^\pm$ is in its kernel.

\begin{lemma} \label{theo:utildenIsQuotient}
	The induced maps 
	$\tilde \pi^\pm_\lambda
	\colon
	\univ(\plusminus{\lie n}) / \univ \lie g \cdot \lie g_\lambda^\pm 
	\to \univ(\plusminus{\tilde{\lie n}})$ 
	are 
	isomorphisms of 
	$\univ\lie g$-modules.
\end{lemma}
\begin{proof}
	It is easy to check that the quotient map induced by the inclusion 
	$\univ(\plusminus{\tilde{\lie n}}) \to \univ(\plusminus{\lie n})$ defines an inverse.
\end{proof}
As before there are isomorphisms
$\tilde M_\lambda^* \tensor_{\univ\lie g} \tilde M_\lambda
\cong
\tilde{\mathbb C}^-_\lambda \tensor_{\univ(\minus{\tilde{\lie b}})} \univ\lie g  
\tensor_{\univ(\plus{\tilde{\lie b}})} \tilde{\mathbb C}_\lambda
\cong
\tilde{\mathbb C}^-_\lambda \tensor_{\univ(\lie g_\lambda)} \tilde{\mathbb C}_\lambda
\cong
\mathbb C$,
which we use to define the Shapovalov pairings
${\langle \argumentdot, \argumentdot \rangle^{\sim}_{\lambda}}'
\colon 
\tilde M_\lambda^* \times \tilde M_\lambda \to \mathbb C$,
$(x,y) \mapsto {\langle x , y \rangle^\sim_{\lambda}}' \coloneqq x \tensor y$
and 
\begin{equation}
\langle \argumentdot, \argumentdot \rangle^\sim_{\lambda} 
\colon 
\univ(\plus{\tilde{\lie n}}) \times \univ(\minus{\tilde{\lie n}}) 
\to \mathbb C \komma
\quad
\langle u, v\rangle^\sim_{\lambda} = 
{\langle 
1 \tensor S(u), v \tensor 1
\rangle_{\lambda}^\sim}' = 1 \tensor S(u) v \tensor 1 \punkt
\end{equation}
In the same way as in \autoref{theo:shapovalovPairing:computations}
one proves that this pairing can be computed by
\begin{equation}
\langle u, v \rangle_\lambda^\sim = \pi_\lambda(S(u)v) \punkt
\end{equation}
Note that
$\tilde\pi^-_\lambda \circ \tilde\pi_\lambda ^*
= 
\tilde\pi_\lambda^* \circ \tilde \pi_\lambda^- 
= 
\pi_\lambda^* \circ \pi_\lambda^- 
= 
\pi_\lambda$,
so there is no need to introduce a $\tilde \pi_\lambda$. 

\begin{lemma}\label{theo:pairing:restricts}
	Let $u \in \univ(\plus{\lie n})$ and $v \in \univ(\minus{\lie n})$.
	Then we have
	$\langle \tilde \pi_\lambda^+ u, \tilde \pi_\lambda v \rangle^\sim_{\lambda} 
	= \langle u,v \rangle_\lambda$. 
	In particular
	$\langle \argumentdot, \argumentdot \rangle_\lambda
	\at{\univ(\plus{\lie n}) \times \univ \lie g \cdot \lie g_\lambda^-} 
	= \langle \argumentdot, \argumentdot \rangle_\lambda
	\at{\univ \lie g \cdot \lie g_\lambda^+ \times \univ(\minus{\lie n})}
	= 0$.
\end{lemma}

\begin{proof}
	Using \eqref{eq:piLambdaIsModuleHom} twice, we compute
	\begin{multline*}
	\langle \tilde \pi_\lambda^+ u, \tilde \pi_\lambda^- v \rangle_\lambda^\sim
	=
	\pi_\lambda (S (\tilde \pi_\lambda^+ u ) \tilde \pi_\lambda^- v)
	=
	\tilde \pi_\lambda^* \circ \tilde \pi_\lambda^- (\tilde \pi_\lambda^* (Su) \tilde \pi_\lambda^- 
	v)
	=
	\tilde \pi_\lambda^* \circ \tilde \pi_\lambda^- (\tilde \pi_\lambda^* (Su) v)
	= \\ =
	\tilde \pi_\lambda^- \circ \tilde \pi_\lambda^* (\tilde \pi_\lambda^* (Su) v)
	=
	\tilde \pi_\lambda^- \circ \tilde \pi_\lambda^* (S(u) v)
	=
	\pi_\lambda(S(u)v)
	=
	\langle u, v\rangle_\lambda \punkt
	\end{multline*}
\end{proof}
Define the set
\begin{equation}
\tilde\Lambda
=
\lbrace
	\lambda \in \lie h^* \mid p_\lambda(\mu) \neq 0 \,\,\,
		\forall \mu \in \mathbb N_0 \hat\Delta^+ \setminus\lbrace 0 \rbrace 
\rbrace \punkt
\end{equation}
Furthermore, let $\tilde \wordset$ be the set of words $w \in \wordset$
such that $\alpha_{w_{i\dots \abs w}} \in \mathbb N_0 \hat \Delta^+$
for all $i = 1, \dots, \abs w$. 
Since $\tilde\pi_\lambda^+(X_w) = 0$ and $\tilde\pi_\lambda^-(Y_w) = 0$
for $w \in \wordset\setminus\tilde\wordset$,
the following theorem is not surprising. 

\begin{theorem}\label{theo:canonicalelement:general}
	Let $\lambda \in \tilde \Lambda$.
	Then the Shapovalov pairing
	$\langle\argumentdot,\argumentdot\rangle^\sim_{\lambda} \colon 			
	\univ(\plus{\tilde{\lie n}}) \times \univ(\minus{\tilde{\lie n}})
	\to \mathbb C$
	is non-degenerate and its canonical element
	$F_\lambda \in 
	\univ(\plus{\tilde{\lie n}}) \tensorhat \univ(\minus{\tilde{\lie n}})$
	is given by
	\begin{equation} \label{eq:canonicalelement:general}
	F_\lambda 
	= \sum_{w\in\tilde\wordset} p^w_{\lambda}(\alpha_w)^{-1} \tilde\pi_\lambda^+(X_w) 
	\tensor \tilde\pi_\lambda^-(Y_w)
	= \sum_{w\in\tilde\wordset} 
	\prod_{i = 1}^{\abs w} p_{\lambda}(\alpha_{w_{i \dots \abs w}})^{-1} 
	\tilde\pi_\lambda^+(X_w) \tensor \tilde\pi_\lambda^-(Y_w) \punkt
	\end{equation}
\end{theorem}

\renewcommand{\tikzcircleradius}{0.04}
\renewcommand{\tikzlinewidth}{1pt}
\newcommand{\tikzlinesep}{0.045}
\newcommand{\tikzlinesepx}{\tikzlinesep*\tikzvertspace}
\newcommand{\tikzlinesepy}{\tikzlinesep*0.5}
\definecolor{qqzzqq}{rgb}{0,0.6,0}
\definecolor{uuuuuu}{rgb}{0.27,0.27,0.27}
\definecolor{ffqqqq}{rgb}{1,0,0}

\figurewithtwoimages{
	\begin{tikzpicture}[line cap=round,line join=round,>=triangle 
	45,x=1.0cm,y=1.0cm, scale=.95]
	\clip(-2.2,-2.2) rectangle (2.2,2.2);
\draw [domain=-9.1:9.99] plot(\x,{(-0-0*\x)/1});
	\draw [domain=-9.1:9.99] plot(\x,{(-0-0.87*\x)/-0.5});
	\draw [domain=-9.1:9.99] plot(\x,{(-0--0.87*\x)/-0.5});
\fill [color=blue!20] (1.5,0.87) circle (5pt);
	\fill [color=green!30] (-1.5,0.87) circle (5pt);
	\draw [color=uuuuuu] (-1.5,0.87) circle (5pt);
	\draw [color=uuuuuu] (1.5,0.87) circle (5pt);
\draw [line width=1pt,color=blue,dashed] (0,0) -- (-1.5,-0.87);
	\draw [->,line width=1pt,color=ffqqqq] (0,0) -- (0,-1.73);
	\draw [->,line width=1pt,color=ffqqqq] (0,0) -- (1.5,-0.87);
\draw [->,line width=1pt,color=qqzzqq] (0,0) -- (-1.5,0.87);
	\draw [line width=1pt,color=blue,dashed] (0,0) -- (1.5,0.87);
	\draw [->,line width=1pt,color=qqzzqq] (0,0) -- (0,1.73);
\draw (-1.5,0.98) node[anchor=south] {$\alpha_1$};
	\draw (0,1.84) node[anchor=south] {$\alpha_2$};	
	\draw (1.5,0.98) node[anchor=south] {$\alpha_3$};
\fill [color=blue] (-0.805,1.4) circle (1.5pt);
	\draw [color=blue](-0.835,1.65) node[anchor=north east] {$\lambda$};
	\draw [color=blue](-1.6,-0.4) node[anchor=west]{$\Delta'$};
\begin{scriptsize}
	\fill [color=uuuuuu] (0,0) circle (1.5pt);
	\fill [color=uuuuuu] (0,-1.73) circle (1.5pt);
	\fill [color=uuuuuu] (-1.5,-0.87) circle (1.5pt);
	\fill [color=uuuuuu] (-1.5,0.87) circle (1.5pt);
	\fill [color=uuuuuu] (0,1.73) circle (1.5pt);
	\fill [color=uuuuuu] (1.5,0.87) circle (1.5pt);
	\fill [color=uuuuuu] (1.5,-0.87) circle (1.5pt);
	\end{scriptsize}
	\end{tikzpicture}
}{
	\begin{tikzpicture}[>=triangle 45, scale=1.45]
	\draw[color=red, line width=\tikzlinewidth] (\tikzvertspace, -1.5) -- (0, -1);
	\draw[color=red, line width=\tikzlinewidth,->] (0, -1) -- (0,-1+\tikzshorten*1.0);
	
	\draw[color=red, line width=\tikzlinewidth,->] (0, -1) -- (-\tikzshorten*\tikzvertspace, 
	-1+\tikzshorten*0.5);
	
	\draw[color=red, line width=\tikzlinewidth] (0, -1) -- (\tikzvertspace, -0.5);
	\draw[color=red, line width=\tikzlinewidth,->] (\tikzvertspace-\tikzlinesepx, 
	-0.5+\tikzlinesepy) -- 
	(\tikzvertspace-\tikzlinesepx+\tikzshorten*\tikzvertspace, 
	\tikzlinesepy-0.5+\tikzshorten*0.5);
	\draw[color=red, line width=\tikzlinewidth,->] (\tikzvertspace-\tikzlinesep, -0.5) -- 
	(\tikzvertspace-\tikzlinesep, -0.5+\tikzshorten*1.0);
	\draw[color=red, line width=\tikzlinewidth,->] (\tikzvertspace-\tikzlinesepx, 
	-0.5-\tikzlinesepy) -- 
	(-\tikzlinesepx+\tikzvertspace-\tikzshorten*\tikzvertspace, 
	-\tikzlinesepy-0.5+\tikzshorten*0.5);
	
	\draw[color=blue, line width=\tikzlinewidth] (\tikzvertspace, -1.5) -- (\tikzvertspace, -0.5);
	\draw[color=blue, line width=\tikzlinewidth,->] (\tikzvertspace+\tikzlinesepx, 
	-0.5-\tikzlinesepy) -- 
	(\tikzvertspace+\tikzlinesepx+\tikzshorten*\tikzvertspace, 
	-\tikzlinesepy-0.5+0.5*\tikzshorten);
	\draw[color=blue, line width=\tikzlinewidth,->] (\tikzvertspace+\tikzlinesep, -0.5) -- 
	(\tikzvertspace+\tikzlinesep, -0.5+\tikzshorten*1.0);
	\draw[color=blue, line width=\tikzlinewidth,->] (\tikzvertspace+\tikzlinesepx, 
	-0.5+\tikzlinesepy) -- 
	(\tikzlinesepx+\tikzvertspace-\tikzshorten*\tikzvertspace, 
	\tikzlinesepy-0.5+\tikzshorten*0.5);
	
	\draw[color=qqzzqq, line width=\tikzlinewidth] (\tikzvertspace, -1.5) -- (2*\tikzvertspace, -1);
	\draw[color=qqzzqq, line width=\tikzlinewidth] (2*\tikzvertspace, -1) -- (\tikzvertspace, -0.5);
	\draw[color=qqzzqq, line width=\tikzlinewidth,->] (2*\tikzvertspace, -1) -- 
	(2*\tikzvertspace+\tikzshorten*\tikzvertspace, -1+\tikzshorten*0.5);
	\draw[color=qqzzqq, line width=\tikzlinewidth,->] (2*\tikzvertspace, -1) -- 
	(2*\tikzvertspace+\tikzshorten*2*\tikzlinesepx, {-1+\tikzshorten*(1-2*\tikzlinesepy)});
	\draw[color=qqzzqq, line width=\tikzlinewidth,->] (\tikzvertspace, -0.5) -- 
	(\tikzvertspace+\tikzshorten*\tikzvertspace, -0.5+\tikzshorten*0.5);
	\draw[color=qqzzqq, line width=\tikzlinewidth,->] (\tikzvertspace, -0.5) -- (\tikzvertspace, 
	-0.5+\tikzshorten*1.0);
	\draw[color=qqzzqq, line width=\tikzlinewidth,->] (\tikzvertspace, -0.5) -- 
	(\tikzvertspace-\tikzshorten*\tikzvertspace, -0.5+\tikzshorten*0.5);
	
\foreach \n in {0} {
		\fill ({(-\n+0)*\tikzvertspace},-0.0-0.5*\n) circle (\tikzcircleradius);
	}
	\foreach \n in {1,2} {
		\draw ({(-\n+0)*\tikzvertspace},-0.0-0.5*\n) circle (\tikzcircleradius);
	}
	\foreach \n in {0,1} {
		\fill ({(-\n+1)*\tikzvertspace},-0.5-0.5*\n) circle (\tikzcircleradius);
	}
	\foreach \n in {2} {
		\draw ({(-\n+1)*\tikzvertspace},-0.5-0.5*\n) circle (\tikzcircleradius);
	}
	\foreach \n in {0,1,2} {
		\fill ({(-\n+2)*\tikzvertspace},-1.0-0.5*\n) circle (\tikzcircleradius);
		\fill ({(-\n+3)*\tikzvertspace},-1.5-0.5*\n) circle (\tikzcircleradius);
	}
	\node at (\tikzvertspace,-2.9) {$\vdots$};
	\node at (3.5*\tikzvertspace,-1.75) {$\ddots$};
	\node at (-2.5*\tikzvertspace,-1.25) {\reflectbox{$\ddots$}};
\node at (0,0.2) {$\lambda$};
\node at (\tikzvertspace,-1.5-0.2) {$\lambda-\mu$};
	\end{tikzpicture}
}{
	The tree $T$ used in the proof of 
	\autoref{theo:canonicalelement:general} for 
	$\lie g = \lie{sl}_3(\mathbb C)$ and $\mu = 2\alpha_1+\alpha_3$.
	Compare this with \autoref{figure:tree:proof:special}.
	Elements of the tree starting with $1$, $2$ and $3$
	are coloured red, blue and green, respectively.
	Only the weight spaces marked with filled dots are non-trivial
	(but might have a different dimension than in the case where $\Delta' = \emptyset$),
	and all weight spaces marked with circles only contain $0$.
	In particular, the weight spaces at maximal elements of the tree are trivial,
	except for $V^\lambda$.
	All non-maximal weight spaces are non-trivial.
	}{
	figure:tree:proof:general
}

\begin{proof}
	It suffices to prove that
	$\sum_{w\in\tilde\wordset} 
	p^w_{\lambda}(\alpha_w)^{-1} 
	\tilde\pi_\lambda^-(Y_w) \langle
	\tilde\pi_\lambda^+(X_w), \tilde v
	\rangle^\sim_{\lambda}
	= \tilde v$
	for all $\tilde v \in \univ(\minus{\tilde{\lie n}})$
	and that
	$\sum_{w\in\tilde\wordset} 
	p^w_{\lambda}(\alpha_w)^{-1}
	\tilde\pi_\lambda^+(X_w) \langle
	\tilde u, \tilde\pi_\lambda^-(Y_w) 
	\rangle^\sim_{\lambda}
	= \tilde u$ 
	for all $\tilde u \in \univ(\plus{\tilde{\lie n}})$
	by using an analogue of \autoref{theo:canonicalElement:FormulaToCheck}.
Let $v \in \univ(\minus{\lie n})$ be the image of $\tilde v$ under the 
	inclusion $\univ(\minus{\tilde{\lie n}}) \to \univ(\minus{\lie n})$, so 
	that $\tilde\pi_\lambda^-(v) = \tilde v$.
Assume that $v = \sum_{\mu \in \mathbb N_0 \hat\Delta^+} v_{-\mu}$
	is the weight decomposition of $v$.
Then
	\begin{align*}
	\sum_{w\in\tilde\wordset} p^w_{\lambda}(\alpha_w)^{-1} &
	\tilde \pi_\lambda^-(Y_w) \langle \tilde\pi_\lambda^+(X_w), \tilde v \rangle^\sim_{\lambda}  
	\\
	&= \sum_{w\in\tilde\wordset} p^w_{\lambda}(\alpha_w)^{-1} 
	\tilde\pi_\lambda^-(Y_w) \langle X_w, v \rangle_\lambda
	\\
	&= \tilde\pi^-_\lambda \bigg(
	\sum_{w\in\tilde\wordset} p^w_{\lambda}(\alpha_w)^{-1} 
	Y_w \langle X_w, v_{-\alpha_w} \rangle_\lambda \bigg) \\
	&= \tilde\pi^-_\lambda\bigg(
	\sum_{\mu \in \mathbb N_0 \hat\Delta^+}
	\sum_{w\in\tilde\wordset_\mu} (-1)^{\abs w}p^w_{\lambda}(\alpha_w)^{-1} 
	Y_w \acts_\lambda^- X_{w^\opp} \acts_\lambda^- v_{-\mu} \bigg) \komma
	\end{align*}
	where $\tilde W_\mu = 
	\lbrace w \in \tilde W \mid \alpha_w = \mu\rbrace$.
We claim that there is an admissible tree $T$ and
	$v' \in \univ\lie g \cdot \lie g_\lambda^-$ such that 
	\begin{equation*}
	\sum_{w\in\tilde\wordset_\mu}  
	(-1)^{\abs w} p^w_{\lambda}(\alpha_w)^{-1} 
	Y_w \acts_\lambda^- X_{w^{\mathrm{opp}}} \acts_\lambda^- v_{-\mu} 
	= v'+ \sum_{w\in\max T}  
	(-1)^{\abs w} p^w_{\lambda}(\alpha_w)^{-1} 
	Y_w \acts_\lambda^- X_{w^{\mathrm{opp}}} \acts_\lambda^- v_{-\mu} \komma
	\end{equation*}
	which would finish the proof by using \autoref{theo:casimirIterated}.	
	Indeed, let 
	\begin{equation*}
	T = \lbrace \emptyset \rbrace \cup \lbrace 
	w \in W \mid 
	\exists w' \in \tilde \wordset_\mu
	\text{ and $0 \leq i \leq \abs{w'}-1$
		such that $w_{1 \dots \abs w -1} = w'_{1 \dots i}$} 
	\rbrace
	\end{equation*}
	be the smallest tree containing $\tilde \wordset_\mu$. 
	Since 
	$\lambda \in\tilde \Lambda$, this tree is admissible. 
	Furthermore $\tilde \wordset_\mu \subseteq \max T$ and any element $w \in 
	\max T$ satisfies exactly one of the 
	following two conditions. Either $\alpha_w = 
	\mu$, so that $w \in \tilde W_\mu$ appears in the sum on the left hand 
	side of the above equation. Or $\mu - \alpha_{w} \notin \mathbb N_0 
	\hat\Delta^+$, so that $X_{w^{\mathrm{opp}}} v_{-\mu}$ would have to be of 
	weight $\alpha_w - \mu \notin -\mathbb N_0 \hat\Delta^+$ and does therefore 
	either vanish or lie in $\univ\lie g \cdot \lie g_\lambda^-$.
The statement for $\tilde u$ is proven similarly.
\end{proof}
Using the inclusions 
$\univ(\plusminus{\tilde{\lie n}}) \to \univ \lie g$
and the projection
$\univ \lie g \to \univ \lie g / \univ \lie g \cdot \lie g_\lambda$, 
we map $F_\lambda$ to 
$\smash{(\univ \lie g / \univ \lie g \cdot \lie g_\lambda)^{\hat\tensor 2}}$.
Note that, as before, $\univ \lie g \cdot \lie g_\lambda$ is a homogeneous
ideal in $\univ \lie g$, so the grading of $\univ \lie g$
stays well-defined on the quotient. 
The action of $\lie g_\lambda$ on $(\univ \lie g)^{\tensor 2}$
also passes to the quotient and extends to a continuous action
on the completed tensor product.

\begin{theorem}[{Alekseev--Lachowska \cite{alekseev.lachowska:2005a}}]
	\label{theo:alekseevLachowska:general}
	Let $\lambda \in \tilde\Lambda$.
Then 
	$F_\lambda \in (\univ\lie g/ \univ\lie g \cdot \lie g_\lambda)^{\hat\tensor 2}$
	is $\lie g_\lambda$-in\-vari\-ant and satisfies
	\begin{equation}\label{eq:associativity:general}
	((\id \tensor \Delta) F_\lambda) 1 \tensor F_\lambda
	= 
	((\Delta \tensor \id) F_\lambda) F_\lambda \tensor 1 
	\end{equation}
	in $(\univ\lie g/ \univ\lie g \cdot \lie g_\lambda)^{\hat\tensor 3}$.
\end{theorem}

\begin{proof}
	Note that the $\lie g$-invariance of the Shapovalov pairing 
	(proven similarly as in \autoref{theo:shapovalovPairing:computations}) 
	implies that
	$F_\lambda 
	\in 
	\univ( \plus{\tilde{\lie n}}) \tensorhat 
	\univ(\minus{\tilde{\lie n}})$
	is also $\lie g$-invariant.
	Then  
	$F_\lambda \in \smash{(\univ \lie g / \univ \lie g \cdot \lie g_\lambda)^{\hat\tensor 2}}$
	is $\lie g_\lambda$-invariant since the map 
	$\univ( \plus{\tilde{\lie n}}) \times 
	\univ(\minus{\tilde{\lie n}}) 
	\to
	\smash{(\univ \lie g / \univ \lie g \cdot \lie g_\lambda)^{\hat\tensor 2}}$
	is $\lie g_\lambda$-equivariant.
	Equation \eqref{eq:associativity:general} is proven in 
	\cite[Section 4]{alekseev.lachowska:2005a}.
\end{proof}
It will be convenient in the following to write $F_\lambda$
as a sum of elements that are all invariant under $\lie g_\lambda$.

\begin{lemma} \label{theo:componentsOfF}
	Let $\lambda \in \tilde \Lambda$.
	Then there is a partition of $\tilde W$ into finite subsets
	$\tilde W_\ell$, $\ell \in \mathbb N_0$ such that
	\begin{equation} \label{eq:Fhbarl}
	F_{\lambda, \ell}
	\coloneqq 
	\sum_{w\in\tilde\wordset_\ell} p^w_{\lambda}(\alpha_w)^{-1} 
	\tilde\pi_\lambda^+(X_w) \tensor \tilde\pi_\lambda^-(Y_w)
	\end{equation}
	is $\lie g_\lambda$-invariant.
\end{lemma}

\begin{proof}
	It will be convenient to introduce a different grading $d'$ on $\lie g$,
	for which $\lie g_\lambda$ is of degree $0$.
	To this end, let $\lie h$ and the root spaces 
	of simple roots in $\Delta'$ be of degree $0$,
	and let the root spaces of simple roots in $\hat\Delta$ be of degree $1$.
	Since any root is a unique linear combination of simple roots
	this assignment extends to a grading on $\lie g$.
	More explicitly, if $\sigma_1, \dots, \sigma_r \in \Delta$ are the simple roots,
	with $\sigma_1, \dots, \sigma_{r'} \in \Delta'$,
	then the root space of a root $\alpha = \sum_{i=1}^r c^i \sigma_i$ is of degree
	$d'(\alpha) = \sum_{i=r'+1}^r c^i$.
	Since $\lie g_\lambda$ is spanned by $\lie h$ and the root spaces of roots in $\Delta'$,
	and since the invariance of the ordering implies that any root in $\Delta'$
	is a linear combination of simple roots in $\Delta'$,
	it follows that every element of $\lie g_\lambda$ is homogeneous of degree $0$.
	This grading is coarser than the grading given by $d$,
	in the sense that the graded components with respect to the new grading $d'$
	are direct sums of the graded components with respect to $d$.
	The restrictions of the maps $\tilde \pi_\lambda^\pm$
	to $\univ (\plusminus{\lie n})$
	are homogeneous of degree $0$ with respect
	to (the restriction of) the $\mathbb Z$-grading on $\univ \lie g$ induced by $d'$.
	
	For $w \in \wordset$ set $d'(w) \coloneqq d'(\alpha_{w_1}) + \dots + d'(\alpha_{w_{\abs w}})$,
	and define 
	$\smash{\tilde \wordset_\ell} 
	\coloneqq 
	\lbrace w \in \smash{\tilde \wordset} \mid d'(w) = \ell \rbrace$.
	It follows from \autoref{theo:sumsInHatCannotHaveManyPrimedElements}
	that $\smash{\tilde \wordset_\ell}$ is finite for every $\ell$.
	The elements $F_{\lambda,\ell}$ defined from $\smash{\tilde \wordset}_\ell$
	as in \eqref{eq:Fhbarl} have a nice description in terms of the grading $d'$.	
	Since all graded components of $\plus{\tilde{\lie n}}$ resp.\ $\minus{\tilde{\lie n}}$
	are of degree $\geq 1$ resp.\ $\leq -1$,
	$d'$ induces a grading of  
	$\univ(\plus{\tilde{\lie n}}) \tensor \univ(\minus{\tilde{\lie n}})$
	by $\mathbb N_0 \times (-\mathbb N_0)$.
	Using the homogeneity of $\tilde \pi_\lambda^\pm$,
	it follows directly from the definition of $\smash{\tilde\wordset_\ell}$
	that $F_{\lambda, \ell}$ is precisely the component of $F_\lambda$ 
	of degree $(\ell, -\ell)$ with respect to this grading.
	Since $\lie g_\lambda$ is of degree $0$, the action of $\lie g_\lambda$ 
	on $\univ(\plus{\tilde{\lie n}}) \tensor \univ(\minus{\tilde{\lie n}})$
	preserves the graded components,
	and the $\lie g_\lambda$-invariance of $F_\lambda$ implies that 
	all the graded components $F_{\lambda,\ell}$ must also be $\lie g_\lambda$-invariant.
\end{proof} 
\subsection{The induced formal and strict products}
\label{subsec:starProductFromTwist}

In this subsection we construct associative products from the element $F_\lambda$
obtained at the end of the last subsection.
We will rescale $\lambda$ in order to introduce a parameter 
playing the role of Planck's constant in the construction.
Then we would like to use the results of \autoref{subsec:diffopsOnHomogeneousSpaces}
to obtain bidifferential operators from (the rescaled) $F_\lambda$.
However, since $F_\lambda$ is only in the completed tensor product, 
applying these results naively would give a sum of bidifferential operators
of increasing orders and we have to deal with its convergence.
There are essentially two solutions to this problem:
Firstly, we can take a formal expansion in the parameter $\hbar$,
which will give us a well-defined power series of bidifferential operators of increasing order.
Secondly, we can restrict ourselves to applying these operators
only to some polynomial functions, for which only finitely many terms of the
infinite sum give a non-zero contribution.
We discuss both approaches in detail, starting with the formal one.

Let us first introduce the rescaling. Define the set 
\begin{equation}
P_\lambda = \lbrace 0 \rbrace \cup \lbrace \hbar\in\mathbb C \setminus\lbrace 
0 \rbrace
\mid \I\lambda / \hbar \notin \tilde \Lambda\rbrace \komma
\end{equation}
and for $\hbar \in \mathbb C \setminus P_\lambda$ 
set $F_\hbar \coloneqq F_{\I \lambda / \hbar}$ 
and $F_{\hbar,\ell} \coloneqq F_{\I \lambda/\hbar, \ell}$,
where $F_{\I \lambda / \hbar}$ was computed in \autoref{theo:alekseevLachowska:general}
and $F_{\I \lambda/\hbar, \ell}$ was defined in \autoref{theo:componentsOfF}.
Note that $\lie g_{\I\lambda/\hbar} = \lie g_{\lambda}$, so
$F_\hbar \in 
\smash{((\univ \lie g / \univ \lie g \cdot \lie g_\lambda)^{\tensorhat 2})^{\lie g_\lambda}}$
holds for all $\hbar \in \mathbb C \setminus P_\lambda$.
Furthermore, the projections
$\smash{\tilde\pi_{\I\lambda/\hbar}^\pm \at{\univ(\plusminus{\lie n})}} \colon
\univ(\plusminus{\lie n}) \to \univ(\plusminus{\tilde{\lie n}})$
are independent of $\hbar$, which one can easily see from their definition in
\eqref{eq:projPiLambda}.

\begin{proposition} \label{theo:PlambdaCountable}
	Let $\lie g$ be a complex semisimple Lie algebra, 
	$\lie h$ a Cartan subalgebra of $\lie g$,
	and $\lambda \in \lie h^*$.
	Fix an invariant ordering on $\Delta$,
	and assume that $(\lambda,\mu) \neq 0$ 
	for all $\mu \in \mathbb N_0 \hat\Delta^+$ 
	satisfying $\frac 1 2 (\mu,\mu) = (\rho,\mu)$. 
	Then the set $P_\lambda$ is countable and accumulates only at zero.	
\end{proposition}

\begin{proof}
	From the definition of $P_\lambda$ we obtain
	\begin{equation*}
	P_\lambda = \lbrace 0 \rbrace \cup \lbrace \hbar \in \mathbb C 
	\setminus\lbrace 0 \rbrace \mid p_{\I \lambda / \hbar}(\mu) = 
	0 
	\text{ for some $\mu \in \mathbb N_0 \hat \Delta^+ \setminus \lbrace 0 
		\rbrace$} \rbrace \punkt
	\end{equation*}
	Under our assumptions the function
	$\hbar \mapsto p_{\I \lambda / \hbar} (\mu) 
	= \frac 1 2 (\mu, \mu) - (\rho, \mu) - \frac \I \hbar (\lambda, \mu)$
	has the only root $\I (\lambda,\mu) / (\frac 1 2 (\mu,\mu) - (\rho,\mu))$ 
	if $\frac 1 2 (\mu, \mu) - (\rho, \mu) \neq 0$ and no root otherwise. 
	Therefore $P_\lambda$ is countable since $\mathbb N_0 \hat\Delta^+ 
	\setminus \lbrace 0 \rbrace$ is countable.
	Furthermore, $P_\lambda$ accumulates only at zero since
	\begin{equation*}
	\left|\frac{\I(\lambda, \mu)}{\frac 1 2 (\mu,\mu) - (\rho,\mu)}\right|
	\leq
	\frac{\norm\lambda \norm\mu}{\frac 1 2 \norm\mu^2 - \norm\mu \norm\rho} 
	= 
	\frac{\norm\lambda}{\frac 1 2 \norm \mu - \norm \rho} 
	\end{equation*}
	if $\norm\mu > 2 \norm \rho$. Note that there are only finitely many 
	elements $\mu \in \mathbb N_0 \hat\Delta^+$ with $\norm \mu \leq 2 \norm 
	\rho$.
\end{proof}
\begin{remark} \label{remark:whatPlambdaCountableMeans}
If the ordering in the previous proposition is standard,
then any element $\mu \in \mathbb N_0 \hat\Delta^+$
automatically satisfies $(\lambda, \mu) \neq 0$:
For all $\alpha \in \hat\Delta^+$ we have $(\lambda, \alpha) \in S$
and since $S$ is closed under addition this implies $(\lambda, \mu) \in S$
for all $\mu \in \mathbb N_0 \hat\Delta^+$. Note that $0 \notin S$, so in particular
$(\lambda , \mu) \neq 0$.

Note also that $\frac 1 2 (\mu, \mu) = (\rho, \mu)$
implies $\norm \mu \leq 2 \norm \rho$,
so there can only be finitely many elements $\mu \in \mathbb N_0 \Delta$ 
satisfiying $\frac 1 2 (\mu, \mu) = (\rho, \mu)$.
Among those are all simple roots and the element $2 \rho$.
However, simple roots which are in $\smash{\mathbb N_0 \hat\Delta}$ 
are by definition not orthogonal to $\lambda$.
An example of an element that is not a simple root and not $2 \rho$
in the case of $\lie g = \liesl[3]$ with root system as in 
\autoref{figure:roots:verma:tree} is $\mu = \alpha_1 + \alpha_2$.
\end{remark}
We say that $F_\hbar$ depends rationally on $\hbar$
if all the $F_{\hbar,\ell}$ depend rationally on $\hbar$.
This makes sense since $F_{\hbar, \ell}$
takes values in a finite dimensional subspace of 
$(\univ \lie g / \univ \lie g \cdot \lie g_\lambda)^{\tensor 2}$
that is independent of $\hbar$.

\begin{theorem}[Alekseev--Lachowska 
	\cite{alekseev.lachowska:2005a}]\label{theo:starproduct:alekseev}
	Let $\lambda \in \lie h^*$ and assume that $P_\lambda$ is countable.
	Then $F_\hbar$ depends rationally on $\hbar$, with no pole at zero.
	In particular, the Taylor series expansion of $F_\hbar$ around $0$ makes sense,
	and it gives an element
	$F \in (\univ \lie g / \univ \lie g \cdot \lie g_\lambda)^{\tensor 2} \formal\formParam$,
	where the tensor product is the usual (not completed) tensor product.
	Furthermore, $F$ satisfies \eqref{eq:associativity:general}
	in $(\univ \lie g / \univ \lie g \cdot \lie g_\lambda)^{\tensor 3}\formal\formParam$
	and is $\lie g_\lambda$-invariant.
\end{theorem}

\begin{proof}
	As mentioned before, $\lie g_{\I \lambda / \hbar}$ and 
	$\smash{\tilde\pi_{\I\lambda/\hbar}^\pm \at{\univ(\plusminus{\lie n})}} \colon
	\univ(\plusminus{\lie n}) \to \univ(\plusminus{\tilde{\lie n}})$
	are independent of $\hbar$, so only the coefficients 
	$\smash{p_{\I \lambda/\hbar}^w(\alpha_w)^{-1}}$ in the formula for $F_{\I\lambda/\hbar}$ 
	obtained in \autoref{theo:canonicalelement:general} depend on $\hbar$.
	Since they are products of elements of the form
	\begin{equation*}
	p_{\I \lambda/\hbar}(\mu)^{-1} 
	= 
	\left(\frac 1 2 (\mu, \mu) - (\rho, \mu) 
	- \left(\frac{\I \lambda}{\hbar}, \mu\right)\right)^{-1} 
	= \frac{\hbar}{\left(\frac 1 2 (\mu, \mu) - (\rho, \mu)\right)\hbar - (\I \lambda, \mu)}
	\end{equation*}
	with $\mu \in \mathbb N_0 \hat\Delta^+ \setminus \lbrace 0 \rbrace$,
	their dependence on $\hbar$ is rational without a pole at zero.
	(Observe that $\frac 1 2 (\mu,\mu) - (\rho,\mu)$ and $(\I\lambda, \mu)$ cannot
	vanish simultaneously since $P_\lambda$ is assumed to be countable.) 
	Consequently, we may take the Taylor expansion of $F_{\I \lambda/\hbar}$
	around $\hbar = 0$. To see that this yields an element in the usual tensor product,
	note that the formal expansion of $\smash{p_{\I \lambda/\hbar}(\mu)^{-1}}$ 
	is a multiple of $\formParam$ unless 
	$(\lambda, \mu) = 0$.
	Now
	$\smash{p_{\I\lambda/\hbar}^w(\alpha_w)^{-1}} 
	= 
	\smash{\prod_{i=1}^{\abs w} p_{\I \lambda/\hbar} (\alpha_{w_{i \dots \abs w}})^{-1}}$,
	and if the formal expansions of both 
	$p_{\I \lambda/\hbar} (\alpha_{w_{i \dots \abs w}})^{-1}$ and 
	$p_{\I \lambda/\hbar} (\alpha_{w_{i+1 \dots \abs w}})^{-1}$ 
	are not multiples of $\hbar$, then $(\lambda, \alpha_{w_i}) = 0$,
	i.e.\ $\alpha_{w_i} \in \Delta'$.
	However, \autoref{theo:sumsInHatCannotHaveManyPrimedElements}
	ensures that this cannot happen too often:
	If $M$ is the constant obtained in that lemma, 
	then at least $\lceil\abs w / (M+1)\rceil$ many elements 
	in the formal expansion of $p_{\I \lambda/\hbar}^w(\alpha_w)^{-1}$
	are multiples of $\formParam$,
	so this expansion is of order 
	at least $\smash{\formParam^{\lceil \abs w / (M+1) \rceil}}$.
	Consequently, only finitely many words contribute to a given order in $\formParam$,
	so that we do not need to complete the tensor product.
	Since every $F_\hbar$ satisfies \eqref{eq:associativity:general} 
	and is $\lie g_\lambda$-invariant,
	this is also true for the formal expansion $F$.
\end{proof}
Let us now apply this theorem to quantize complex coadjoint orbits.
Let $G$ be a complex connected semisimple Lie group with coadjoint
orbit $\orb\lambda$ through a semisimple element
$\lambda \in \lie g^*$. Pick a Cartan subalgebra $\lie h$
containing $\lambda^\sharp$. 
Choose an invariant ordering for which $P_\lambda$ is countable
(e.g.\ a standard invariant ordering).

By \autoref{theo:stabilizerConnected} we know that $G_\lambda$ is connected.
Therefore the $\lie g_\lambda$-invariance of the elements $F$ and $F_\hbar$ 
constructed previously implies their $G_\lambda$-invariance. 
Consequently we can apply the results of 
\autoref{subsec:diffopsOnHomogeneousSpaces} in order
to obtain holomorphic $G$-invariant bidifferential operators on 
$\orb\lambda \cong G / G_\lambda$.
Define the formal product
	\begin{equation} 
	\label{eq:defStarProduct:Formal:Extended}
	\star \colon 
	\Cinfty(\orb\lambda)\formal\formParam \times 
	\Cinfty(\orb\lambda)\formal\formParam
	\to
	\Cinfty(\orb\lambda)\formal\formParam
	\komma  \quad
	(f, g) \mapsto f \star g \coloneqq \Psi(F)(f, g) \komma
	\end{equation} 
which is well-defined since the previous theorem asserts that
$F \in (\univ\lie g / \univ \lie g \cdot \lie g_\lambda)^{\tensor 2} \formal\formParam$.

\begin{proposition}\label{theo:starProduct:associative}
	The product $\star$ is associative and restricts to a product
	\begin{equation}
	\star \colon 
	\Hol(\orb\lambda)\formal\formParam \times 
	\Hol(\orb\lambda)\formal\formParam
	\to
	\Hol(\orb\lambda)\formal\formParam
	\end{equation}
	on formal power series of holomorphic functions.
	Moreover, $\star$ is $G$-in\-vari\-ant, in the sense that
	$(g \acts f_1) \star (g \acts f_2) = g \acts (f_1 \star f_2)$ holds for all $g \in G$
	and $f_1, f_2 \in \Cinfty(\orb\lambda)\formal\formParam$.
\end{proposition}
\begin{proof}
	It is a standard argument that the twist condition \eqref{eq:associativity:general}
	translates into associativity of the induced product. That $\star$ restricts
	to power series of holomorphic functions and is $G$-invariant
	is immediate since the image of $\Psi$
	consists of holomorphic $G$-invariant bidifferential operators.
\end{proof}
In order to define strict star products from $F_\hbar$ directly,
i.e.\ without taking a formal power series expansion, 
we need to ensure that $\Psi(F_\hbar)$ is well-defined.
To do that we introduce polynomials on the coadjoint orbit.
It will turn out that only finitely many elements of the infinite sum
defining $F_\hbar$ contribute non-trivially
when $\Psi(F_\hbar)$ is applied to polynomials.

Recall from \autoref{subsec:generalities} that we may assume without loss of
generality that $G$ is a closed complex Lie subgroup of $\GL[N]$.
We fix a way to realize $G$ as such a matrix Lie group once and for all.
In particular, the Lie algebra $\lie g$ of $G$ is realized 
as a complex Lie subalgebra of $\liegl[N]$.

\begin{definition}[Polynomials on $\orb\lambda$] 
	\label{def:polynomials:orbit}
	Let $\orb\lambda \subseteq \lie g^*$ be a complex coadjoint orbit.
	Then 
	\begin{equation}
	\Pol(\orb\lambda) 
	= \lbrace 
		p \colon \orb\lambda \to \mathbb C 
		\mid 
		p = P \at {\orb\lambda}\text{ for some holomorphic polynomial $P$ on $\lie g^*$}
	\rbrace
	\end{equation}
	is called the algebra of polynomials on $\orb\lambda$.
\end{definition}
Recall that the symmetric algebra $\Sym \lie g$ of $\lie g$ is isomorphic (as an algebra)
to the algebra $\Pol(\lie g^*)$ of polynomials on $\lie g^*$.
The isomorphism sends an element $X_1 \vee \dots \vee X_j \in \Sym^j \lie g$
to $\xi \mapsto \xi(X_1) \dots \xi(X_j)$.

\begin{definition}[Polynomials on $G$]
	\label{def:polynomials:lieGroup}
	For a complex linear Lie group $G$,
	the algebra of polynomials $\Pol(G)$ is the unital
	complex	subalgebra of $\Cinfty(G)$ generated by the functions 
	$P_{ij} \colon G \to \mathbb C$, $g \mapsto g_{ij}$.
\end{definition}
Polynomials on a complex Lie group $G$ are holomorphic.
In the case of semisimple connected Lie groups
both the Lie group itself and the coadjoint orbit are affine algebraic varieties,
see \autoref{theo:coadjointOrbitsAreClosed},
and our definition of polynomials coincides with
the definition of regular functions on algebraic varieties. 
If $G$ is connected and semisimple, then the definition of polynomials on $G$
is independent of the way in which $G$ is realized as a linear group,
which can be proven as outlined in \appendixref{appendix:gfinites}.

\begin{proposition}\label{theo:pullbackOfPolynomials}
	Assume that the complex linear Lie group $G$ is semisimple and connected. Then 
	$\pi^* \colon \Hol(\orb\lambda) \cong \Hol(G/G_\lambda) \to \Hol(G)^{G_\lambda}$ 
	restricts to an isomorphism
	$\pi^* \colon \Pol(\orb\lambda) \to \Pol(G)^{G_\lambda}$.
\end{proposition}

\begin{proof}
	Since the Lie algebra $\lie g$ is semisimple,
	we have $\lie g = [\lie g, \lie g]$, i.e.\ every element
	of $\lie g$ can be written as a sum of commutators.
	Consequently the trace of any element of $\lie g$ is zero.
	Therefore any element in a sufficiently small neighbourhood
	of the identity of $G$ must have determinant $1$, and consequently
	$G$ is a Lie subgroup of $\SL[N]$.
	
	Let $E_{ij} \in \liegl[N]$ be the matrix 
	that is $1$ at position $(i,j)$ and $0$ otherwise. 
Extend $\lambda$ to a linear functional $\tilde\lambda \in \liegl[N]^*$.
For an element $X \in \lie g = \Sym^1 \lie g$, which we identify
	with a polynomial on $\lie g^*$, we compute
	\begin{multline*}
	\pi^* \big(X \at{\orb\lambda}\big)(g) 
	= X \at{\orb\lambda}(\pi(g)) 
	= X \at{\orb\lambda}(\Ad_g^* \lambda)
	= X \at{\orb\lambda} (\lambda(g^{-1} \cdot g))
	= \\ = \lambda(g^{-1} X g)
	= \sum_{i,j} \tilde\lambda((g^{-1} X g)_{ij} E_{ij})
	= \sum_{i,j,k,\ell} (g^{-1})_{ik} g_{\ell j} X_{k\ell} \tilde\lambda(E_{ij}) \punkt
	\end{multline*}
	Since $\det g=1$ we can write $(g^{-1})_{i k}$
	as a polynomial in the entries of $g$,
	so that $\smash{\pi^* \big(X|_{\orb\lambda}\big)}$ itself is a polynomial
	in the entries of $g$.
Since $\Pol(\orb\lambda)$ is generated by $\smash{X |_{\orb\lambda}}$
	and $\pi^*$ is an algebra homomorphism,
	it follows that $\pi^* p \in \Pol(G)$ for any $p \in \Pol(\orb\lambda)$.
Injectivity of $\pi^*$ is immediate.
Surjectivity is harder to prove. One can either use methods
	from algebraic geometry
	(making use of \autoref{theo:coadjointOrbitsAreClosed},
	see for example \cite[Chapter 12]{humphreys:1995a})
	or work in a more differential geometric setting using
	$G$-finite functions as outlined in \appendixref{appendix:gfinites}.
\end{proof}
Recall the degree $d'$ introduced in the proof of \autoref{theo:componentsOfF}.

\begin{lemma} \label{theo:starFiniteOnPolys}
	For any polynomial $p \in \Pol(\GL[N])$,
	there is a constant $N_p \in \mathbb N$ such that
	$\leftinvholo u p = \leftinvholo v p = 0$ holds
	for any $u \in \univ(\plus{\tilde{\lie n}}) \subseteq \univ(\liegl[N])$
	of degree $d'$ greater than $N_p$ and any 
	$v \in \univ(\minus{\tilde{\lie n}}) \subseteq \univ(\liegl[N])$
	of degree $d'$ smaller than $-N_p$.
\end{lemma}

\begin{proof}
	Using the Leibniz rule we may assume that $p = P_{k\ell}$
	in the notation of \autoref{def:polynomials:lieGroup}.
	Let $E_{ij} \in \liegl[N]$ 
	be the matrix that is $1$ at position $(i,j)$ and $0$ otherwise. 
	It is easy to check that 
	$\leftinv{E_{ij}} P_{k\ell} = \delta_{j\ell}P_{ki}$ 
	and therefore 
	$\leftinv X P_{k\ell} 
	= 
	\smash{\leftinv{\big(\sum_{i,j} X_{ij} E_{ij} \big)} P_{k\ell}}
	=
	\sum_{i} X_{i \ell} P_{ki}$
	for all $X \in \liegl[N]$. 
	Since $P_{k \ell}$ is holomorphic, this implies that also
	$\smash{\leftinvholo X P_{k\ell}} 
	= 
	\smash{\leftinv X P_{k \ell}}
	=
	\sum_{i} X_{i \ell} P_{ki}$.
	Consequently, if $u = u_1 \dots u_M \in \univ(\liegl[N])$ with $u_1, \dots, u_M \in 
	\liegl[N]$, then 
	\begin{multline*}
	\leftinvholo u P_{k\ell} 
	= \sum_{i_M} \leftinvholo{(u_1 \dots u_{M-1})} (u_M)_{i_M \ell} P_{ki_M}
	= \\ 
	= \sum_{i_{M-1},i_M}\leftinvholo{(u_1 \dots u_{M-2})} (u_{M-1})_{i_{M-1} i_M} (u_M)_{i_M \ell} 
	P_{ki_{M-1}} 
	= \dots = \\
	= \sum_{i_1, \dots, i_M} (u_1)_{i_1 i_2} \dots (u_{M-1})_{i_{M-1}i_M} (u_M)_{i_M \ell} P_{ki_1}
	= \sum_i (u_1 \dots u_M)_{i \ell} P_{ki} \punkt
	\end{multline*} 
	Since $\ad_X$ is nilpotent for any $X \in \plus{\tilde{\lie n}}$ it follows 
	that $0 = (\ad X)_s = \ad (X_s)$ for $X \in \plus{\tilde{\lie n}}$, 
	where the index $s$ stands for the semisimple part of the Jordan 
	decomposition.
	Since $\lie g$ is semisimple this implies $X_s = 0$,
	so every $X \in \plus{\tilde{\lie n}}$ is realized by a nilpotent matrix. 
	It follows from Engel's theorem that any matrix Lie algebra
	consisting of nilpotent matrices is nilpotent as an algebra,
	so there exists a constant $M \in \mathbb N$ such that products
	of $M$ or more elements of $\plus{\tilde{\lie n}}$ vanish.
	Therefore, if $u$ is a product of at least $M$ elements of 
	$\plus{\tilde{\lie n}}$
	the above calculation shows that $\leftinv u P_{k \ell} = 0$.
If $M'$ is an upper bound for the degree $d'$ of elements of $\plus{\tilde{\lie n}}$
	then we can set $N_{P_{k\ell}} \coloneqq M M'$.
	It is easy to check that this constant also works for $\minus{\tilde{\lie n}}$.
\end{proof}

\begin{corollary} \label{theo:starWellDefOnPolynomials}
	For all $p, q \in \Pol(\orb\lambda)$ and all $\hbar \in \mathbb C \setminus P_\lambda$,
	the sum $\sum_{\ell = 0}^\infty \Psi(F_{\hbar,\ell})(p,q)$ is finite,
	and $\sum_{\ell = 0}^\infty \Psi(F_{\hbar,\ell})(p,q) \in \Pol(\orb\lambda)$.
\end{corollary}

\begin{proof}
    \autoref{theo:pullbackOfPolynomials} implies that 
    $\pi^* p$ and $\pi^* q$ are polynomials.
    By \autoref{theo:componentsOfF} the components $F_{\hbar,\ell}$ are of degree $(\ell, -\ell)$,
	and then the previous lemma implies that only finitely many summands of
	$\sum_{\ell=0}^\infty \leftinvholo F_{\hbar,\ell} (\pi^* p, \pi^* q)$
	are non-zero.
	Its proof shows that 
	$\smash{\sum_{\ell=0}^\infty \leftinvholo F_{\hbar, \ell} (\pi^* p, \pi^* q)}$ 
	is again a polynomial.
	The components $F_{\hbar,\ell}$ are $\lie g_\lambda$-invariant, and therefore,
	since $G_\lambda$ is connected by \autoref{theo:stabilizerConnected},
	also $G_\lambda$-invariant.
	Consequently 
	$\sum_{\ell=0}^\infty \leftinvholo F_{\hbar,\ell} (\pi^* p, \pi^* q)$ 
	is $G_\lambda$-invariant by \autoref{theo:diffops:Hinvariance}.
	Finally, 	
	$\sum_{\ell=0}^\infty \Psi(F_{\hbar,\ell})(p,q) 
	= 
	\sum_{\ell=0}^\infty \pi_*({\leftinvholo F_{\hbar,\ell}}^{\,\!} (\pi^* p, \pi^* q))$
	is a polynomial by \autoref{theo:pullbackOfPolynomials}.	
\end{proof}
\begin{corollary} \label{theo:starProduct:strict:associative}
	Let $\orb\lambda$ be a semisimple coadjoint orbit of a complex connected
	semisimple Lie group $G$ with Lie algebra $\lie g$. Assume that $\lie h$
	is a Cartan subalgebra of $\lie g$ containing $\lambda^\sharp$,
	and that one has chosen an invariant ordering.
	Then for any $\hbar \in \mathbb C \setminus P_\lambda$, 
	\begin{equation}
	*_\hbar \colon \Pol(\orb\lambda) \times \Pol(\orb\lambda) \to \Pol(\orb\lambda) 
	\komma \quad
	(p, q) \mapsto p *_\hbar q \coloneqq \sum_{\ell = 0}^\infty \Psi(F_{\hbar,\ell})(p, q)
	\label{eq:defStarProduct:Strict:Restricted}
	\end{equation}
	defines an associative and $G$-invariant product
	(where $G$-invariant means that
	$(g \acts p) *_\hbar (g \acts q) 
	= 
	g \acts (p *_\hbar q)$
	holds for any $g \in G$ and $p, q \in \Pol(\orb\lambda)$).
	For $p, q \in \Pol(\orb\lambda)$, $p *_\hbar q$ depends rationally on $\hbar$,
	and the formal expansion of $*_\hbar$ around $\hbar = 0$
	coincides with the formal product $\star$.
\end{corollary}

\begin{proof}
As in the formal case, it is a standard argument to show that 
\eqref{eq:associativity:general} implies the associativity of $*_\hbar$.
Since the codomain of $\Psi$ consists of $G$-invariant bidifferential operators,
it is clear that $*_\hbar$ is $G$-invariant.
Since the dependence of $F_\hbar$ on $\hbar$ is rational without pole at $0$,
it follows that $*_\hbar$ also depends rationally on $\hbar$ without pole at $0$,
and since $\star$ was constructed from the formal expansion of $F_\hbar$,
it coincides with the formal expansion of $*_\hbar$. 
\end{proof}

\begin{remark} \label{remark:leaveOutProjections}
	When considering $\Psi(F_{\hbar,\ell})$, 
	we may leave out the projections $\tilde \pi_\lambda^\pm$
	in the formula for $F_{\hbar,\ell}$ from \autoref{theo:componentsOfF}
	to obtain the same result.
	Indeed, by \autoref{theo:utildenIsQuotient} the difference of $F_{\hbar,\ell}$ and 
	\begin{equation} \label{eq:fhbarprime}
	F'_{\hbar,\ell}
	\coloneqq
	\sum_{w\in\tilde\wordset_\ell} p^w_{\I\lambda/\hbar}(\alpha_w)^{-1} X_w \tensor Y_w
	\in \univ(\plus{\lie n}) \tensor \univ(\minus{\lie n})
	\end{equation} 
	is an element in the ideal 
	$\univ \lie g \cdot \lie g_\lambda \tensor \univ \lie g 
	+ \univ \lie g \tensor \univ \lie g \cdot \lie g_\lambda$
	and therefore contained in the kernel of $\Psi$ by \autoref{theo:diffops:Hinvariance}.
\end{remark}
Recall that we obtained a condition for $P_\lambda$ being countable 
in \autoref{theo:PlambdaCountable}, and that this condition is satisfied in particular
when the ordering is standard, see \autoref{remark:whatPlambdaCountableMeans}.

\begin{proposition}
	\label{theo:firstOrderIsKKS}
	Assume that $P_\lambda$ is countable.
	Then the first order commutator of $\star$ coincides with the 
	Poisson bracket induced by the KKS form
	$\omega_{\mathrm{KKS}}$ defined in \eqref{eq:KKSform}.
\end{proposition}

\begin{proof}
	Note that the formal expansion of 
	\begin{equation*}
	p_{\I\lambda/\hbar}(\mu)^{-1} 
	= \left(
	\frac 1 2 (\mu,\mu) - (\rho,\mu) - \frac \I \hbar (\lambda, \mu)
	\right)^{-1} 
	= \I \hbar \left(
	\frac{\I\hbar} 2(\mu, \mu) - \I\hbar (\rho, \mu) + (\lambda, \mu)
	\right)^{-1}
	\end{equation*}
	is of order $\formParam$ if $(\lambda, \mu) \neq 0$.
	It follows from \autoref{theo:canonicalelement:general}
	that the element $F$ is the formal expansion of
	\begin{align*}
	\sum_{\substack{w\in\tilde\wordset\\\abs w \leq 1}}
		p^w_{\I\lambda/\hbar}(\alpha_w)^{-1} 
		\tilde\pi_{\lambda}^+(X_w) 
		\tensor
		\tilde\pi_{\lambda}^-(Y_w) 
	+
	\sum_{\substack{w\in\tilde\wordset\\\abs w \geq 2}} 
		p^w_{\I\lambda/\hbar}(\alpha_w)^{-1}
		\tilde\pi_{\lambda}^+(X_w) 
		\tensor 
		\tilde\pi_{\lambda}^-(Y_w)
\punkt
	\end{align*}
	Using that the words $w \in \tilde \wordset$ with $\abs w \leq 1$ 
	are precisely the empty word and the one-letter words $(\ell)$ 
	with $\alpha_\ell \in \hat\Delta^+$,
	i.e.\ $(\lambda, \alpha_\ell) \neq 0$,
the first sum expands to 
	$1 + \I \formParam \sum_{\alpha \in \hat\Delta^+}
		(\lambda, \alpha)^{-1} X_\alpha \tensor Y_\alpha + \mathcal O(\formParam^2)$.
	Let us argue why the formal expansion of the second sum is of order $\formParam^2$. 
	By definition
	$\smash{p_{\I\lambda/\hbar}^w(\alpha_w)^{-1}} 
	= 
	\smash{\prod_{i = 1}^{\abs w} p_{\I\lambda/\hbar}(\alpha_{w_{i \dots \abs w}})^{-1}}$.
	Since, by definition of $\smash{\tilde\wordset}$, 
	we have $\smash{\alpha_{w_{\abs w}}} \in \hat\Delta^+$, 
	it is clear that the formal expansions of all summands with 
	$(\lambda, \smash{\alpha_{w_{\abs w - 1}}} + \smash{\alpha_{w_{\abs w}}}) \neq 0$
	are of order $\formParam^2$ 
	(because both $p_{\I\lambda/\hbar}(\smash{\alpha_{(w_{\abs w-1}, w_{\abs w})}})^{-1}$ and
	$p_{\I\lambda/\hbar}(\smash{\alpha_{w_{\abs w}}})^{-1}$ are of order $\formParam$).
	So assume $(\lambda, \alpha_{w_{\abs w - 1}} + \alpha_{w_{\abs w}}) = 0$,
	in which case $\alpha_{w_{\abs w - 1}} \in \smash{\hat\Delta^+}$ and, 
	by invariance of the ordering,
	$\alpha_{w_{\abs w - 1}} + \alpha_{w_{\abs w}}$ is not a root. 
	Therefore $X_{w_{\abs w-1}} X_{w_{\abs w}} = X_{w_{\abs w}} X_{w_{\abs w-1}}$, 
	and if $w' = (w_1, \dots, w_{\abs w-2}, w_{\abs w}, w_{\abs w-1})$ 
	is the word obtained form $w$ by switching the last two letters
	then $X_w = X_{w'}$.
	Similarly $Y_w = Y_{w'}$. Furthermore, by definition of $\alpha_w$,
	we have $\smash{\alpha_{w_{i \dots \abs w}}} = \smash{\alpha_{w'_{i \dots \abs{w'}}}}$ 
	for all $i < \abs w$ and
	\begin{equation*}
	p^w_{\I\lambda/\hbar}(\alpha_w)^{-1} + p^{w'}_{\I\lambda/\hbar}(\alpha_{w'})^{-1} 
	=
\left(
		p_{\I\lambda/\hbar}(\alpha_{w_{\abs w}})^{-1} 
		+ p_{\I\lambda/\hbar}(\alpha_{w_{\abs w-1}})^{-1}
	\right)
\prod_{i = 1}^{\abs w-1} p_{\I\lambda/\hbar}(\alpha_{w_{i \dots \abs w}})^{-1} 
	\punkt
	\end{equation*}
	But under our assumptions 
	$(\alpha_{w_{\abs w}}, \lambda)^{-1} + 
	(\alpha_{w_{\abs w-1}}, \lambda)^{-1} = 0$,
	and therefore the formal expansion of 
	$p_{\I\lambda/\hbar}(\alpha_{w_{\abs w}})^{-1} 
	+ p_{\I\lambda/\hbar}(\alpha_{w_{\abs w-1}})^{-1}$
	is $\I \formParam (\alpha_{w_{\abs w}}, \lambda)^{-1} 
	+ \I \formParam (\alpha_{w_{\abs w-1}}, \lambda)^{-1} 
	+ \mathcal O(\formParam^2) 
	=
	\mathcal O(\formParam^2)$.
	Consequently, the summands which could potentially be of order $\formParam$ 
	in the sum over $w \in \tilde\wordset$ with $\abs w \geq 2$ cancel out,
	and this sum is therefore of order $\formParam^2$ as claimed.
	
	To conclude the proof, note that antisymmetrizing the first order gives indeed
	\begin{equation*}
	F_{(1)}^{\mathrm{antisym}}
	= \I\sum_{\alpha\in\hat\Delta^+} \lambda(\alpha^\sharp)^{-1} 
	(X_\alpha \tensor Y_\alpha - Y_\alpha \tensor X_\alpha)
	= \I\sum_{\alpha\in\hat\Delta}  \lambda([X_\alpha, Y_\alpha])^{-1}
	X_\alpha \tensor Y_\alpha
	= \I \pi_{\mathrm{KKS}}\komma
	\end{equation*}
	where $\pi_{\mathrm{KKS}}$ denotes the Poisson tensor
	associated to the KKS symplectic form.
\end{proof}
We conclude this subsection by saying a bit more about the directions
in which $\star$ and $*_\hbar$ differentiate.
\begin{lemma} \label{theo:distributionSpacesWellDefd}
	For any $\xi = \Ad_g^* \lambda \in \orb\lambda$, the subspaces
	\begin{subequations}
		\begin{align}
		L_{+,\xi} &=
		\spann \left\lbrace (\Ad_g X_\alpha)_{\orb\lambda}\at \xi\komma\,\, 
		\alpha 
		\in \hat\Delta^+\right\rbrace \subseteq 
		\Tangent_\xi \orb\lambda\komma\\
		L_{-,\xi} &=
		\spann \left\lbrace (\Ad_g X_\alpha)_{\orb\lambda}\at \xi\komma\,\, 
		\alpha 
		\in \hat\Delta^-\right\rbrace \subseteq 
		\Tangent_\xi \orb\lambda
		\end{align}
	\end{subequations} are
	independent of the choice of $g \in G$.
\end{lemma}
\begin{proof}
	Any two choices $g, g' \in G$ differ by an element of $G_\lambda$, 
	that is $g' = g x$ with $x \in G_\lambda$.
So it suffices to prove that
	$\spann \lbrace \Ad_x X_\alpha\komma\,\, \alpha \in \hat\Delta^\pm\rbrace
	= \spann \lbrace X_\alpha\komma\,\, \alpha \in \hat\Delta^\pm\rbrace$.
	This follows from the invariance of the ordering
	and the connectedness of $G_\lambda$.
\end{proof}
Therefore the distributions $L_+$ and $L_-$ in $\Tangent \orb\lambda$
spanned by $L_{+,\xi}$ and $L_{-,\xi}$, respectively, are well-defined.

\begin{corollary} \label{theo:distributions:complex}
	The star product $*_\hbar$ derives the first argument only in the directions
	of $L_+^{(1,0)}$ and the second argument only in the directions of $\smash{L_-^{(1,0)}}$.
\end{corollary}

\begin{proof}
	This follows from the 
	formula for $F_\hbar$ 
	obtained in \autoref{theo:canonicalelement:general},
	from \autoref{remark:leaveOutProjections},
	and from \autoref{theo:calculationOfLeftInvDiffOpsOnOrbit}.
\end{proof} 
\subsection{Examples}
\label{subsec:examples}

In this subsection we derive formulas for $F_\hbar$ in the case 
$G = \SL[1+n]$ for the largest non-trivial stabilizer $G_\lambda$.
When restricting to real coadjoint orbits in \autoref{subsec:realExamples},
this example allows us to obtain quantizations of
complex projective spaces and hyperbolic discs.

\begin{example}[{{$\SL[1+n]$}}]\label{ex:SLC}
	Let $G = \SL[1+n]$ be the Lie group of matrices with determinant $1$.
	Its Lie algebra $\lie g = \liesl[1+n]$ consists of matrices with trace $0$.
	Number the rows and columns of a matrix $X \in \lie g$ by $0, \dots, n$.
	Let $\lambda \colon \lie g \to \mathbb C$, 
	$X \mapsto -\I r X_{0,0}$ where $r \in \mathbb C$.
Using that the Killing form $B$ satisfies $B(X,Y) = 2(n+1) \trace(XY)$,
	where $\trace$ is the usual (not normalized) matrix trace,
	it follows that $\lambda^\sharp$ is a multiple 
	of the diagonal matrix $\diag(n,-1, \dots, -1)$,
	and therefore
	\begin{subequations}
	\begin{align}
	\lie g_\lambda &= \lbrace X \in \lie{sl}_{1+n}(\mathbb C) \mid 
	X_{0,i} = X_{i,0} = 0 \text{ for $1 \leq i \leq n$}\rbrace \komma \\
	G_\lambda &= \lbrace g \in \group{SL}_{1+n}(\mathbb C) \mid 
	g_{0,i} = g_{i,0} = 0 \text{ for $1 \leq i \leq n$}\rbrace \punkt
	\label{eq:stab:SL}
	\end{align}
	\end{subequations}
	We choose the Cartan subalgebra $\lie h$
	consisting of the diagonal matrices in $\lie g$.
	The roots are then given by $\alpha_{i,j} = L_i - L_j$
	for $0 \leq i,j \leq n$ with $i \neq j$,
	where $L_{i} \in \lie h^*$, $L_{i}(X) = X_{i,i}$.
If we let the roots $\alpha_{i,j}$ with $i < j$ be positive,
	then the simple roots are $\alpha_{0,1}, \alpha_{1,2}, \dots, \alpha_{n-1, n}$.
As before, denote the matrix with entry $1$ at position $(i,j)$
	by $E_{i,j}$, and define 
	$X_{i,j} \coloneqq E_{i,j} \in \lie g^{\alpha_{i,j}}$ and 
	$Y_{i,j} \coloneqq E_{j,i} \in \lie g^{\alpha_{j,i}} = \lie g^{-\alpha_{i,j}}$.
	Note that $B(X_{i,j}, Y_{i,j}) = 2(n+1) \trace(X_{i,j} Y_{i,j}) = 2(n+1)$,
	so we use a normalization different from that in \autoref{subsec:twist:restricted}. 
	
	If $n=1$, it is easy to simplify the formula for $F_\hbar$
	obtained in \autoref{theo:canonicalelement:general}:
	There is only one positive root $\alpha = \alpha_{0,1}$,
	and there is a unique word $w_\ell$ of a given length $\ell \in \mathbb N_0$.
	Note that $\lambda = - \I r \alpha/2$ and $\rho = \alpha/2$, so 
	$p_{\I \lambda/\hbar}(m \alpha) 
	=
	\frac 1 2 m^2 (\alpha, \alpha) 
		- \frac 1 2 m (\alpha, \alpha) 
		- \frac{1}{2\hbar} m r (\alpha, \alpha) 
	=
	\frac 1 4 m(m - 1 - \frac r \hbar)$.
	Therefore
	\begin{equation*}
	p_{\I\lambda/\hbar}^{w_\ell} (\alpha_{w_\ell}) = \prod_{m=1}^\ell \frac 4 { 
		m\left(m-1-\frac r \hbar\right)} = \frac {(-4)^\ell} {\ell!  
		\frac r \hbar 
		\left(\frac r \hbar - 1\right) \dots \left(\frac r \hbar - 
		(\ell-1)\right)}\punkt
	\end{equation*}
	We set $X \coloneqq X_{0,1}$ and $Y \coloneqq Y_{0,1}$.
	Since $B(X,Y)=4$ we have to plug the normalized elements $X/2$ and $Y/2$ into 
	\eqref{eq:canonicalelement:general}, and obtain
	\begin{equation}\label{eq:example:nIsOne}
	F_\hbar = \sum_{\ell \in \mathbb N_0} \frac {(-1)^\ell} {\ell!  
		\frac{r}{\hbar} 
		\left( \frac{r}{\hbar} - 1\right) \dots \left( 
		\frac{r}{\hbar} - (\ell-1)\right)} X^\ell 
	\tensor Y^\ell \punkt
	\end{equation}
	This result was already obtained in \cite[Example 4.16]{alekseev.lachowska:2005a},
	but the following result for arbitrary $n$ is new.
	We prove it by computing the canonical element of the Shapovalov pairing directly,
	instead of simplifying \eqref{eq:canonicalelement:general}.
\end{example}
\begin{proposition}\label{theo:example:standardOrdering}
	For $G = \SL[1+n]$, the same $\lambda$ and the same ordering as above, one has
	\begin{equation} \label{eq:example:standardOrdering}
	F_\hbar = \sum_{\ell \in \mathbb N_0} \frac {(-1)^\ell} {\ell!  
		\frac{r}{\hbar} 
		\left(\frac{r}{\hbar} - 1\right) \dots \left( 
		\frac{r}{\hbar} - (\ell-1)\right)} (X_{0,1} \tensor Y_{0,1} 
	+ \dots + 
	X_{0,n} \tensor Y_{0,n})^\ell \punkt
	\end{equation}
\end{proposition}
\begin{proof}
	The Lie algebras $\plus{\tilde{\lie n}}$ and $\minus{\tilde{\lie n}}$
	are commutative Lie algebras spanned by 
	$X_{0,1}, \dots, X_{0,n}$ and $Y_{0,1}, \dots, Y_{0,n}$ respectively,
	and therefore 
	$\lbrace X^I \coloneqq X_{0,1}^{I_1} \dots X_{0,n}^{I_n} \mid I \in \mathbb N_0^n \rbrace$ 
	and 
	$\lbrace Y^J \coloneqq Y_{0,1}^{J_1} \dots Y_{0,n}^{J_n} \mid J	\in \mathbb N_0^n \rbrace$
	are bases of $\univ(\plus{\tilde{\lie n}})$ and $\univ(\minus{\tilde{\lie n}})$.
	The Lie algebra $\plus{\lie n}$ is spanned by $X_{i,j}$ with $i < j$
	and we can view $X^I$ also as an element of $\univ(\plus{\lie n})$.
	Then $\smash{\tilde\pi_\lambda^+(X^I)} = X^I$ and similarly 
	$\smash{\tilde\pi_\lambda^-(Y^J)} = Y^J$.
	Consequently $\smash{\langle X^I, Y^J 
	\rangle^\sim_{\I\lambda/\hbar}}
	= 
	\smash{\langle X^I, Y^J \rangle_{\I \lambda / \hbar}}$.
	For	degree reasons the bases $\lbrace X^I \rbrace$ 
	and $\lbrace Y^J \rbrace$ are orthogonal, 
	meaning that $\smash{\langle X^I, Y^J \rangle_{\I\lambda/\hbar}} = 0$ for $I \neq J$.
	Indeed, $X^I$ and $Y^J$ are homogeneous with respect to the degree $d$
	defined in the beginning of \autoref{subsec:twist:restricted},
	$d(X^I) 
	= 
	I_1 d(X_{0,1}) + \dots + I_n d(X_{0,n}) 
	= 
	I_1 \alpha_{0,1} + \dots + I_n \alpha_{0,n}$,
	and
	$d(Y^J) = -(J_1 \alpha_{0,1} + \dots + J_n \alpha_{0,n})$.
	Since the $\alpha_{0,i}$ are linearly independent, 
	\autoref{theo:shapovalovPairing:computations} implies 
	the claimed orthogonality.
	
	Therefore it suffices to determine the normalization
	$\langle X^I, Y^I \rangle_{\I\lambda/\hbar}$. 
	Define $H_i \coloneqq [X_{0,i}, Y_{0,i}] = E_{0,0} - E_{i,i}$.
	Given a multiindex $I \in \mathbb N_0^n$ we can form a sequence
	that starts with $I_1$ many $1$'s, then has $I_2$ many $2$'s,
	$\dots$, then $I_n$ many $n$'s.
	Denote the $k$-th element of this sequence by $I_{\lbrace k \rbrace}$.
	Introduce the projection $(\argumentdot)_0$ to $\univ \lie h$ 
	in the decomposition
	$\univ \lie g
	=
	\univ \lie h \oplus (
		\minus{\lie n} \cdot \univ \lie g
		+ \univ \lie g \cdot \plus{\lie n}
	)$, so that $\pi_\lambda(u) = \lambda((u)_0)$.
	Then we claim that
	\begin{equation} \label{eq:claim:XIYI}
	(X^I Y^I)_0 = I! \prod_{\ell=0}^{\abs I-1} (H_{I_{\lbrace\ell\rbrace}} - \ell) \punkt
	\end{equation}
	To see that this formula implies the proposition, note that
	\begin{equation*}
	\langle X^I, Y^I \rangle_{\I\lambda/\hbar} 
	=
	\pi_{\I \lambda/\hbar}(S(X^I) Y^I)
	=
	(-1)^{\abs I}\left(\frac\I\hbar\lambda\right)((X^I Y^I)_0)
	\end{equation*}
	and that
	$\frac \I \hbar \lambda(H_i) = \frac r \hbar$
	for all $i = 1, \dots, n$. So
	\begin{equation*}
	F_\hbar 
	= 
	\sum_{I \in \mathbb N_0^n} 
		\frac 1 {\langle X^I, Y^I\rangle_{\I \lambda/\hbar}} X^I \tensor Y^I
	=
	\sum_{I \in \mathbb N_0^n}
		\frac{(-1)^{\abs I}}
			 {I! \frac r \hbar 
			 		   \left(\frac r \hbar -1 \right) \dots 
			 		   \left(\frac r \hbar -(\abs I -1)\right)}
		X^I \tensor Y^I
	\end{equation*}
	and an application of the multinomial theorem gives \eqref{eq:example:standardOrdering}.
	
	It remains to prove \eqref{eq:claim:XIYI}.
	For $n=1$ this is the statement of \cite[Lemma 5.2]{me}.
	Note that this also means that 
	$Z \coloneqq X_{0,n}^{I_n} Y_{0,n}^{I_n} - I_n! H_n (H_n -1) \dots (H_n - I_n + 1)
	\in \univ(\spann \lbrace X_{0,n}, Y_{0,n}, H_n \rbrace)$
	satisfies $(Z)_0 = 0$.
	We proceed by induction and assume that \eqref{eq:claim:XIYI} holds for $n-1$.
	Writing $I_- = (I_1, \dots, I_{n-1},0)$ 
	and noting that $[H_n, X_{0,i}] = X_{0,i}$ for $1 \leq i \leq n-1$,
	we compute
	\begin{align*}
	&(X^I Y^I)_0 = (X^{I_-} X_{0,n}^{I_n} Y_{0,n}^{I_n} Y^{I_-})_0 \\
	&= (X^{I_-} (I_n! H_n (H_n-1) \dots (H_n-I_n+1) + Z) Y^{I_-})_0 \\
	&= I_n! ((H_n-\abs{I_-}) (H_n-\abs{I_-}-1) \dots (H_n-\abs{I_-}-I_n+1) 
	X^{I_-} Y^{I_-})_0 + (X^{I_-} Z Y^{I_-})_0 \\
	&= I_n! (H_n-\abs{I_-}) (H_n-\abs{I_-}-1) \dots (H_n-\abs{I_-}-I_n+1) 
	(X^{I_-} Y^{I_-})_0 + (X^{I_-} Z Y^{I_-})_0 .
	\end{align*}
	Since $(Z)_0 = 0$ and  
	$d(Z) = d(X_{0,n}^{I_n} Y_{0,n}^{I_n} - I_n! H_n (H_n -1) \dots (H_n - I_n + 1)) = 0$
	we can write $Z = Y_{0,n} Z' X_{0,n}$ for some
	$Z' \in \univ (\spann \lbrace X_{0,n}, Y_{0,n}, H_n \rbrace)$.
	Since $Y_{0,n} \in \lie g^{\alpha_{n,0}}$ any commutator of $Y_{0,n}$
	with elements of $\lie g^{\alpha_{0,1}}, \dots, \lie g^{\alpha_{0,n-1}}$
	has degree $d$ equal to $L_n - \sum_{i=0}^{n-1} c_i L_i$ for some $c_i \in \mathbb Z$,
	so it must either be $0$ or in a negative root space.
	Therefore $(X^{I_-} Z Y^{I-})_0 = 0$, 
	and the claim follows by applying the induction hypothesis to the first summand
	in the equation above.
\end{proof}

\begin{corollary}\label{theo:example:oppositeOrdering}
	Let $G = \group{SL}_{1+n}(\mathbb C)$ and $\lambda$ be as above,
	but choose the opposite ordering,
	for which $\alpha_{i,j}$ with $i > j$ is positive.
	Then
	\begin{equation} \label{eq:example:oppositeOrdering}
	F_\hbar = \sum_{\ell \in \mathbb N_0} \frac {1} {\ell! 
		\frac{r}{\hbar} 
		\left(\frac{r}{\hbar} + 1\right) \dots \left( 
		\frac{r}{\hbar} + (\ell-1)\right)} (Y_{0,1} \tensor X_{0,1} 
	+ \dots + 
	Y_{0,n} \tensor X_{0,n})^\ell \punkt
	\end{equation}
\end{corollary}
\begin{proof}
	The only change in the computation above is that the roles of 
	$X_{0,i}$ and $Y_{0,i}$ are swapped. Now $[Y_{0,i}, X_{0,i}] = 
	E_{i,i}-E_{0,0}$, so $\frac \I \hbar \lambda([Y_{0,i}, X_{0,i}]) = 
	-\frac r\hbar$, which means that 
	$r$ changes sign.
\end{proof} 
\section{Continuity}
\label{sec:continuity}
In this section, we will extend the product
$*_\hbar \colon \Pol(\orb\lambda) \times \Pol(\orb\lambda) \to \Pol(\orb\lambda)$ 
obtained in \autoref{theo:starProduct:strict:associative} to a product 
$*_\hbar \colon \Hol(\orb\lambda) \times \Hol(\orb\lambda) \to \Hol(\orb\lambda)$
on all holomorphic functions on the coadjoint orbit,
that is continuous with respect to the topology of locally uniform convergence.
More precisely, we prove the following theorem.

\begin{theorem} \label{theo:mainContinuityTheorem}
	Let $\orb \lambda$ be a complex semisimple coadjoint orbit
	of a complex semisimple connected Lie group $G$.
Then for any $\hbar \in \mathbb C \setminus P_\lambda$
	the product $*_\hbar$ on $\Pol(\orb\lambda)$ 
	is continuous with respect to the topology of locally uniform convergence 
	and extends to a continuous and $G$-invariant product 
	$*_\hbar \colon \Hol(\orb\lambda) \times \Hol(\orb\lambda) \to \Hol(\orb\lambda)$ 
	on the space of all holomorphic functions on $\orb\lambda$.
\end{theorem}
The proof of this theorem proceeds as follows:
In \autoref{subsec:continuityEstimates} we prove the continuity of $*_\hbar$ 
with respect to a topology that we call the reduction-topology and in 
\autoref{subsec:reductionTopIsLocallyUniformConvergence} we prove that the 
reduction-topology coincides with the topology of locally uniform convergence.
Consequently $*_\hbar$ extends to the completion of the space
of polynomials on $\orb\lambda$.
Using the results of \autoref{subsec:stein} we prove in 
\autoref{subsec:reductionTopIsLocallyUniformConvergence} that this completion 
is the space $\Hol(\orb\lambda)$ of all holomorphic functions on $\orb\lambda$.

In the whole section we assume that the complex connected semisimple Lie group $G$
is concretely realized as a complex Lie subgroup of $\GL[N]$ for some $N \in \mathbb N$,
as explained in \autoref{subsec:generalities}.
In particular, since $G$ is semisimple,
it is a closed submanifold of $\mathbb C^{N \times N}$.

\subsection{Continuity in the reduction-topology}
\label{subsec:continuityEstimates}

In this subsection we prove the continuity of the star product $*_\hbar$
with respect to a topology that we call the reduction-topology,
and which is defined below.
Recall that a sequence of functions $f_i \colon X \to \mathbb C$ on a
topological space $X$ is said to be locally uniformly convergent
if for every $x \in X$ there is a neighbourhood $U \subseteq X$
such that $f_i$ converges uniformly to $f$ on $U$,
i.e.\ $\lim_{i \to \infty} \sup_{y \in U} \abs{f_i(y) - f(y)} = 0$.
In this work, $X$ will always be a manifold. Then the topology of 
locally uniform convergence coincides with the topology of 
compact convergence (for every compact subset $K \subseteq X$,
$f_i$ converges uniformly on $K$), and is therefore a locally convex topology,
defined by the seminorms $\norm f_K \coloneqq \sup_K \abs f$.

Denote the ideal of polynomials in $\Pol(\mathbb C^{N \times N})$
whose restriction to $G$ vanishes by $\vanideal{G}$.

\begin{definition}[Reduction-topology] \label{def:reductionTopology}
	The topology $\mathcal T_{\mathrm{lc}}$ of locally uniform convergence
	on the space $\Pol(\mathbb C^{N\times N})$ of polynomials
	on $\mathbb C^{N \times N}$ induces a quotient topology on the space
	$\Pol(G) \cong \Pol(\mathbb C^{N\times N}) / \vanideal{G}$
	of polynomials on $G$,
	and we call the subspace topology on the space
	$\Pol(\orb\lambda) \cong \Pol(G)^{G_\lambda}$ 
	of polynomials on the coadjoint orbit $\orb\lambda$
	the reduction-topology.
\end{definition}
In \autoref{subsec:reductionTopIsLocallyUniformConvergence}
we will prove that the reduction-topology coincides with the 
topology of locally uniform convergence on $\orb\lambda$.

This topology is convenient for obtaining continuity estimates for $\starhbar$,
since we gave a description of $\Psi(F_\hbar)$ via bidifferential operators
on $G$ in \autoref{subsec:diffopsOnHomogeneousSpaces}. 
Since we assume that the Lie group $G$ is concretely realized 
as a complex Lie subgroup of $\GL[N]$,
its Lie algebra $\lie g$ is realized as a Lie subalgebra of
$\liegl[N]$.
Considering the element $F'_{\hbar,\ell}$ defined in \eqref{eq:fhbarprime}
as an element of $\univ(\liegl[N]) \tensor \univ(\liegl[N])$,
we let
\begin{equation} \label{eq:defStarProduct:Strict:onCn}
*'_\hbar \colon 
\Pol(\mathbb C^{N\times N}) \times \Pol(\mathbb C^{N\times N}) 
\to \Pol(\mathbb C^{N\times N}) \komma \quad
(p, q) \mapsto p *'_\hbar q 
\coloneqq 
\sum_{\ell=0}^\infty \leftinvholo{(F'_{\hbar,\ell})}(p, q) \punkt
\end{equation}
This is well-defined since the sum over $\ell$ is finite by \autoref{theo:starFiniteOnPolys}
and $\leftinvholo{(F'_{\hbar,\ell})}(p,q)$ is again a polynomial.
Note that $*'_\hbar$ is (in general) not associative
since $\sum_{\ell=0}^\infty \smash{F'_{\hbar, \ell}}$ satisfies \eqref{eq:associativity:general} 
only after passing to the quotient.
However, since $\smash{F'_{\hbar, \ell}}$ lies in the subspace
$\univ\lie g \tensor \univ\lie g$ it induces a product on 
$\Pol(G) \cong \Pol(\mathbb C^{N\times N}) / \vanideal{G}$.
As in \autoref{remark:leaveOutProjections} it follows that
the restriction of this product to $\Pol(G)^{G_\lambda} \cong \Pol(\orb\lambda)$
coincides with $\starhbar$.

\begin{theorem} \label{theo:continuity:onCn}
	For $\hbar \in \mathbb C \setminus P_\lambda$ the product $*'_\hbar$ on
	$\Pol(\mathbb C^{N\times N})$ is continuous with respect to	the
	topology of locally uniform convergence $\mathcal T_{\mathrm{lc}}$.
\end{theorem}
Before proving this theorem in the rest of this section,
we would like to note the following consequence,
which motivates the definition of the reduction-topology given above.

\begin{corollary} \label{theo:continuity:onOrbits}
	For $\hbar \in\mathbb C \setminus P_\lambda$ the product $\starhbar$ on 
	$\Pol(\orb\lambda)$ is continuous with respect to the re\-duc\-ti\-on-to\-po\-lo\-gy.
\end{corollary}
\begin{proof}
	This follows immediately from the previous theorem and the construction of 
	the re\-duc\-ti\-on-to\-po\-lo\-gy.
\end{proof}
\begin{remark}
	It is interesting to point out that the proof of 
	\autoref{theo:continuity:onCn} will not use anything 
	about the actual Lie algebra structure but semisimplicity 
	and the form of the element 
	$F_\hbar$. In fact, we only need that the coefficients of $F_\hbar$ behave 
	like 
	$p^w_\lambda(\alpha_w) \approx \abs w^2$ for large $\abs w$.
	The rest of the proof consists in counting terms and
	checking that there are not too many.	
\end{remark}
The strategy to prove \autoref{theo:continuity:onCn} is as follows.
We first introduce a different locally convex topology
that is better suited for obtaining continuity estimates.
Then we prove that this topology is equivalent
to the topology of locally uniform convergence
and we prove the continuity of $*'_\hbar$ with respect to this topology.

Set $m = N^2$.
Let $B = \lbrace b_1, \dots, b_m \rbrace$ be the standard basis
of $\mathbb C^m$ and denote the dual basis of $(\mathbb C^m)^*$
by $B^* = \lbrace b_1^*, \dots, b_m^* \rbrace$.
Elements of $\Pol(\mathbb C^m) \cong \Sym((\mathbb C^m)^*)$
(where $\Sym$ denotes the symmetric tensor algebra) 
can be written uniquely in the form $\sum_{I \in \mathbb N_0^m} a_I b_I^*$.
Here $I \in \mathbb N_0^m$ is a multiindex, 
$\smash{b_I^*} = \smash{(b_1^*)^{\vee I_1}} \vee \dots \vee \smash{(b_m^*)^{\vee I_m}}$
and only finitely many of the coefficients $a_I \in \mathbb C$ are non-zero.  
For any $R \in \mathbb R^+$ define a norm
\begin{equation}
\mynorm R {{\sum}_{I \in \mathbb N_0^m} a_I b_I^* } 
\coloneqq
{\sum}_{I \in \mathbb N_0^m} \abs{a_I} R^{\abs I} \punkt
\end{equation}
Note that these norms coincide with the $T_0$-norms with respect to the basis $B^*$, 
studied for example in \cite{waldmann:2014a}.
We denote the locally convex topology given by 
endowing $\Pol(\mathbb C^m) \cong \Sym((\mathbb C^m)^*)$ with the 
seminorms $\mynorm R {{}\cdot{}}$ by $\mathcal T_{\mynorm{}{{}\cdot{}}}$.
This topology can equivalently be defined by the countable set
of norms $\mynorm R {{}\cdot{}}$ with $R \in \mathbb N$.

Note that $\mynorm R {{}\cdot{}}$ is submultiplicative with respect to the 
classical 
product:
\begin{multline*}
\mynorm R {
	\left( {\sum}_{I \in \mathbb N_0^m} a_I b_I^* \right)
	\vee 
	\left( {\sum}_{J \in \mathbb N_0^m} a'_J b_J^* \right)
} 
=
\mynorm R {
	{\sum}_{I,J \in \mathbb N_0^m} a_I a'_J b_I^* \vee b_J^* 
} 
\leq 
\\
\leq 
{\sum}_{I,J\in \mathbb N_0^m} \abs{a_I} \abs{a'_J} R^{\abs I + \abs J}
= 
\left( {\sum}_{I\in \mathbb N_0^m} \abs{a_I} R^{\abs I} \right)
\left( {\sum}_{J\in \mathbb N_0^m} \abs{a'_J} R^{\abs J} \right) 
=
\\
=
\mynorm R {{\sum}_{I \in \mathbb N_0^m} a_I b_I^*} 
\mynorm R {{\sum}_{J \in \mathbb N_0^m} a'_J b_J^*} \punkt
\end{multline*}

\begin{proposition}\label{theo:topologies:agree}
	The topologies $\mathcal T_{\mynorm{}{{}\cdot{}}}$ and $\mathcal 
	T_{\mathrm{lc}}$ 
	coincide.
\end{proposition}
\begin{proof}
	Assume $p = \sum_{I \in \mathbb N_0^m} a_I b_I^* \in \Pol(\mathbb C^m)$
	is a polynomial.
Given a compact subset $K \subseteq \mathbb C^m$,
	choose $R \in \mathbb R$ such that 
	$\abs z \leq R$ holds for all $z \in K$.
Then on the one hand we have
	\begin{equation*}
	\norm p_K 
	= \max_{z \in K} \abs{p(z)} 
	\leq \sum_{I \in \mathbb N_0^m} \abs{a_I} R^{\abs I}
	= \mynorm R p \punkt
	\end{equation*}
	On the other hand, if 
	$D_R = \lbrace (z_1, \dots, z_m) \in \mathbb C^m \mid \abs{z_i} \leq R 
	\text{ for all $i = 1, \dots, m$} \rbrace\subseteq \mathbb C^m$ 
	denotes a closed polydisc of radius $R$, 
	then Cauchy's integral formula yields
	\begin{equation*}
	\abs {a_I} = \frac 1 {I!} \abs{\partial_I p(0)} = \frac 1 {(2\pi)^m}\left| 
	\int_{\abs{z_i} = 
		R} 
	\frac{p(z)}{z^{I+(1, \dots, 1)}} \D z^{I} \right|
	\leq \max_{z \in D_R} \abs{p(z)}\frac{R^m}{R^{\abs{I+(1, \dots, 
				1)}}} 
	= \frac{1}{R^{\abs{I}}} \max_{z \in D_R} \abs{p(z)} 
	\punkt
	\end{equation*}
	Applying this estimate for a polydisc of radius $2 m R$ yields
	\begin{multline*}
	\mynorm R p = \sum_{I \in \mathbb N_0^m} \abs{a_I} R^{\abs I}
	\leq \sum_{I \in \mathbb N_0^m} \frac{1}{(2 m R)^{\abs{I}}} R^{\abs 
		I}\max_{z \in D_{2 m 
			R}} 
	\abs{p(z)} \leq \\
	\leq  \max_{z \in D_{2 m R}} 
	\abs{p(z)}\sum_{I \in \mathbb N_0^m} (2m)^{-\abs{I}}
	\leq 2 \max_{z \in D_{2 m R}} 
	\abs{p(z)} = 2 \norm p_{D_{2 m R}}\punkt
	\end{multline*}
	Consequently we can estimate any norm of $\mathcal T_{\mynorm{}{{}\cdot{}}}$
	by a seminorm of $\mathcal T_{\mathrm{lc}}$ and vice versa,
	so the topologies $\mathcal T_{\mynorm{}{{}\cdot{}}}$ and 
	$\mathcal T_{\mathrm{lc}}$ coincide.
\end{proof}
Because of the previous proposition we can and will work with the norms 
$\mynorm R \argumentdot$ instead of the seminorms $\norm \argumentdot_K$ in the 
following.
To obtain continuity estimates, we need to estimate
the coefficients $p_\lambda(\mu)$ defined in \eqref{eq:def:plambda}.

\begin{lemma}[Estimates for $p_\lambda$] \label{theo:estimates:pLambda}
	For any fixed compact set $K \subseteq \lie h^*$ there are constants $C > 
	0$ and $M$ such that 
	$p_\lambda(\alpha_w)$ defined in 
	\eqref{eq:def:plambda} satisfies
	\begin{equation}
	\abs{p_\lambda(\alpha_w)} \geq C \abs w^2
	\end{equation}
	for all words $w \in \wordset$ of length $\abs w \geq M$ and all $\lambda 
	\in K$.
\end{lemma}

\begin{proof}
	Assume that the positive roots $\alpha_1, \dots, \alpha_k \in \Delta^+$ are 
	ordered in such a way that $\alpha_1, \dots, \alpha_r$ are the simple roots.
Write $\alpha_w = \sum_{i=1}^r c_{w,i} \alpha_i$ as a linear 
	combination of simple roots, where $c_{w,i} \in \mathbb N_0$ satisfy 
	$\abs w \leq \sum_{i=1}^r c_{w,i} \leq c \abs w$ with $c$ depending only on 
	the 
	root system.
	Since $(\rho, \alpha_i) > 0$ for all $1 \leq i \leq r$ we can choose 
	$c_\rho, C_\rho \in \mathbb R^+$ such that $c_\rho \leq (\rho, \alpha_i) 
	\leq C_\rho$ holds for all $1 \leq i \leq r$. Similarly, there is 
	$C' \in \mathbb R^+$ with $\abs{(\lambda, \alpha_i)} \leq C'$ for all 
	$\lambda \in K$ and $ 1 \leq i \leq r$. Then
	\begin{multline*}
	(\alpha_w, \alpha_w)
	\geq
	\frac{1}{(\rho, \rho)} (\alpha_w, 
	\rho)^2 
	=
	\frac 1 {(\rho, \rho)} \left(\sum_{i=1}^r (c_{w,i} \alpha_i , \rho) 
	\right)^2
	\geq
	\frac {c_\rho^2} {(\rho, \rho)} \left(\sum_{i=1}^r c_{w,i} \right)^2
	\geq
	\frac {c_\rho^2} {(\rho, \rho)} \abs{w}^2
	\end{multline*}
	and for all $\lambda \in K$ we obtain
	\begin{equation*}
	\abs{(\rho+\lambda, \alpha_w)} 
	\leq 
	\sum_{i=1}^r c_{w,i} (\abs{(\rho,\alpha_i)} + \abs{(\lambda, \alpha_i)}) 
	\leq
	(C_\rho + C') \sum_{i=1}^r c_{w,i}
	\leq
	c (C_\rho + C') \abs w \punkt
	\end{equation*}
	Setting $C \coloneqq \frac 1 {4 (\rho, \rho)}c_\rho^2$, 
	$C_1 \coloneqq c(C_\rho + C')$, and 
	$M \coloneqq \frac{C_1}{C}$,
	and assuming $\abs w \geq M$ we obtain
	\begin{equation*}
	\abs{p_\lambda(\alpha_w)}
	\geq
	\frac 1 2 (\alpha_w, \alpha_w) - \abs{(\rho+\lambda, \alpha_w)}
	\geq
	2 C \abs{w}^2 - C_1 \abs w
	\geq
	2 C \abs{w}^2 - C \abs w^2 
	=
	C \abs w^2 \punkt
	\end{equation*}
\end{proof}

\begin{corollary}[Estimates for $p_\lambda^w$] 
	\label{theo:estimates:pLambdaW}
	Fix $\lambda \in \lie h^*$.
	For any compact set
	$K \subseteq \mathbb C \setminus P_\lambda$
	there is a constant $C_p > 0$ such that 
	$p^w_{\I \lambda/\hbar}(\alpha_w)$ defined in \eqref{eq:def:plambda} 
	satisfies
	\begin{equation}
	\abs{p^w_{\I \lambda/\hbar}(\alpha_w)^{-1}} 
	\leq \frac{C_p^{\abs w}}{(\abs 	w!)^2}
	\end{equation}
	for all words $w \in \tilde\wordset$ and all $\hbar \in K$.
\end{corollary}

\begin{proof}
	Note that $K' = \lbrace \I \lambda / \hbar \mid \hbar \in K\rbrace$
	is a compact subset of $\tilde \Lambda$.
Let $M$ and $C$ be the constants obtained by applying the previous lemma to $K'$,
	so $\abs{p_{\lambda'}(\alpha_w)} \geq C \abs w^2$ 
	for all $w \in W$ with $\abs w \geq M$ and all $\lambda' \in K'$.
Since $\I \lambda / \hbar \in \tilde \Lambda$, we have
	$\min_{w \in \tilde \wordset, \abs w < M}
	\abs{p_{\I \lambda / \hbar}(\alpha_w)} > 0$ for all $\hbar \in K$.
Since this quantity depends continuously on $\hbar$
	the minimum for $\hbar \in K$ exists and must also be positive. 
Hence we may decrease the constant $C$ such that 
	$\abs{p_{\I \lambda/\hbar}(\alpha_{w})} \geq C \abs {w}^2$
	also holds for the finitely many words $\smash{w \in \tilde W}$ with $\abs w < M$.
	Consequently $\abs{p_{\I \lambda/\hbar}(\alpha_{w})} \geq C \abs {w}^2$
	holds for all words $\smash{w \in \tilde W}$.
	Setting $C_p \coloneqq 1/C$, the corollary follows by rearranging.
\end{proof}
We have now collected all the results needed to prove 
\autoref{theo:continuity:onCn}.

\begin{proofof}[Proof of \autoref{theo:continuity:onCn}:] First, we note that 
	it 
	suffices to prove the existence of a constant $M$ such that for any 
	multiindices $I, J \in \mathbb N_0^m$ we have $\mynorm R {b_I^* *'_\hbar 
	b_J^*} 
	\leq  
	(RM)^{\abs I + \abs J}$. Indeed, this statement 
	implies the continuity of $*'_\hbar$ since for $p = \sum_{I \in \mathbb 
	N_0^m} 
	p_I b_I^*$ and $q = \sum_{I \in \mathbb N_0^m} q_I b_I^*$ we estimate
	\begin{align*}
	\mynorm R {p *'_\hbar q} 
	&= \mynorm R {
		{\sum}_{I \in \mathbb N_0^m} p_I b_I^* *'_\hbar {\sum}_{J \in \mathbb 
			N_0^m} q_J b_J^* } \\
	&\leq \sum_{I \in \mathbb N_0^m} \sum_{J \in \mathbb N_0^m}
	\abs{p_I} \abs{q_J} \mynorm R {b_I^* *'_\hbar  b_J^*} \\
	&\leq \sum_{I \in \mathbb N_0^m} \sum_{J \in \mathbb N_0^m}
	\abs{p_I} \abs{q_J} (RM)^{\abs I + \abs J} \\
	&= \mynorm{RM}{{\sum}_{I \in \mathbb N_0^m} p_I b_I^*} 
	\mynorm{RM}{{\sum}_{J \in \mathbb N_0^m}
		q_J b_J^*} \\
	&= \mynorm{RM}{p} 
	\mynorm{RM}{q}\punkt
	\end{align*}
	Using the notation $I_{\lbrace j \rbrace}$ introduced in the proof of 
	\autoref{theo:example:standardOrdering} we estimate
	\newcommand{\somespace}{\phantom{XX}}
	\begin{align*}
	&\mynorm R {b_I^* *'_\hbar b_J^*} 
	= \mynorm R {{\sum}_{\ell=0}^\infty \leftinvholo{(F'_{\hbar,\ell})}(b_I^*, b_J^*)} 
	\\
	&\somespace
	\leq  
	\mynorm R {
		{\sum}_{w \in \tilde\wordset} 
		p^w_{\I\lambda/\hbar}(\alpha_{w})^{-1}
		\leftinvholo{(X_w \tensor Y_w)} (b_I^*, b_J^*)
	}
	\\
	&\somespace
	\textuebersign{\leq}{(1)}  
	\sum_{w \in\tilde \wordset} 
	\abs{p^w_{\I\lambda/\hbar}(\alpha_w)^{-1}} \cdot{} \\ 
	&\phantom{SPACE} {}\cdot
	\sum_{w_{(1)}, \dots, w_{(\abs I)}}  
	\sum_{w'_{(1)}, \dots, w'_{(\abs J)}} 
	\mynorm R{\leftinvholo{X_{w_{(1)}}} {b^*_{I_{\lbrace 1\rbrace}}}} \dots
	\mynorm R{\leftinvholo{X_{w_{(\abs I)}}} {b^*_{I_{\lbrace\abs I\rbrace}}}} 
	\cdot{} \\
	&\phantom{SOMEVERYVELARGESPACE} {}\cdot
	\mynorm R{\leftinvholo{Y_{w'_{(1)}}} {b^*_{J_{\lbrace 1 \rbrace}}}} \dots
	\mynorm R{\leftinvholo{Y_{w'_{(\abs J)}}} {b^*_{J_{\lbrace\abs J\rbrace}}}} 
	\\
	&\somespace
	\textuebersign{\leq}{(2)}  
	\sum_{w\in\tilde\wordset} 
	\frac{C_p^{\abs w}}{(\abs w!)^2} 
	\abs I^{\abs w} \abs J^{\abs w} C^{2 \abs w} R^{\abs I + \abs J}
	\\
	&\somespace
	\textuebersign{\leq}{(3)} R^{\abs I + \abs J}
	\sum_{\ell = 0}^\infty (k C_p C^2)^{\ell} 
	\frac{\abs I^{\ell} \abs J^{\ell}}{(\ell!)^2}
	\\
	&\somespace
	\textuebersign{\leq}{(4)} R^{\abs I + \abs J} 
	\sum_{\ell = 0}^\infty 
	\frac{(k^{1/2} C_p^{1/2} C \abs I)^{\ell}}{\ell!}
	\sum_{\ell' = 0}^\infty 
	\frac{(k^{1/2} C_p^{1/2} C \abs J)^{\ell'}}{\ell'!} 
	\\
	&\somespace
	\leq R^{\abs I + \abs J} 
	\E^{k^{1/2} C_p^{1/2} C \abs I} 
	\E^{k^{1/2} C_p^{1/2} C \abs J}  
	\\
	&\somespace
	= (R \E^{k^{1/2} C_p^{1/2} C})^{\abs I + \abs J} \punkt
	\end{align*}
	The sum $\sum_{w_{(1)}, \dots, w_{(\abs I)}}$ introduced in (1) is
	over all partitions of $w \in \tilde \wordset$ into words $w_{(1)}, \dots, w_{(\abs I)}$.
To be more precise, consider a partition $P_1, \dots, P_{\abs I}$
	of $\lbrace 1, \dots, \abs w \rbrace$ into $\abs I$ many subsets.
If $P_i = \lbrace p_{i,1}, \dots, p_{i,j_i}\rbrace$ 
	with $p_{i,1} < \dots < p_{i,j_i}$, 
	then associate the word $w_{(i)} = w_{p_{i,1}} w_{p_{i,2}} \dots w_{p_{i, j_i}}$.
Then we sum over all partitions.
The other sum is defined similarly.
We also used submultiplicativity of $\mynorm R \argumentdot$ in this step.
To justify (2), we note that for any $Z \in \lie{gl}_N(\mathbb C)$, 
	$\leftinvholo{Z} b^*_i$ is of degree 1,
	so that $\leftinvholo{X_{w_{(\ell)}}} b^*_{I_{\lbrace\ell\rbrace}}$ is of degree 1. 
Defining $C \coloneqq \max_{i \in \lbrace 1, \dots, m\rbrace,\alpha \in \Delta} 
	\mynorm 1 {\leftinvholo X_{\alpha} b^*_i}$
	we obtain
	\begin{equation*}
	\mynorm R {\leftinvholo{X_{w_{(\ell)}}} b^*_{I_{\lbrace\ell\rbrace}}}
	\leq C^{\abs{w_{(\ell)}}} R \punkt
	\end{equation*}
The sum over $w_{(1)}, \dots, w_{(\abs I)}$ has $\abs I^{\abs w}$ many terms, 
	since for each letter of $w$ we can choose in which of the $\abs I$ many sets
	we want to have it.
The same holds true for the other sum.
In (3) we used that there are at most $\smash{k^{\abs w}}$ many words
	of a given length $\abs w$ in $\smash{\tilde \wordset}$
	and (4) holds, because we just added some positive extra terms.
\end{proofof}

\begin{remark}\label{theo:continuity:estimatesAreLocallyUniform}
	For a fixed compact set $K \subseteq \mathbb C \setminus P_\lambda$ 
	the proof above shows that there is a constant $M \in \mathbb R^+$
	such that for any $\hbar \in K$ we have
	\begin{equation}
	\mynorm R {p *'_\hbar q} \leq \mynorm {RM} p \mynorm {RM} q
	\end{equation}
	since \autoref{theo:estimates:pLambdaW} 
	gives uniform estimates for all $\hbar \in K$.
\end{remark}
 
\subsection{Stein manifolds and extension of holomorphic functions}
\label{subsec:stein}

In this subsection, we discuss extension properties of holomorphic
functions on closed complex submanifolds of Stein manifolds
or, more generally, on analytic subsets of Stein manifolds.
We will use the results in the next subsection to identify
the reduction-topology with the topology of locally uniform convergence
and to determine the completion of the space of 
polynomials with respect to this topology.

Since analytic subsets in a Stein manifold are a 
very natural setting to prove the extendability results,
we formulate them in this generality
(even though we only need the case of closed submanifolds most of the time). 
The content of this subsection has been known for long and can be 
found e.g.\ in the textbook \cite{hoermander:1990b}.

Recall that for a complex manifold $M$, we denote the vector space
of holomorphic functions on $M$ by $\Hol(M)$.

\begin{definition}[Holomorphic convex hull]
	For a compact subset $K$ of a complex manifold $M$ we define its 
	\emph{holomorphic convex hull} to be the set
	\begin{equation}
	\hat K_M = \lbrace z \in M \mid \abs{f(z)} \leq \sup_{K} \abs f \text{ for 
		all $f \in \Hol(M)$} \rbrace \punkt
	\end{equation}
\end{definition}

\begin{definition}[Stein manifold]
	A complex manifold $M$ of dimension $n$ is said to be \emph{Stein} if
	\begin{definitionlist}
		\item for any compact subset $K \subseteq M$ its holomorphic convex 
		hull $\hat K_M$ is compact,
		\item for every $z \in M$ there are functions $f_1, \dots, f_n \in 
		\Hol(M)$ that form a coordinate system around $z$.
	\end{definitionlist}
\end{definition}
Stein manifolds should be thought of as domains of holomorphicity for holomorphic 
functions of several complex variables. Clearly $\mathbb C^n$ is Stein.

\begin{definition} \label{def:analyticSubset}
	A subset $V \subseteq M$ of a complex manifold is called \emph{analytic}, 
	if for every point $z \in M$ there is a neighbourhood $U \subseteq M$ of $z$
	such that there is a family of holomorphic functions $f_j \in \Hol(U)$
	indexed by $j$ in some index set $J$, such that
	\begin{equation}
	V \cap U = \lbrace z \in U \mid f_j(z) = 0 \text{ for all $j \in 
		J$}\rbrace\punkt
	\end{equation} 
\end{definition}

\begin{example}\label{ex:closedSubmanifoldsAreAnalytic}
	Any closed complex submanifold $M$ of $\mathbb C^n$ is an analytic subset. 
	Indeed, around any $z \in M$ we can find a submanifold chart,
	that is a neighbourhood $U$ and coordinates $z = (z_1, \dots, z_n)$
	such that $M \cap U$ is given by the vanishing
	of the first $n - \dim M$ coordinates.
	Therefore we can take $f_j = z_j$ for $j = 1, \dots, n-\dim M$ 
	in \autoref{def:analyticSubset}.
	Around any $z \notin M$ there is a neighbourhood $U$ such that $U \cap M = 
	\emptyset$ and we may pick $f_1 = 1$ in \autoref{def:analyticSubset}.
\end{example}

\begin{definition}
	A function $f \colon V \to \mathbb C$ on an analytic subset $V \subseteq M$ of a 
	complex manifold is called holomorphic, if for every point $z \in V$ there 
	is a neighbourhood $U \subseteq M$ of $z$ and a holomorphic function $g \in 
	\Hol(U)$ such that $g \at{U \cap V} = f \at{U \cap V}$.
\end{definition}

\begin{example}
	If $V$ is a closed complex submanifold of $\mathbb C^n$ as in 
	\autoref{ex:closedSubmanifoldsAreAnalytic},
	then this definition of a holomorphic function coincides with the 
	usual definition.
	Indeed, in any submanifold chart $(U, z)$ as in 
	\autoref{ex:closedSubmanifoldsAreAnalytic},
	a holomorphic function on $U \cap V$ can be extended constantly
	along the first $n - \dim M$ variables to a holomorphic function on $U$.
	The reverse implication is clear.
\end{example}

\begin{proposition}\label{theo:holextendability:HolVIsFrechet}
	Let $V$ be an analytic subset of a Stein manifold $M$. Then $\Hol(V)$ 
	endowed with the topology of locally uniform convergence is a Fr\'echet 
	space.
\end{proposition}

\begin{proof}
	It follows from the definition of analytic subsets that $V$ is closed. 
Therefore the restriction of any compact exhaustion of $M$ to $V$
	gives a compact exhaustion $K_i$ of $V$.
The seminorms $\norm f_{K_i} = \sup_{K_i} \abs f$ define a countable system 
	of seminorms inducing the topology of locally uniform convergence.
The completeness of $\Hol(V)$ with respect to this topology is a 
	non-trivial result and proved in \cite[Theorem 7.4.9]{hoermander:1990b}.
\end{proof}
The crucial property of an analytic subset $V$ of a Stein manifold is the following 
extendability property for any holomorphic function on $V$, see \cite[Theorem 7.4.8]{hoermander:1990b}:

\begin{theorem}[Extendability of holomorphic functions]\label{theo:holextendability:extension}
	Let $V$ be an analytic subset of a Stein manifold $M$.
	Any holomorphic function $f \in \Hol(V)$ can be extended
	to a holomorphic function $f \in \Hol(M)$. 
	In other words, the restriction map $\Hol(M) \to \Hol(V)$ is surjective.
\end{theorem}
For an analytic subset $V$ of a complex manifold $M$ we denote the
subspace of $\Hol(M)$ consisting of functions that vanish on $V$ by $\vanideal V$.
Note that the restriction map $\Hol(M) \to \Hol(V)$ descends
to a map on the quotient, $r \colon \Hol(M)/\vanideal V \to \Hol(V)$.
This map is clearly injective by definition of $\vanideal V$,
and if $M$ is Stein it is surjective by the previous theorem.

\begin{corollary}
	\label{theo:topologies:quotientTopIsLocallyUniformConvergence}
	Assume that $M$ is Stein and that $V \subseteq M$ is an analytic subset.
	If $\Hol(M)/\vanideal V$ is endowed with the quotient topology
	of the topology of locally uniform convergence and 
	$\Hol(V)$ is endowed with the topology of locally uniform convergence
	then the map $r \colon \Hol(M)/\vanideal V \to \Hol(V)$ is a homeomorphism.
\end{corollary}

\begin{proof}
	We know that $r$ is bijective, so it only remains to prove the continuity 
	of $r$ and $r^{-1}$. Both $\Hol(M)$ and $\Hol(V)$ are Fr\'echet spaces (for
	$\Hol(M)$ this is well-known, for $\Hol(V)$ it is 
	the statement of \autoref{theo:holextendability:HolVIsFrechet}). Since 
	$\vanideal V$ is closed, $\Hol(M)/ \vanideal V$ is also a Fr\'echet space. 
	Clearly the locally uniform convergence of a sequence $f_i \in \Hol(M)$ 
	implies the locally uniform convergence of the sequence of restrictions 
	$f_i \at V \in \Hol(V)$, so the map $r$ is continuous. The statement then 
	follows from the open mapping theorem for Fr\'echet spaces.
\end{proof}
 
\subsection{Characterizing the reduction-topology}
\label{subsec:reductionTopIsLocallyUniformConvergence}

In this subsection, we show that the reduction-topology on $\orb\lambda$
as defined in \autoref{subsec:continuityEstimates} is the topology
of locally uniform convergence and that the completion
of the space of polynomials $\Pol(\orb\lambda)$ on $\orb\lambda$
with respect to this topology is exactly the space of 
holomorphic functions $\Hol(\orb\lambda)$ on $\orb\lambda$.

\begin{proposition}\label{theo:topologies:reductionIsLocallyUniformConvergence}
	The reduction topology $\mathcal T_{\mathrm{red}}$ on $\orb\lambda$
	coincides with the topology of locally uniform convergence.
\end{proposition}
\begin{proof}
	By the assumption at the beginning of this section
	(see also \autoref{subsec:generalities}),
	$G$ is a closed complex submanifold of $\mathbb C^{N\times N}$,
	hence an analytic subset by \autoref{ex:closedSubmanifoldsAreAnalytic}. 
Applying \autoref{theo:topologies:quotientTopIsLocallyUniformConvergence} 
	yields that the quotient topology on $\Hol(G)$ induced by the topology
	of locally uniform convergence on $\mathbb C^{N \times N}$
	is precisely the topology of locally uniform convergence on $G$.
	
	By \autoref{def:reductionTopology} the reduction-topology
	is the restriction of this topology to the subspace of
	right $G_\lambda$-invariant holomorphic functions.
	Using that this subspace is closed, and that a sequence
	$f_i \in \Hol(\orb\lambda)$ converges locally uniformly
	if and only if the sequence $\pi^*(f_i) \in \Hol(G)^{G_\lambda}$
	converges locally uniformly, one can easily check that
	the reduction-topology coincides with the topology of
	locally uniform convergence on $\Hol(\orb\lambda)$.
\end{proof}
Finally, we would like to determine the completion 
$\widehat{\Pol}(\orb\lambda)$ of $\Pol(\orb\lambda)$ with respect to 
the topology of locally uniform convergence.

\begin{proposition} \label{theo:topologies:polCompletesToHol}
	We have $\widehat{\Pol}(\orb\lambda) = \Hol(\orb\lambda)$.
\end{proposition}

\begin{proof}
	The inclusion $\widehat{\Pol}(\orb\lambda) \subseteq \Hol(\orb\lambda)$
	is trivial, since $\Pol(\orb\lambda) \subseteq \Hol(\orb\lambda)$ and 
	the limit of a locally uniformly convergent sequence
	of holomorphic functions is again holomorphic.
	
	The other inclusion is easy to see if one uses 
	that semisimple coadjoint orbits are affine algebraic varieties, 
	see \autoref{theo:coadjointOrbitsAreClosed}:
In particular they are analytic subsets of the Stein manifold $\lie g^*$
	and therefore we can use \autoref{theo:holextendability:extension}
	to extend any $f \in \Hol(\orb\lambda)$ to a holomorphic function 
	$\tilde f \in \Hol(\lie g^*)$, which can be approximated by polynomials.
Restricting these approximating polynomials to $\orb\lambda$
	gives a sequence of polynomials in $\Pol(\orb\lambda)$
	converging locally uniformly to $f$.
	
	Alternatively, we know that $G$ is a closed submanifold of the Stein manifold
	$\mathbb C^{N \times N}$, so the same argument yields
	that any $f \in \Hol(G)$ can be approximated by some $p_n \in \Pol(G)$.
Assume that $f \in \Hol(G)^{G_\lambda}$.
	Let $K_\lambda$ be a maximal compact subgroup of $G_\lambda$.
Averaging $p_n$ over $K_\lambda$ gives a sequence
	$p'_n \in \Pol(G)^{K_\lambda}$ that converges locally uniformly to $f$.
Now $p'_n$ is even $G_\lambda$-invariant
	since the action of $G$ is holomorphic, 
	so $\pi_* p'_n \in \Pol(\orb\lambda)$ converges to
	$\pi_* f \in \Hol(\orb\lambda)$. 
\end{proof}
We are now able to prove the main theorem stated in the introduction to this section.

\begin{proofof}[Proof of \autoref{theo:mainContinuityTheorem}:]
	From \autoref{subsec:continuityEstimates} we know that the
	product $*_\hbar$ is continuous with respect to the reduction-topology.
	We showed in \autoref{theo:topologies:reductionIsLocallyUniformConvergence} 
	that the reduction-topology coincides with the topology
	of locally uniform convergence on $\orb\lambda$.
	The previous proposition shows that the completion 
	of $\Pol(\orb\lambda)$ in this topology is $\Hol(\orb\lambda)$.
	Finally $G$-invariance of the product on the completion is clear 
	since the action of $G$ on $\Pol(\orb\lambda)$ is continuous
	with respect to the topology of locally uniform convergence.
\end{proofof}
We close this section by the following proposition,
which asserts that the dependence of $*_\hbar$ on $\hbar$ is holomorphic.

\begin{proposition}[Holomorphic dependence on \boldmath$\hbar$]
	\label{theo:holomorphicDependenceOnHbar:complex}
	For two fixed holomorphic functions $p, q \in \Hol(\orb\lambda)$
	and $x \in \orb\lambda$ the map 
	$\mathbb C \setminus P_\lambda \to \mathbb C$,
	$\hbar \mapsto p \starhbar q(x)$ is holomorphic.
\end{proposition}

\begin{proof}
	By construction of $\starhbar$ in \autoref{sec:starProduct}, the map 
	$\mathbb C \setminus P_\lambda \to \mathbb C$,
	$\hbar \mapsto p' \starhbar q'(x)$ is rational 
	for $p', q' \in \Pol(\orb\lambda)$.
Assume that $p_n$, $q_n$ are sequences of polynomials on $\orb\lambda$ 
	such that $p_n \to p$ and $q_n \to q$ locally uniformly.
Since the estimates of \autoref{subsec:continuityEstimates}
	are locally uniform in $\hbar$,
	see \autoref{theo:continuity:estimatesAreLocallyUniform},
	it follows that $p_n \starhbar q_n \to p \starhbar q$
	locally uniformly in $\hbar$.
But clearly the evaluation at $x$ is continuous,
	so that $\hbar \mapsto p \starhbar q(x)$ is a locally uniform limit
	of rational functions and therefore holomorphic.
\end{proof} 
\section{Quantizing real coadjoint orbits}
\label{sec:properties}
We have seen in the previous sections how to construct (formal and strict) quantizations
of complex coadjoint orbits.
In this section, we will use these results to obtain (formal and strict) quantizations
of real coadjoint orbits.

In \autoref{subsec:complexification} and \autoref{subsec:analytic}
we collect some preliminary results on the complexification
of a real coadjoint orbit $\orb\lambda$ and a real Lie group $G$.
We define a certain class of analytic functions
that we denote by $\analytics(\orb\lambda)$ and $\analytics(G)$.
In \autoref{subsec:realOrbits} we construct a quantization of real orbits
by restricting the quantization of a complexification.
We discuss the examples of complex projective spaces and hyperbolic discs in
\autoref{subsec:realExamples}.
Finally, we show that point evaluation functionals are positive
for certain coadjoint orbits in \autoref{subsec:positiveLinearFunctionals}
and compare the quantum algebras obtained for coadjoint orbits of real Lie groups
with the same complexification in \autoref{subsec:wickrotation}.
Most results in the later subsections follow almost directly
from the results in the complex case.

From now on, all complex Lie groups and Lie algebras will be denoted with a hat
and letters without decoration will be used to denote real objects.
We will also use hats for maps between complex objects, e.g.\ we rename the map
defined in \eqref{eq:diffops:bijection:real} to $\hat\Psi$.

\subsection{Complexification}
\label{subsec:complexification}

In this subsection we define the complexification of a real coadjoint orbit $\orb\lambda$ 
and a real Lie group $G$, and show how they are related.

For a real Lie algebra $\lie g$, denote the space of 
real-valued real-linear functionals on $\lie g$ by $\lie g^*$.
As before, $\hat{\lie g}^*$ denotes the space of complex-valued
complex-linear functionals on a complex Lie algebra $\hat{\lie g}$.
In the following, we will always assume 
that $\hat{\lie g} = \lie g \tensor \mathbb C$ is the complexification of $\lie g$.
In this case, any element of $\lie g^*$ has a unique extension to an element of $\hat {\lie g}^*$.
We will perform this extension implicitly whenever necessary, without mentioning it.
For example, in the following proposition,
the coadjoint orbit $\orbext{\lambda}$ is really the coadjoint orbit
through the extension of $\lambda \in \lie g^*$ to an element of $\hat{\lie g}^*$.

\begin{proposition}
	Let $\orb\lambda \subseteq \lie g^*$ be a coadjoint orbit of a real connected Lie group,
	and assume that $\hat{\lie g}$ is the complexification of $\lie g$.
	Then $\orb\lambda$ is a submanifold of a unique complex coadjoint orbit
	$\orbext\lambda \subseteq \hat{\lie g}^*$ of a complex connected Lie group
	with Lie algebra $\hat{\lie g}$.
	The tangent space $\Tangent_\xi \orbext\lambda$ of this orbit $\smash{\orbext\lambda}$
	is the complexification of $\Tangent_\xi \orb\lambda$ for every $\xi \in \lie g^*$.
\end{proposition}
\begin{proof}
	By \autoref{theo:coadjointOrbitsAreSymplecticLeaves}
	the coadjoint orbit $\orb\lambda$ is the symplectic leaf through $\lambda$
	of the linear Poisson structure on $\lie g^*$ defined just
	before \autoref{theo:coadjointOrbitsAreSymplecticLeaves}.
	Similarly the coadjoint orbits in $\smash{\hat{\lie g}^*}$ are symplectic leaves
	of the linear Poisson structure on $\smash{\hat{\lie g}^*}$, 
	and the symplectic leaf containing $\smash{\lambda \in \hat{\lie g}^*}$
	contains the whole orbit $\orb\lambda$.
	This proves existence and uniqueness of $\orbext\lambda$.
	
	As in \autoref{subsec:generalities},
	we can identify $\Tangent_\xi \orb\lambda$ with $\lie g / \lie g_\xi$
	(as real vector spaces)
	and $\smash{\Tangent_\xi \orbext\lambda}$ with $\smash{\hat {\lie g} / \hat {\lie g}_\xi}$
	(as complex vector spaces)
	for all $\xi \in \orb\lambda$.
	Therefore $\smash{\Tangent_\xi\orbext\lambda}$ is indeed the complexification of 
	$\Tangent_\xi\orb\lambda$.
\end{proof}
We refer to the complex coadjoint orbit $\orbext\lambda$ of the previous proposition
as the \emph{complexification} of $\orb\lambda$.
We will show how to realize it explicitly as the coadjoint orbit of some Lie group $\hat G$.

\begin{definition}
	Let $G$ be a real Lie group. A \emph{complexification} of $G$
	is a complex connected Lie group $\hat G$
	together with an embedding $\iota \colon G \to \hat G$,
	such that the corresponding Lie algebra $\hat{\lie g}$ is isomorphic
	to the complexification $\lie g \tensor \mathbb C$ of $\lie g$
	and such that the map $\Tangent_e \iota \colon \lie g \to \hat{\lie g}$ corresponds
	to the injection $X \mapsto X \tensor 1$ under this isomorphism.
\end{definition}
Note that a complexification according to this definition may fail to exist
or may not be unique, if it exists. See the paragraph after 
\autoref{theo:polynomials:restrictionIsIso:group} for an example of a Lie group
with non-unique complexification.
For a connected semisimple Lie group $G$ a complexification exists if and only if
the group can be realized as a linear group: Existence for linear Lie groups is shown below,
and the reverse implication follows since semisimplicity of $G$ implies semisimplicity 
of the complexification and complex connected semisimple Lie groups are always
matrix Lie groups, see \autoref{theo:coadjointOrbitsAreClosed}.
There is a different notion of a universal complexification
that does always exist,
but that does not enjoy the property that $\hat{\lie g} \cong \lie g \tensor \mathbb C$.
We will not use universal complexifications in this paper.
\begin{proposition}\label{theo:complexificationExists}
	If $G$ is a real connected closed linear Lie group,
	then it admits a complexification $\hat G$.
\end{proposition}
\begin{proof}
	We may assume that both $G$ and its Lie algebra $\lie g$ are realized by real matrices.
Then the complexification $\hat{\lie g} = \lie g \tensor \mathbb C$
	is a Lie subalgebra of $\liegl[N]$.
	We can use the exponential map to construct an immersed complex Lie subgroup
	$\hat G$ of $\GL[N]$ containing $G$ as a subgroup and 
	having $\hat{\lie g}$ as Lie algebra,
	see e.g.~\cite[Chapter 5.9]{hall:2003a}.
Since $G$ is a closed subgroup of $\GL[N]$, 
	it is also a closed subgroup of $\hat G$.
\end{proof}
Note that we did not claim that $\hat G$ is a closed subgroup of $\GL[N]$.
For semisimple Lie groups this follows automatically from \autoref{theo:coadjointOrbitsAreClosed}.

\begin{lemma} \label{theo:complexifications:embeddingOfRealLieGroup}
	Let $G$ be a real connected Lie group with complexification $\hat G$ and let $\orb\lambda$
	be a coadjoint orbit of $G$ with complexification $\orbext\lambda$.
	Then $\orbext\lambda$ is a coadjoint orbit of $\hat G$ and the embedding 
	$\iota \colon G \to \hat G$ descends to an embedding 
	$\orb\lambda \cong G / G_\lambda \to \hat G / \hat G_\lambda \cong \orbext\lambda$.
\end{lemma}

\begin{proof}
	Since $\hat G$ is connected and has the Lie algebra $\hat{\lie g}$,
	it follows from \autoref{theo:coadjointOrbitsAreSymplecticLeaves}
	that its coadjoint orbit through $\lambda$ is $\smash{\orbext\lambda}$.
	We identify $G$ with a subgroup of $\hat G$.
	Since the coadjoint action of $\hat G$ on $\hat{\lie g}$ is holomorphic,
	$\hat G_\lambda$ is a complexification of $G_\lambda = \hat G_\lambda \cap G$.
	So the map $\iota$ descends to a map 
	$\orb\lambda 
	\cong 
	G/G_\lambda \to \hat G /\hat G_\lambda 
	\cong
	\smash{\orbext\lambda}$ that is still injective.
	To see that it is an embedding, note that the actions of $G_\lambda$
	and $\hat G_\lambda$ on $\hat G$ are proper and free,
	so $\hat G$ is a principal $G_\lambda$ resp.\ $\hat G_\lambda$ bundle
	over $\hat G / G_\lambda$ resp.\ $\hat G / \hat G_\lambda$.
	This implies first that $G / G_\lambda \to \hat G / G_\lambda$ is still an embedding,
	and then that $G / G_\lambda \to \hat G / \hat G_\lambda$ also is.
\end{proof} 
\subsection{Polynomials and analytic functions}
\label{subsec:analytic}

In this subsection we introduce polynomials $\Pol(\orb\lambda)$
and a certain class of analytic functions $\analytics(\orb\lambda)$
on a real coadjoint orbit $\orb\lambda$.
$\analytics(\orb\lambda)$ consists of restrictions of holomorphic functions
on the complexification.
In analogy to the complex case, $\analytics(\orb\lambda)$ is the completion of $\Pol(\orb\lambda)$ 
with respect to some locally convex topology.

All our polynomials are complex-valued. 
So for a real finite dimensional vector space $V$
we define $\Pol(V)$ to be the unital complex subalgebra of $\Cinfty(V)$
generated by the linear maps. (Remember that $\Cinfty(V)$ consists of
smooth functions $V \to \mathbb C$.)
So $\Pol(V) \cong \Sym(V^*_\mathbb{C})$ where $V^*_{\mathbb C}$
is the complexification of 
$V^* = \lbrace \phi \colon V \to \mathbb R, \text{ $\phi$ linear}\rbrace$.

\begin{definition}[Polynomials] 
	\label{def:polynomials:real}
	Let $\orb\lambda$ be a coadjoint orbit of a real connected Lie group $G$
	with Lie algebra $\lie g$.
	Then 
	\begin{equation}
	\Pol(\orb\lambda) 
	= \lbrace 
	p \colon \orb\lambda \to \mathbb C 
	\mid 
	p = P \at {\orb\lambda}\text{ for some polynomial $P$ on $\lie g^*$}
	\rbrace
	\end{equation}
	is called the algebra of polynomials on $\orb\lambda$.
\end{definition}
Note that polynomials on a complex orbit $\orbext\lambda$ were assumed to be holomorphic
and do therefore not coincide with polynomials on the underlying real orbit.
We will always use holomorphic polynomials on complexifications, so this will hopefully
not cause any confusion.

Denote the ideal of polynomials on $\lie g^*$ resp.\ $\hat{\lie g}^*$ 
vanishing on $\orb\lambda$ resp.\ $\orbext\lambda$ 
by $\vanideal{\orb\lambda}$ resp.\ $\vanideal{\orbext\lambda}$.
It is clear that the maps 
$\Pol(\lie g^*) / \vanideal{\orb\lambda} \to \Pol(\orb\lambda)$ 
and 
$\Pol(\hat{\lie g}^*) / \vanideal{\orbext\lambda} \to \Pol(\orbext\lambda)$
are isomorphisms.
We would now like to relate polynomials on
$\orb\lambda$ and $\orbext\lambda$.

\begin{proposition}\label{theo:polynomials:restrictionIsIso:orbit}
	Let $\orb\lambda \subseteq\lie g^*$ be a real coadjoint orbit 
	with complexification $\orbext\lambda \subseteq \hat{\lie g}^*$.
	Then the restriction map 
	$\smash{(\argumentdot)\at{\orb\lambda}} 
	\colon 
	\Cinfty(\orbext\lambda) \to \Cinfty(\orb\lambda)$ 
	restricts to an isomorphism
	$\smash{(\argumentdot)\at{\orb\lambda}} 
	\colon
	\Pol(\orbext\lambda) \to \Pol(\orb\lambda)$.
\end{proposition}

\begin{proof}
	Since restriction to $V$ is a bijection between complex linear maps
	$V \tensor \mathbb C \to \mathbb C$ and real linear maps $V \to \mathbb C$
	for any finite dimensional real vector space $V$,
	it follows that the restriction map $\Pol(\hat{\lie g}^*) \to \Pol(\lie g^*)$ is an isomorphism.
	If we can prove that the restriction map 
	$\vanideal{\orbext\lambda} \to \vanideal{\orb\lambda}$
	is also an isomorphism, then we are done since 
	$\Pol(\orbext\lambda)
	\cong \Pol(\hat {\lie g}^*) / \vanideal{\orbext\lambda}
	\to \Pol(\lie g^*) / \vanideal{\orb\lambda}
	\cong \Pol(\orb\lambda)$
	would be an isomorphism.
	
	Since any map vanishing on $\orbext\lambda$ vanishes
	in particular on $\orb\lambda \subseteq \orbext\lambda$,
	the restriction map $\vanideal{\orbext\lambda} \to \vanideal{\orb\lambda}$
	is well-defined and it is injective since it is the
	restriction of the injective map $\Pol (\hat{\lie g}^*) \to \Pol(\lie g^*)$.
So we only need to prove surjectivity, meaning that
	if a polynomial $p$ on $\lie g^*$ vanishes on $\orb\lambda$,
	then its unique extension to a polynomial
	$\hat p$ on $\hat{\lie g}^*$ vanishes on $\orbext\lambda$.
Since $\orbext\lambda$ is a complex submanifold of $\hat{\lie g}^*$,
	the restriction of $\hat p$ to $\orbext\lambda$ is holomorphic.
	As such it is determined by its derivatives	(of all orders)	at $\lambda$.
	It is even determined by its derivatives 
	in the direction of $\Tangent_\lambda \orb\lambda$ since
	$\Tangent_\lambda \orbext\lambda$ is the complexification of
	$\Tangent_\lambda \orb\lambda$.
	But all these derivatives vanish since the restriction of
	$\hat p$ to $\orb\lambda$ vanishes.
\end{proof}

\begin{definition} \label{def:realPolynomials}
	Let $G$ be a linear real Lie group. 
	Its algebra of polynomials $\Pol(G)$ is the unital
	complex	subalgebra of $\Cinfty(G)$ generated by the functions 
	$P_{ij} \colon G \to \mathbb C$, $g \mapsto g_{ij}$.
\end{definition}
In contrast to the complex case, the algebra of polynomials $\Pol(G)$ may depend on the
way in which $G$ is realized as a linear group, even in the semisimple case.
We will give an instructive example after stating the following proposition,
which can be proven in much the same way as \autoref{theo:polynomials:restrictionIsIso:orbit}.

\begin{proposition}\label{theo:polynomials:restrictionIsIso:group}
	Let $G \subseteq \GLR[N]$ be a linear connected Lie group 
	with complexification $\hat G  \subseteq \GL[N]$.
	Then the restriction map 
	$\smash{(\argumentdot) \at{G}} \colon \Cinfty(\hat G) \to \Cinfty(G)$
	restricts to an isomorphism 
	$\smash{(\argumentdot) \at{G}} \colon \Pol(\hat G) \to \Pol(G)$. 
\end{proposition}
The reason why the algebra of polynomials $\Pol(G)$ may depend on the linear structure of $G$,
is essentially that $G$ may not have a unique complexification.
Consider the linear semisimple Lie group $\SLR[3] \subseteq \GLR[3]$,
which has $\SL[3]$ as a complexification.
The images of $\SLR[3]$ and $\SL[3]$ under $\Ad$ are again semisimple Lie groups.
Furthermore, $\Ad(\SLR[3]) \cong \SLR[3]$
since $\mathrm{SL}(3, \mathbb R)$ has trivial center, and 
$\Ad(\SL[3]) \cong \SL[3] / \lbrace 1, \E^{2\pi \I/3}, \E^{4\pi \I/3} \rbrace$
is a complexification of $\Ad(\SLR[3])$.
By the previous proposition
$\Pol(\Ad(\SLR[3])) 
\cong 
\Pol(\SL[3] / \lbrace 1, \E^{2\pi \I/3}, \E^{4\pi \I/3} \rbrace)
\to
\Pol(\SL[3])
\cong \Pol(\SLR[3])$
where the map in the middle is not surjective, 
since there are polynomials on $\SL[3]$ 
that are not constant on $\lbrace 1, \E^{2\pi \I/3}, \E^{4\pi \I/3}\rbrace$.

We denote the inverses of the isomorphisms in \autoref{theo:polynomials:restrictionIsIso:orbit}
and \autoref{theo:polynomials:restrictionIsIso:group} by
\begin{equation}
\hat{\argumentdot} \colon \Pol(\orb\lambda) \to \Pol(\orbext\lambda) 
\quad\text{and}\quad
\hat{\argumentdot} \colon \Pol(G) \to \Pol(\hat G) \punkt
\end{equation}

\begin{lemma} \label{theo:realInvarianceImpliesComplexInvariance}
	Let $G$ be a real connected linear Lie group with complexification $\hat G$,
	and let $\lambda \in \lie g^*$ be such that $\hat G_\lambda$ is connected.
	If $f \in \Pol(\hat G)$ satisfies $f \at{G} \in \Pol(G)^{G_\lambda}$ 
	then $f \in \smash{\Pol(\hat G)^{\hat G_\lambda}}$.
\end{lemma}

\begin{proof}
	Let $f$ be as in the statement of the lemma.
	Since $f \at {G} = (g \acts f) \at{G}$ holds for all $g\in G_\lambda$ 
	it follows from the injectivity of $\smash{(\argumentdot)\at{G}}$ 
	that $f = g \acts f$, 
	so $f\in \smash{\Pol(\hat G)^{G_\lambda}}$. 
	Therefore $f$ is in particular invariant under $\lie g_\lambda$, 
	thus also under $\hat {\lie g}_\lambda$ since the action is holomorphic.
	Since $\hat G_\lambda$ is connected 
	we obtain that $f$ is $\hat G_\lambda$-invariant.
\end{proof}

\begin{corollary}\label{theo:polynomials:restrictionIsIso:groupInv}
	Let $G$ be a real connected semisimple linear Lie group
	with complexification $\hat G$,
	and assume that $\lambda \in \lie g^*$ is semisimple.
	In this case the restriction map
	$\smash{(\argumentdot) \at{G}} 
	\colon
	\smash{\Pol(\hat G)^{\hat G_\lambda}} \to \Pol(G)^{G_\lambda}$
	is an isomorphism.
\end{corollary}

\begin{proof}
	Note that $G_\lambda$ is connected by \autoref{theo:stabilizerConnected}.
	So this is an immediate consequence of
	\autoref{theo:polynomials:restrictionIsIso:group}
	and \autoref{theo:realInvarianceImpliesComplexInvariance}.
\end{proof}

\begin{corollary}\label{theo:hatIsAnIso:real}
	Let $G$ be a real connected semisimple linear Lie group with complexification $\hat G$,
	and assume that $\lambda \in \lie g^*$ is semisimple.
	Then the map 
	$\pi^* \colon \Pol(\orb\lambda) \to \Pol(G)^{G_\lambda}$
	is an isomorphism.
\end{corollary}

\begin{proof}
	The composition
	$\Pol(\orb\lambda)    
	\xrightarrow{\hat{\argumentdot}}
	\Pol(\orbext\lambda) 
	\xrightarrow{\hat\pi^*} 
	\Pol(\hat G)^{\hat G_\lambda} 
	\xrightarrow{(\argumentdot) \mid_{G}} 
	\Pol(G)^{G_\lambda}$
	equals 
	$\pi^*$ and is an isomorphism because of 
	\autoref{theo:polynomials:restrictionIsIso:orbit},
	\autoref{theo:pullbackOfPolynomials}, and
	\autoref{theo:polynomials:restrictionIsIso:groupInv}.
\end{proof}

\begin{corollary}\label{theo:extAndHatIsomorphisms:polynomials}
	Let $G$ be a real connected semisimple linear Lie group with complexification $\hat G$,
	and assume that $\lambda \in \lie g^*$ is semisimple.
	Then the following diagram commutes and all arrows are isomorphisms:
	\begin{equation}
	\begin{tikzcd}
	\Pol(\hat G)^{\hat G_\lambda}
	\arrow[shift left]{r}{\invpihat} 
	\arrow[shift left]{d}{(\cdot)\mid_{G}} &
	\Pol(\orbext\lambda) 
	\arrow[shift left]{d}{(\cdot)\mid_{\orb\lambda}}
	\arrow[shift left]{l}{\hat\pi^*}
	\\
	\Pol(G)^{G_\lambda}
	\arrow[shift left]{u}{\hat{\argumentdot}}
	\arrow[shift left]{r}{\invpi} &
	\Pol(\orb\lambda) \,\,\,\,\,\,\,.\!\!\!\!\!\!\!\!\!
	\arrow[shift left]{l}{\pi^*}
	\arrow[shift left]{u}{\hat{\argumentdot}}
	\end{tikzcd}
	\end{equation}
\end{corollary}
Next, we want to introduce a class of analytic functions,
that becomes the closure of the polynomials
with respect to a certain locally convex topology.
To this end, assume that $\orb\lambda$ is a coadjoint orbit
with complexification $\orbext\lambda$,
and that $G$ is a real connected Lie group
with complexification $\hat G$.
Then define
\begin{align}
\analytics(\orb\lambda) &= \mathrm{im}\left( (\argumentdot) \at{\orb\lambda} 
\colon 
\Hol(\orbext\lambda) 
\to \Cinfty(\orb\lambda) \right)
\shortintertext{and}
\analytics(G) &= \mathrm{im}\left( (\argumentdot) \at{G} 
\colon
\Hol(\hat G) 
\to \Cinfty(G) \right)\punkt
\end{align}
Note that an element $f \in \analytics(\orb\lambda)$
determines a unique element $\hat f \in \Hol(\orbext\lambda)$:
Existence follows by definition of $\analytics(\orb\lambda)$ and 
$\hat f$ is determined by all its derivatives at $\lambda$.
Since the complexification of $\Tangent_\lambda \orb\lambda$ is just
$\Tangent_\lambda \orbext\lambda$,
see \autoref{theo:complexifications:embeddingOfRealLieGroup},
it suffices to take derivatives in the direction of $\Tangent_\lambda \orb\lambda$.
But these derivatives are determined by $f$.
A similar reasoning holds for $G$ and $\hat G$.
We obtain a commuting square that is similar to the square for polynomials
obtained in \autoref{theo:extAndHatIsomorphisms:polynomials}.

\begin{proposition} \label{theo:commutativeSquare:holomorphic}
	The following diagram is commutative and all arrows are isomorphisms:
	\begin{equation}
	\begin{tikzcd}
	\Hol(\hat G)^{\hat G_\lambda}
	\arrow[shift left]{r}{\invpihat} 
	\arrow[shift left]{d}{(\cdot)\at{G}} &
	\Hol(\orbext\lambda) 
	\arrow[shift left]{d}{(\cdot)\at{\orb\lambda}}
	\arrow[shift left]{l}{\hat\pi^*}
	\\
	\analytics(G)^{G_\lambda}
	\arrow[shift left]{u}{\hat\argumentdot}
	\arrow[shift left]{r}{\invpi} &
	\analytics(\orb\lambda) \,\,\,\,\,\,\,.\!\!\!\!\!\!\!\!\!
	\arrow[shift left]{l}{\pi^*}
	\arrow[shift left]{u}{\hat\argumentdot}
	\end{tikzcd}
	\end{equation}
\end{proposition}
\begin{proof}
	We know from \autoref{subsec:generalities} that 
	$\hat\pi^* \colon \Hol(\orbext\lambda) \to \Hol(\hat G)^{\hat G_\lambda}$
	is an isomorphism. In the previous paragraph we explained that
	$\hat\argumentdot \colon \analytics(\orb\lambda) \to \smash{\Hol(\orbext\lambda)}$
	and	$\hat\argumentdot \colon \analytics(G) \to \Hol(\hat G)$
	are isomorphisms and as in 
	\autoref{theo:realInvarianceImpliesComplexInvariance}
	it follows that the same is true for 
	$\hat\argumentdot \colon
	\analytics(G)^{G_\lambda} \to \smash{\Hol(\hat G)^{\hat G_\lambda}}$.
	Composing these isomorphisms we obtain that
	$\pi^* \colon \analytics(\orb\lambda) \to \analytics(G)^{G_\lambda}$
	is an isomorphism.
\end{proof}
Since $\Pol(\orbext\lambda) \subseteq \Hol(\orbext\lambda)$ it follows
that $\Pol(\orb\lambda) \subseteq \analytics(\orb\lambda)$.
We can define a topology $\smash{\mathcal T_{\widehat{\mathrm{lu}}}}$ of
\emph{extended locally uniform convergence} on $\analytics(\orb\lambda)$
as follows:
A sequence $f_n \in \analytics(\orb\lambda)$ converges to some
$f \in \analytics(\orb\lambda)$ if and only if the sequence
$\smash{\hat f_n} \in \smash{\Hol(\orbext\lambda)}$ converges locally uniformly
to $\hat f \in \Hol(\orbext\lambda)$.
Clearly the maps 
$\hat\argumentdot \colon \analytics(\orb\lambda) \to \Hol(\orbext\lambda)$ and 
$(\argumentdot)\at{\orb\lambda} \colon
\Hol(\orbext\lambda) \to \analytics (\orb\lambda)$
are both homeomorphisms. From \autoref{theo:topologies:polCompletesToHol}
it follows that the closure of $\Pol(\orb\lambda)$ with respect to the topology 
of extended locally uniform convergence is $\analytics(\orb\lambda)$.  
 
\subsection{Formal and strict star products on real coadjoint orbits}
\label{subsec:realOrbits}

In a sense all constructions in \autoref{sec:preliminaries}, 
\autoref{sec:starProduct}, and \autoref{sec:continuity}
are compatible with the restriction to real forms.
In this subsection we want to make this statement precise.
In particular, we will show that we can restrict formal and strict 
products from a complexification $\orbext{\lambda}$ of a semisimple
coadjoint orbit $\orb\lambda$ of a real connected semisimple Lie group $G$
to formal and strict star products on $\orb\lambda$. These star products
can---as before---be computed by applying fundamental vector fields or
by passing to the Lie group by using the maps $\pi^*$ and $\pi_*$.
We will determine when the star products on $\orb\lambda$ are
of (pseudo) Wick type or of standard ordered type.

\begin{proposition} \label{theo:starProduct:real:construction}
	Let $\orb \lambda$ be a semisimple coadjoint orbit
	of a semisimple connected real Lie group $G$.
	By \autoref{theo:complexifications:embeddingOfRealLieGroup}
	it has a complexification $\orbext\lambda$,
	and for $\hbar \in \mathbb C \setminus P_\lambda$
	there are strict products 
	$\hat*_\hbar\colon 
	\Pol(\orbext\lambda) \times 
	\Pol(\orbext\lambda) \to 
	\Pol(\orbext\lambda)$ 
	with extensions
	$\hat*_\hbar \colon 
	\Hol(\orbext\lambda) \times 
	\Hol(\orbext\lambda) \to 
	\Hol(\orbext\lambda)$
	constructed in \autoref{theo:starProduct:strict:associative}
	and \autoref{theo:mainContinuityTheorem}.
	These products restrict to $G$-invariant strict products
	\begin{equation}
	*_\hbar\colon 
	\Pol(\orb\lambda) \times 
	\Pol(\orb\lambda) \to 
	\Pol(\orb\lambda)
	\quad\text{ and }\quad
	*_\hbar\colon
	\analytics(\orb\lambda) \times 
	\analytics(\orb\lambda) \to 
	\analytics(\orb\lambda)
	\end{equation}
	for all $\hbar \in \mathbb C \setminus P_\lambda$.
	For fixed $p, q \in \Pol(\orb\lambda)$, 
	the dependence of $p *_\hbar q$ on $\hbar$ is rational with no pole at zero,
	and for fixed $f, g \in \analytics(\orb\lambda)$ and $x \in \orb\lambda$,
	the dependence of $f *_\hbar g(x)$ on $\hbar$ is holomorphic.
	Both products are continuous with respect to the topology of 
	extended locally uniform convergence defined at the end of
	\autoref{subsec:analytic}.
\end{proposition}

\begin{proof}
	Since the restriction maps $\Pol(\orbext\lambda) \to \Pol(\orb\lambda)$
	and $\Hol(\orbext\lambda) \to \analytics(\orb\lambda)$ are both homeomorphisms 
	(with respect to the topology of locally uniform convergence on the domains
	and the topology of extended locally uniform convergence on the codomains),
	the statement follows trivially from the corresponding statements for $\hat*_\hbar$,
	obtained in \autoref{theo:starProduct:strict:associative},
	\autoref{theo:mainContinuityTheorem}, and
	\autoref{theo:holomorphicDependenceOnHbar:complex}.
\end{proof}
We would like to compute these star products without passing to the complexification.
The construction of bidifferential operators from \autoref{subsec:diffopsOnHomogeneousSpaces}
works completely similarly in the real setting.
Recall that our differential operators act on complex-valued functions,
and therefore any complex vector field $\Secinfty(\Tangent^{\mathbb C}M)$
defines a first order differential operator on $M$.

\begin{proposition}
	Let $G$ be a real Lie group with Lie algebra $\lie g$, 
	and let $\hat{\lie g}$ be the complexification of $\lie g$.
	The map 
	\begin{equation*}
	\leftinv{(\argumentdot)} \colon (\univ \hat{\lie g})^{\tensor k} \to \nDiffop^G(G)
	\end{equation*}
	obtained by extending 
	$\hat{\lie g} \ni X \mapsto \leftinv X \in \Secinfty(\Tangent^{\mathbb C} G)$
	to an algebra homomorphism $\univ \hat{\lie g} \to \Diffop^G(G)$ and further
	to tensor products as in \eqref{eq:extendingToTensors} is an isomorphism.
	If $H$ is a closed Lie subgroup of $G$, then the map 
	\begin{equation}
	\Psi 
	\colon 
	((\univ \hat {\lie g} / \univ \hat{\lie g} \cdot \hat{\lie h})^{\tensor k})^H
	\to 
	\nDiffop^G(G/H) 
	\komma \quad 
	\Psi([\vec u])(\vec f) 
	=
	\pi_*(\leftinv{\vec u}(\pi^* \vec f))
	\end{equation}
	is also an isomorphism.
\end{proposition}
\begin{proof}
	With the obvious modifications the proofs of \autoref{theo:leftInvDiffOps:complex}
	and \autoref{theo:diffops:bijection:real} given in \appendixref{appendix:diffops}
	apply also to the real situation.
\end{proof}
To be consistent with the notation of this chapter, we denote the map defined in 
\eqref{eq:diffops:bijection:real} by $\hat\Psi$.

\begin{lemma} \label{theo:relationOfDiffops}
	Let $G$ be a real Lie group with closed subgroup $H$ 
	and assume that the complex Lie group $\hat G$ is a complexification of $G$
	and contains a complex closed subgroup $\hat H$ that is a complexification of $H$.
	The maps $\leftinv{(\argumentdot)}$ and $\Psi$ are compatible with the maps
	$\leftinvholo{(\argumentdot)}$ and $\hat\Psi$ in the sense that the diagrams
	\begin{equation}\label{eq:restrictionscommute:real}
	\begin{tikzcd}
	\Hol(\hat U)^k \arrow[r, "\left.(\cdot)\right|_{U}^{\times k}"] 
	\arrow[d, "{\leftinvholo {\vec u}}", swap] 
	& 
	\Cinfty(U)^k \arrow[d, "\leftinv {\vec u}"]\\
	\Hol(\hat U) \arrow[r, "\left.(\cdot)\right|_{U}", swap] & 
	\Cinfty(U)
	\end{tikzcd}\qquad\text{and}\qquad
	\begin{tikzcd}
	\Hol(\hat V)^k \arrow[r, "\left.(\cdot)\right|_{V}^{\times k}"] 
	\arrow[swap]{d}{\hat\Psi ([\vec v])} 
	& 
	\Cinfty(V)^k \arrow{d}{\Psi([\vec v])}\\
	\Hol(\hat V) \arrow[r, "\left.(\cdot)\right|_{V}", swap] & 
	\Cinfty(V)
	\end{tikzcd}
	\end{equation}
	commute for all open subsets $\hat U \subseteq \hat G$ and $\hat V \subseteq \hat G / \hat H$,
	with $U \coloneqq \hat U \cap G$ and $V \coloneqq \hat V \cap G / H$,
	and all elements
	$\vec u \in (\univ \hat{\lie g})^{\tensor k}$ and
	$\vec v \in ((\univ \hat {\lie g} / \univ \hat{\lie g} \cdot \hat{\lie h})^{\tensor k})^{\hat H}$.
\end{lemma}
\begin{proof}
	The commutativity of the second diagram follows easily from commutativity of the first,
	since the restrictions are compatible with $\pi^*$ and $\invpi$.
	To prove commutativity of the first diagram, 
	assume that $k=1$ and $\vec u = X \in \hat{\lie g} \subseteq \univ\hat{\lie g}$.
	The tangent map of a holomorphic function commutes with 
	the multiplication by $\I$.
	We compute 
	\begin{multline*}
	\leftinvholo X f (g) 
	= \frac 1 2 \left( \Tangent_g f \circ \Tangent_e \leftact_g (X) 
	- \I \Tangent_g f \circ \Tangent_e \leftact_g (\I X) \right) = \\
	= \Tangent_g f \circ \Tangent_e \leftact_g (X)
	= \leftinv X f \at{U}(g) 
	\end{multline*}
	for $f \in \Hol(\hat U)$ and $g \in U$. 
	The general case follows from this computation
	by the way in which $\leftinvholo{(\argumentdot)}$ and $\leftinv{(\argumentdot)}$ 
	are extended to $(\univ\hat{\lie g})^{\tensor k}$.
\end{proof}
\begin{corollary}
	Let $\orb \lambda$ be a semisimple coadjoint orbit
	of a semisimple connected real Lie group $G$.
	For $\hbar \in \mathbb C \setminus P_\lambda$ and $p,q\in \Pol(\orb\lambda)$, 
	the product $*_\hbar \colon \Pol(\orb\lambda) \times \Pol(\orb\lambda) \to \Pol(\orb\lambda)$
	defined in \autoref{theo:starProduct:real:construction} can be computed by
	\begin{equation}
	p *_\hbar q = \sum_{\ell = 0}^\infty \Psi(F_{\hbar,\ell})(p,q) \punkt
	\end{equation}
\end{corollary}
\begin{proof}
	The previous lemma implies
	\begin{equation*}
	p *_\hbar q 
	= 
	(\hat p \starhat \hat q) \at{\orb\lambda} 
	=
	\sum_{\ell = 0}^\infty \hat\Psi(F_{\hbar,\ell})(\hat p, \hat q) \at{\orb\lambda}
	=
	\sum_{\ell = 0}^\infty \Psi(F_{\hbar,\ell})(p,q) \punkt
	\end{equation*}
	Note that the sum over $\ell$ is finite by 
	\autoref{theo:starWellDefOnPolynomials}.
\end{proof}

\begin{theorem}
	Let $\orb \lambda$ be a semisimple coadjoint orbit
	of a semisimple connected real Lie group $G$.
	The product $\star \colon 
	\Cinfty(\orb\lambda)\formal\formParam \times 
	\Cinfty(\orb\lambda)\formal\formParam 
	\to
	\Cinfty(\orb\lambda)\formal\formParam$
	defined by $f \star g = \Psi(F)(f,g)$ where $F$ was obtained
	in \autoref{theo:starproduct:alekseev} is a $G$-invariant formal star product.
	In particular, it is associative and 
	deforms the KKS symplectic form on $\orb\lambda$.
	Furthermore, $p \star q$ coincides with the formal power series expansion
	of $p *_\hbar q$ around $\hbar = 0$ for $p, q \in \Pol(\orb\lambda)$,
	and $f \star g = \hat f \formalstarhat \hat g \at{\orb\lambda}$
	for $f, g \in \analytics(\orb\lambda)$.
\end{theorem}
\begin{proof}
	It is immediate from the definition of $F$ and $\Psi$ that every order of $\star$
	is given by a $G$-invariant bidifferential operator. 
Since $F$ is the formal power series expansion of $F_\hbar$ around $\hbar = 0$
	and $p *_\hbar q$ is rational with no pole at $0$ for $p, q \in \Pol(\orb\lambda)$,
	it follows that $p \star q$ coincides 
	with the formal power series expansion of $p *_\hbar q$.
	The compatibility with $\formalstarhat$ is immediate from 
	\autoref{theo:relationOfDiffops}.
Since bidifferential operators are uniquely determined by their
	behaviour on $\Pol(\orb\lambda) \subseteq \analytics(\orb\lambda)$,
	the compatibility with $\hat \star$ implies that $\star$ is associative	and,
	using \autoref{theo:firstOrderIsKKS}, that it deforms the KKS symplectic form.
\end{proof}
Recall that we proved in \autoref{theo:distributions:complex} that the product
$\hat*_\hbar$ separates variables with respect to the distributions $L_+$ and $L_-$,
which we call $\hat L_+$ and $\hat L_-$ in this section.
In the real case,
those distributions may have further properties.
They can be real,
or the holomorphic and antiholomorphic tangent spaces with respect to a complex structure.
Before giving further details let us make the following definitions.

\begin{definition}[Star products of standard ordered type]
	\label{def:starProduct:standardOrderedType} A star product 
	$*_\hbar$ on a symplectic manifold $M$ is said to be of 
	\emph{standard 	ordered type} 
	if there are two Lagrangian distributions $L_1, L_2 \subseteq \Tangent M$ 
	spanning the real tangent bundle $\Tangent M$ of $M$ 
	such that 
	the first argument of the star product is derived only in directions of $L_1$ 
	and the second argument only in directions of $L_2$.
\end{definition}

\begin{definition}[Star products of (pseudo) Wick type] A star product 
	$*_\hbar$ on a complex manifold $M$ that is also symplectic is said 
	to be of \emph{pseudo Wick type} if the first argument is derived only in 
	holomorphic directions and the second argument only in antiholomorphic 
	directions. A star product of pseudo Wick type on a 
	Kähler manifold is said to be of \emph{Wick type}.
\end{definition}
For formal star products of Wick type and with respect to the usual ${}^*$-involution
given by complex conjugation, point evaluations are positive linear functionals,
which is not necessarily the case for formal star products of pseudo Wick type. 
Note that the situation is more complicated for strict star products,
as we shall see in \autoref{subsec:positiveLinearFunctionals}.

Let us briefly recall some results 
on the existence of invariant complex structures on coadjoint orbits.
See \appendixref{appendix:complexstructures} for more details.
Let $\orb\lambda$ be a semisimple coadjoint orbit 
of a real connected semisimple Lie group $G$
with Lie algebra $\lie g$, and assume that $G_\lambda$ is compact.
Choose a real Cartan subalgebra $\lie h$ containing $\lambda^\sharp$.
Since $\lie h \subseteq \lie g_\lambda$, it follows that $\lie h$ is compact
(meaning that it integrates to a subgroup of $G$ with compact closure).
Then there are $G$-invariant complex structures on $\orb\lambda$,
and these structures are in bijection to invariant orderings of $\hat\Delta$ 
(we say an ordering on $\hat\Delta$ is invariant if it is the restriction of 
an invariant ordering of $\Delta$ as defined in \autoref{def:invariantOrdering})
as follows.
Recall that 
$\Tangent_\lambda^{\mathbb C} \orb\lambda 
\cong 
\hat{\lie g} / \hat{\lie g}_\lambda 
\cong
\bigoplus_{\alpha \in \hat\Delta} \lie g^\alpha$.
So given an invariant ordering we can define a map
$I_\lambda \colon 
\Tangent_\lambda^{\mathbb C} \orb\lambda 
\to 
\Tangent_\lambda^{\mathbb C} \orb\lambda$
by letting 
$I_\lambda X_\alpha = \I X_\alpha$ if $\alpha \in \smash{\hat\Delta^+}$, and 
$I_\lambda X_\alpha = -\I X_\alpha$ if $\alpha \in \smash{\hat\Delta^-}$.
The map $I_\lambda$ extends $G$-invariantly to an endomorphism $I$
of the complexified tangent bundle $\Tangent^{\mathbb C} \orb\lambda$ 
and restricts to an endomorphism of the real 
tangent bundle $\Tangent\orb\lambda$, thus it defines a complex structure.

If $G$ is compact, there is a unique ordering that makes $\orb\lambda$ with the 
complex structure $I$ and the KKS symplectic form $\omega_{\mathrm{KKS}}$ a 
K\"ahler manifold.
This ordering is characterized by $\alpha \in \smash{\hat\Delta}$ being positive
iff $(\alpha, \I\lambda) > 0$. In particular it is standard.
See \appendixref{appendix:complexstructures} for more details.
\begin{proposition} \label{theo:starProduct:types}
	For a semisimple coadjoint orbit $\orb\lambda$ of a real connected
	semisimple linear Lie group $G$, the product $*_\hbar$ obtained in 
	\autoref{theo:starProduct:real:construction}
	\begin{propositionlist}
		\item
		has poles $P_\lambda \subseteq \mathbb R$ if $\lie h$ is compact,
		\item\label{item:starProduct:types:i} 
		is of pseudo Wick type if $G_\lambda$ is compact 
		and the same ordering is used in the construction of the star product
		and the definition of the complex structure,
		\item\label{item:starProduct:types:ii} 
		is of standard ordered type with poles $P_\lambda \subseteq \I \mathbb R$
		if $\I \lie h \subseteq \hat{\lie g}$ is compact.
	\end{propositionlist}
	In particular, if $G$ is compact and, in the construction of $*_\hbar$,
	one chooses the ordering that makes $\orb\lambda$ with the induced complex structure $I$
	a K\"ahler manifold, then $*_\hbar$ is of Wick type.
\end{proposition}
\begin{proof}
	Roots take purely imaginary values on a compact Lie subalgebra of $\lie h$.
	Since $\lambda \in \lie g^*$ is by definition 
	real on $\lie h \subseteq \smash{\hat{\lie h}}$,
	it follows that 
	$(\lambda, \mu) \in \I \mathbb R$ if $\lie h$ is compact and 
	$(\lambda, \mu) \in \mathbb R$ if $\I \lie h$ is compact.
	Since $\frac 1 2 (\mu,\mu) - (\rho,\mu) \in \mathbb R$,
	this implies that the roots (with respect to $\hbar$) of
	$p_{\I\lambda/\hbar}(\mu) 
	=
	\frac 1 2 (\mu,\mu) - (\rho,\mu) - \frac \I \hbar (\lambda,\mu)$ are
	real if $\lie h$ is compact and 
	purely imaginary if $\I\lie h$ is compact.
	
	Recall the definition of the distributions $L_+$ and $L_-$,
	which we denote by $\hat L_+$ and $\hat L_-$ in this section,
	made just after \autoref{theo:distributionSpacesWellDefd}.
	Restricting them to $\orb\lambda \subseteq \orbext\lambda$
	gives two distributions $L_+, L_- \subseteq \Tangent^{\mathbb C} \orb\lambda$
	of the complexified tangent bundle.
	An analogue of \autoref{theo:calculationOfLeftInvDiffOpsOnOrbit}
	in the real case and the explicit formula for $F_\hbar$ from
	\autoref{theo:canonicalelement:general} together with 
	\autoref{remark:leaveOutProjections} show that $\star$ 
	derives the first argument only in directions of $L_+$,
	and the second argument only in directions of $L_-$.
	
	Assume that $\lie g_\lambda$ is compact.
	The holomorphic tangent space $\smash{\Tangent_\lambda^{(1,0)} \orb\lambda}$ is,
	under the isomorphism 
	$\smash{\Tangent_\lambda^{\mathbb C} \orb\lambda} 
	\cong 
	\smash{\hat{\lie g}/\hat{\lie g}_\lambda}$,
	spanned by 
	$X_\alpha - \I I_\lambda X_\alpha$
	for $\alpha \in \hat\Delta$.
	If $I_\lambda$ is defined using the ordering chosen in the construction of $*_\hbar$
	as described above,
	then $X_\alpha - \I I_\lambda X_\alpha = X_\alpha - \I \cdot \I X_\alpha = 2 X_\alpha$ 
	if $\alpha \in \hat\Delta^+$,
	and $X_\alpha - \I I_\lambda X_\alpha = X_\alpha - \I \cdot (- \I) X_\alpha = 0$
	if $\alpha \in \hat\Delta^-$,
	so 
	$\smash{\Tangent_\lambda^{(1,0)} \orb\lambda} 
	=
	\smash{\spann\lbrace 
		(X_\alpha)_{\orb\lambda} \at{} {}_\lambda,
		\alpha \in \hat\Delta^+ 
	\rbrace}$.
	This coincides exactly with $\smash{L_+ \at\lambda}$, and by $G$-invariance
	it follows that $L_+$ coincides with $\Tangent^{(1,0)} \orb\lambda$.
	Similarly, $L_-$ coincides with $\Tangent^{(0,1)} \orb\lambda$.
	Therefore $\star$ is of pseudo Wick type.	
	
	If $\I \lie h$ is compact, then every $\ad_H$ for $H \in \lie h$ is self-adjoint.
	Since they are all commuting we can find simultaneous eigenvectors in $\lie g$ of all $\ad_H$ 
	(without complexifying $\lie g$). But then we can pick $X_\alpha$ and $Y_\alpha$ to lie 
	in $\lie g$ so that $L_1 = L_+ \cap \lie g$ and $L_2 = L_- \cap \lie g$ are 
	Lagrangian distributions satisfying 
	\autoref{def:starProduct:standardOrderedType}.
\end{proof}

\begin{remark}
	Assume that $\lie g_\lambda$ is compact as in part 
	\refitem{item:starProduct:types:i} of the previous proposition.
	If one uses different invariant orderings in the construction of 
	the star product and in the definition of a complex structure,
	then the distributions $L_+$ and $L_-$ may both contain 
	holomorphic and antiholomorphic directions.
	Since we are mainly interested in star products of (pseudo) Wick type
	(these are the ones for which we would hope to 
	find positive linear functionals on the star product algebra,
	see \autoref{subsec:positiveLinearFunctionals}),
	we will usually assume that the two orderings agree. 
\end{remark} 
\subsection{Examples: complex projective spaces and hyperbolic discs}
\label{subsec:realExamples}

Recall that we have computed the canonical element of the Shapovalov pairing for 
$\SL[1+n]$ and a certain choice of $\lambda$ in \autoref{subsec:examples}.
Let us now specialize this result to the real forms $\group{SU}(1+n)$ and $\group{SU}(1,n)$.

\begin{example}[\boldmath$\mathbb{CP}^n$]\label{ex:CPn}
	The coadjoint orbit of $\group{SU}(1+n)$ through 
	$\lambda \colon \lie{su}_{1+n} \to \mathbb R$, $X \mapsto -\I r X_{0,0}$
	with $r \in \mathbb R^+$ is the complex projective space $\mathbb{CP}^n$.
	$\SL[1+n]$ is a complexification of $\group{SU}(1+n)$.
Using the notation $\hat{\lie h}$ for the Cartan subalgebra
	of $\liesl[1+n]$ introduced in \autoref{subsec:examples},
	we obtain a compact Cartan subalgebra  
	$\lie h \coloneqq \lie{su}_{1+n} \cap \hat{\lie h}$
	of $\lie{su}_{1+n}$.
\autoref{theo:complexStructure:Kaehler} tells us that the 
	K\"ahler complex structure is defined by the ordering of $\hat\Delta$
	for which $\alpha \in \hat\Delta^+$ iff $(\I\lambda, \alpha) > 0$.
	This ordering is the restriction of the ordering on $\Delta$ 
	for which all $\alpha_{i,j}$ with $i < j$ are positive.
Therefore the element $F_\hbar$ from \autoref{theo:example:standardOrdering}
	induces a Wick type star product on $\mathbb{CP}^n$.
This product has poles at $\lbrace \frac 1 n r \mid n \in \mathbb N\rbrace$.
\end{example}
\begin{example}[\boldmath$\mathbb D^n$]\label{ex:Dn}
	Denote the complex hyperbolic disc in $n$ dimensions by $\mathbb D^n$.
	Recall that $\group{SU}(1,n)$ denotes the group of isometries of the indefinite
	scalar product $g(v,w) = - v_0 w_0 + \sum_{i=1}^n v_i w_i$ on $\mathbb R^{1+n}$.
	The coadjoint orbit of $\group{SU}(1,n)$ through
	$\lambda \colon \lie{su}_{1,n} \to \mathbb R$, $X \mapsto -\I r X_{0,0}$
	with $r \in \mathbb R^+$ is the hyperbolic disc $\mathbb D^n$. 
	$\SL[1+n]$ is a complexification of $\group{SU}(1,n)$. 
	Again, $\lie h \coloneqq \lie{su}_{1,n} \cap \hat{\lie h}$ defines 
	a compact Cartan subalgebra of $\lie{su}_{1,n}$.
Now all roots are non-compact,
	so that according to \autoref{theo:complexStructure:Kaehler:nonCompact}
	the K\"ahler complex structure is defined by the ordering on $\hat\Delta$
	for which $\alpha \in \hat\Delta^+$ iff $(\I\lambda, \alpha) < 0$.
	This ordering is the restriction of the ordering on $\Delta$
	for which all $\alpha_{i,j}$ with $i > j$ are positive.
Therefore the element $F_\hbar$ from \autoref{theo:example:oppositeOrdering}
	induces a Wick type star product on $\mathbb D^n$.
This product has poles at $\lbrace -\frac 1 n r \mid n \in \mathbb N\rbrace$.
\end{example}
\begin{remark}
	A star product of Wick type on the hyperbolic disc was also studied in 
	\cite{kraus.roth.schoetz.waldmann:2019a}, where it was obtained from 
	a star product of Wick type on $\mathbb C^{1+n}$ using phase space reduction.
	This product coincides with the star product obtained in 
	\autoref{ex:Dn}.
	To see this, one checks that monomials of degree 1 generate the star product algebra,
	so that it suffices to compare the two formulas for a degree 1 monomial
	and an arbitrary monomial. 
	But for a degree 1 monomial only very few summands are non-zero in both 
	constructions and one can explicitly check that the expressions agree.
\end{remark} 
\subsection{Positive linear functionals}
\label{subsec:positiveLinearFunctionals}

In this subsection we prove that for certain coadjoint orbits
and certain values of $\hbar$ the point evaluation functionals
of the star product algebras constructed in \autoref{subsec:realOrbits}
are positive.
In order to have a meaningful notion of positivity we need a star involution
on $(\analytics(\orb\lambda), *_\hbar)$.
Of course, this star involution should be the restriction
of the complex conjugation of $\Cinfty(\orb\lambda)$,
but we need to prove that this restriction is well-defined.

Assume that $\hat{\lie g} = \lie g \tensor \mathbb C$ is the complexification
of a Lie algebra $\lie g$. The complex conjugation 
$\overline {\argumentdot\vphantom{\lambda}} \colon \hat{\lie g} \to \hat{\lie g}$, 
$X \tensor z \mapsto X \tensor \overline z$
is an antilinear involution on $\hat{\lie g}$.
Then 
$\overline{\argumentdot\vphantom{\lambda}} \colon \hat {\lie g}^* \to \hat{\lie g}^*$,
$\phi \mapsto \overline \phi \coloneqq
\overline{\argumentdot\vphantom{\lambda}} {}\circ \phi \circ{} 
\overline{\argumentdot\vphantom{\lambda}}$
defines an antilinear involution on $\hat{\lie g}^*$.
Note that on the right hand side, we first apply the involution of $\hat{\lie g}$,
then $\phi$, and then the complex conjugation of $\mathbb C$.
Therefore the right hand side defines a 
complex linear functional $\overline\phi \in \hat{\lie g}^*$. 
The map $\phi \mapsto \overline \phi$ is antilinear.

\begin{lemma}
	\label{theo:preliminaries:orbitsFixedByComplexConjugation}
	Let $G \subseteq \GLR[N]$ be a real linear Lie group with complexification $\hat G \subseteq \GL[N]$,
	assume $\lambda \in \lie g^*$, and let $\orbext\lambda$ be the coadjoint
	orbit of $\hat G$ through $\lambda$. 
	Then the map $\overline{\argumentdot\vphantom{\lambda}} \colon \hat{\lie g}^* \to \hat{\lie g}^*$
	restricts to an antilinear involution 
	$\overline{\argumentdot\vphantom{\lambda}} \colon \orbext\lambda \to \orbext\lambda$.
\end{lemma}
\begin{proof}
	Note that since $\lambda \in \lie g^*$ we have $\overline \lambda = \lambda$.
	Therefore we compute
	\begin{equation*}
	\overline{\Ad^*_g \lambda} 
	= 
	\overline{\lambda \circ \Ad_{g^{-1}}}
	=
	\overline \lambda \circ \overline{\argumentdot\vphantom{\lambda}} 
	\circ {\Ad_{g^{-1}}} \circ {} \overline{\argumentdot\vphantom{\lambda}}
	=
	\lambda \circ \Ad_{\overline g^{-1}}
	=
	\Ad^*_{\overline g} \lambda \punkt
	\end{equation*}
	Here $\overline g$ denotes the entrywise complex conjugate of $g \in \hat G$.
	Since the exponential map $\hat{\lie g} \to \hat G$ commutes with the complex conjugation,
	it follows that $\hat G$ is closed under entrywise complex conjugation, and therefore
	$\overline{g} \in \hat G$ and $\smash{\Ad_{\overline g}^* \lambda} \in \orbext\lambda$.
	This proves that $\overline{\argumentdot\vphantom{\lambda}}$ restricts to $\orbext\lambda$,
	and the restriction is clearly still an antilinear involution.
\end{proof}
Note that 
$\Tangent_{\xi} \overline{\argumentdot\vphantom{\lambda}} \circ I_{\xi} 
=
(I_{\overline \xi})^{-1} \circ \Tangent_{\xi} \overline{\argumentdot\vphantom{\lambda}}$
holds for $\xi \in \hat{\lie g}^*$, 
where $\Tangent_\xi \overline{\argumentdot\vphantom{\lambda}} \colon
\Tangent_\xi \hat{\lie g}^* \to \Tangent_{\overline{\xi}} \hat{\lie g}^*$ is the tangent map to the complex 
conjugation of $\hat{\lie g}^*$ and $I_\xi \colon \Tangent_\xi \hat{\lie g}^* \to \Tangent_{\xi} \hat{\lie 
g}^*$ is the complex structure at $\xi$.
Since the complex structure $I$ 
and the complex conjugation $\overline{\argumentdot\vphantom{\lambda}}$
of $\orb\lambda$ are both obtained by restriction from $\hat{\lie g}^*$,
they satisfy the same relation.

For any $f \in \Hol(\orbext\lambda)$ consider the function
$f^* \coloneqq 
\overline{\argumentdot\vphantom{\lambda}} \circ f \circ \overline{\argumentdot\vphantom{\lambda}}$, 
where the left $\overline{\argumentdot\vphantom{\lambda}}$ 
is the complex conjugation of $\mathbb C$
and the right $\overline{\argumentdot\vphantom{\lambda}}$ 
is the antilinear involution obtained in the previous lemma.
Denote the complex structure of $\mathbb C$ by $J$, 
and identify the tangent space of $\mathbb C$ with $\mathbb C$. 
Then
\begin{multline*}
\Tangent_\xi f^* \circ I_{\xi} 
= 
\overline{\argumentdot\vphantom{\lambda}} \circ \Tangent_{\overline\xi} f 
	\circ \Tangent_\xi\overline{\argumentdot\vphantom{\lambda}} \circ I_\xi
= 
\overline{\argumentdot\vphantom{\lambda}} \circ \Tangent_{\overline\xi} f 
	\circ I_{\overline\xi}^{-1} \circ \Tangent_\xi \overline{\argumentdot\vphantom{\lambda}}
= \\ =
\overline{\argumentdot\vphantom{\lambda}} \circ J^{-1} \circ \Tangent_{\overline \xi} f 
	\circ \Tangent_\xi \overline{\argumentdot\vphantom{\lambda}} 
=
J \circ \overline{\argumentdot\vphantom{\lambda}} \circ \Tangent_{\overline\xi} f 
	\circ \Tangent_\xi \overline{\argumentdot\vphantom{\lambda}} 
= J \circ \Tangent_\xi f^*
\end{multline*}
shows that $f^*$ is holomorphic.
Since $\overline{\argumentdot\vphantom{\lambda}}$ restricts to the identity
on $\orb\lambda \subseteq \lie g^*$,
it follows that $f^* |_{\orb\lambda} = \overline {f |_{\orb\lambda}}$.
Consequently, the restriction of ${}^* \colon \Hol(\orbext\lambda) \to \Hol(\orbext\lambda)$ 
to $\analytics(\orb\lambda)$ is just the complex conjugation
$\smash{\overline{\argumentdot\vphantom{\lambda}}} 
\colon
\analytics(\orb\lambda) \to \analytics(\orb\lambda)$.
In other words, the complex conjugation is well-defined on $\analytics(\orb\lambda)$.

\begin{proposition}
	Let $\orb\lambda$ be a semisimple coadjoint orbit of a
	connected semisimple real Lie group $G$.
	Assume that the Cartan subalgebra $\lie h$ used in the
	construction of a star product $*_\hbar$ is compact.
Then $\overline{f *_\hbar g} = \overline g *_{\overline\hbar} \overline f$
	holds for all $f, g \in \analytics(\orb\lambda)$.
\end{proposition}
\begin{proof}
	As in the proof of \autoref{theo:starProduct:types} one argues that 
	since $\lie h$ is compact the coefficients $\smash{p_{\I\lambda}^w(\alpha_w)}$ 
	are real and more generally 
	$\smash{\overline{p_{\I\lambda/\hbar}^w(\alpha_w)}} 
	= 
	\smash{p_{\I\lambda/\smash{\overline \hbar}}^w(\alpha_w)}$.
	From \eqref{eq:tangentVectorsToOrbit} we obtain that 
	$\smash{\overline{X_\alpha \tensor Y_\alpha}} 
	= 
	Y_\alpha \tensor X_\alpha
	=
	\tau(X_\alpha \tensor Y_\alpha)$ 
	for both a compact and a non-compact root $\smash{\alpha \in \hat\Delta^+}$,
	and the same formula holds when $\alpha$ is replaced by a word $w \in \tilde\wordset$.
	Here $\smash{\overline{\argumentdot\vphantom{\lambda}}}$ is the complex conjugation
	of $\hat{\lie g}$ with respect to $\lie g$, extended to $(\univ\hat{\lie g})^{\tensor 2}$,
	and $\tau \colon (\univ\hat{\lie g})^{\tensor 2} \to (\univ\hat{\lie g})^{\tensor 2}$ 
	is the flip of the two tensor factors. Note that $\tau$ stays well-defined on 
	$(\univ\hat{\lie g} / \univ \hat{\lie g}\cdot \hat{\lie g}_\lambda)^{\tensor 2}$,
	and therefore the formula for $F_\hbar$ 
	obtained in \autoref{theo:canonicalelement:general}, \autoref{remark:leaveOutProjections},
	and the computations above imply $\overline{F_{\hbar,\ell}} = \tau(F_{\overline \hbar,\ell})$.
	Consequently
	\begin{multline*}
	\overline{f *_\hbar g} 
	= \sum_{\ell=0}^\infty \overline{\Psi(F_{\hbar, \ell})(f,g)} 
	= \sum_{\ell=0}^\infty \Psi(\overline{F_{\hbar, \ell}})(\overline f, \overline g)
	= \\ 
	= \sum_{\ell=0}^\infty \Psi(\tau(F_{\overline \hbar, \ell}))(\overline f, \overline g)
	= \sum_{\ell=0}^\infty \Psi(F_{\overline \hbar, \ell})(\overline g, \overline f)
	= \overline g *_{\overline \hbar} \overline f
	\end{multline*}
	holds for all $f, g \in \Pol(\orb\lambda)$ 
	and extends to $\analytics(\orb\lambda)$ by continuity.
\end{proof}
A linear functional $\phi$ on a $^*$-algebra $\algebra A$ is said to 
be positive if 
$\phi(a^* a) \geq 0$ for all $a \in \algebra A$.
In the following we formulate our results for the star algebra 
$\algebra A_\hbar \coloneqq (\analytics(\orb\lambda), *_\hbar, \overline{\argumentdot\vphantom{\lambda}})$,
but would like to point out that they also hold for 
$(\Pol(\orb\lambda), *_\hbar, \overline{\argumentdot\vphantom{\lambda}})$. 

\begin{theorem}\label{theo:positiveLinearFunctionals}
	Assume that $\orb\lambda$ is a semisimple coadjoint orbit of 
	a real connected semisimple Lie group $G$.
	Assume further that $\lie h$ is a compact Cartan subalgebra,
	and that all roots 
	(with respect to the complexification $\smash{\hat{\lie h}}$ of $\lie h$)
	in $\smash{\hat\Delta}$ are non-compact.
	Let $*_\hbar$ be the star product constructed with respect
	to the ordering for which $\alpha \in \hat\Delta$ is positive
	if and only if $(\alpha, \I \lambda) < 0$.
Then there is a constant $M>0$ such that 
	for all $\xi \in \orb\lambda$ and $\hbar \in (0, M) \setminus P_\lambda$
	the point evaluation at $\xi$
	is a positive linear functional 
	$\mathrm{ev}_\xi \colon \algebra A_\hbar \to \mathbb C$.
\end{theorem}

\begin{proof}
	Since $(\alpha, \I \lambda) < 0$ for all $\alpha \in \hat\Delta^+$,
	it follows that $- \I (\lambda,\mu) > 0$ holds for all 
	$\mu \in \mathbb N_0 \hat\Delta^+ \setminus \lbrace 0 \rbrace$.
	There are only finitely many $\mu \in \mathbb N_0 \hat\Delta^+$
	with $(\rho, \mu) - \frac 1 2 (\mu,\mu) > 0$,
 	thus we can choose $M > 0$ such that
	$- \smash{\frac \I \hbar} (\lambda,\mu) > (\rho, \mu) - \smash{\frac 1 2 (\mu,\mu)}$
	holds for all $\mu \in \smash{\mathbb N_0 \hat\Delta^+} \setminus\lbrace 0 \rbrace$
	and $\hbar \in (0, M) \setminus P_\lambda$.
	But this says precisely that $p_{\I \lambda/\hbar}(\mu) > 0$,
	and therefore $\smash{p_{\I\lambda/\hbar}^w(\alpha_w)} > 0$ 
	for all $\smash{w \in \tilde \wordset}$.
For a non-compact root we have $\smash{\overline{X_\alpha}} = Y_\alpha$ according 
	to \eqref{eq:tangentVectorsToOrbit:nonCompact}. Consequently, if $g \in G$ is
	such that $\xi = \Ad^*_g(\lambda)$, then
	\begin{align*}
	\ev_\xi(f \starhbar \overline f) 
	&=
	\sum_{\ell=0}^\infty \Psi\bigg(\sum_{w \in \tilde \wordset_\ell} p_{\I\lambda/\hbar}^w(\alpha_w)^{-1}
	\tilde \pi^+(X_w) \tensor \tilde \pi^-(Y_w) \bigg)(f,\overline f)(\xi) \\
	&= 
    \sum_{\ell=0}^\infty \sum_{w \in \tilde\wordset_\ell} p_{\I\lambda/\hbar}^w(\alpha_w)^{-1}
    \leftinv{X_w} (\pi^* f)(g) \cdot \leftinv{Y_w} (\pi^* \overline f)(g)
	\\ 
	&= 
	\sum_{\ell=0}^\infty \sum_{w \in \tilde\wordset_\ell} p_{\I\lambda/\hbar}^w(\alpha_w)^{-1}
	\leftinv{X_w} (\pi^* f)(g) \cdot \overline {\leftinv{X_w} 
		(\pi^* f)(g)} \\
	&\geq 0 
	\end{align*}
	holds for all $f \in \analytics(\orb\lambda)$. 
\end{proof}

\begin{example}[\boldmath$\mathbb D^n$]
	It is straightforward to check that the choices made to quantize
	the hyperbolic disc in \autoref{ex:Dn}
	are such that $\lie h$ is compact, such that every root in $\hat\Delta$ is non-compact,
	and such that $\alpha \in \hat\Delta$ is positive iff $(\alpha, \I \lambda) < 0$.
Therefore the previous theorem implies the existence of a constant $M > 0$ 
	such that all point evaluation functionals are positive if $\hbar \in (0, M)$.
	
	We can prove a stronger result by using the formula for $F_\hbar$ derived in 
	\autoref{theo:example:oppositeOrdering}.
	If $\hbar \in (0,\infty)$ then all the coefficients
	appearing in this formula are positive, and so point evaluations are positive
	for all $\hbar \in (0,\infty)$.
\end{example}
Note that a similar proof does not work for $\mathbb{CP}^n$
since some of the coefficients in \eqref{eq:example:standardOrdering} are negative.
Indeed, one can use the appearing negative coefficients to show
that no point evaluation functional is positive on $\mathbb{CP}^n$
for $\hbar \in (0, \infty) \setminus P_\lambda$.
 
\subsection{A generalized Wick rotation}
\label{subsec:wickrotation}

In this subsection we want to state an immediate corollary
of the construction in the previous sections.
Let $\lie g_1$, $\lie g_2$ be two real semisimple
Lie algebras with the same complexification $\smash{\hat{\lie g}}$.
Assume $\lambda \in \lie g_1^* \cap \lie g_2^*$ 
where we view $\lie g_1^*$ and $\lie g_2^*$ as subspaces of $\hat{\lie g}^*$.
Denote the coadjoint orbits in $\lie g_1^*$ and $\lie g_2^*$ through $\lambda$
by $\orb\lambda^1$ and $\orb\lambda^2$, respectively.
There is an isomorphism $\Pol(\orb\lambda^1) \to \Pol(\orb\lambda^2)$
given by composing the map
$\Pol(\orb\lambda^1) \ni p \mapsto \hat p \in \smash{\Pol(\orbext\lambda)}$
with the restriction to $\orb\lambda^2$.
Here $\orbext\lambda$ is the complex extension of $\orb\lambda$.
It turns out that this isomorphism is still an isomorphism
of both the uncompleted and completed quantum algebras.

\begin{theorem}\label{theo:wickrotation}
	Let $\lie g_1$ and $\lie g_2$ be two real semisimple Lie algebras
	with a common complexification $\hat{\lie g}$ and assume that
	$\lambda \in \lie g_1^* \cap \lie g_2^*$ is semisimple.
Then the algebras 
	$(\Pol(\orb\lambda^1), *^1_\hbar)$ and 
	$(\Pol(\orb\lambda^2), *^2_\hbar)$,
	and also the algebras
	$(\analytics(\orb\lambda^1), *^1_\hbar)$ and 
	$(\analytics(\orb\lambda^2), *^2_\hbar)$,
	constructed with respect to the same Cartan subalgebra
	$\lie h \subseteq \lie g_1 \cap \lie g_2$
	and the same ordering, are isomorphic.
\end{theorem}
\begin{proof}
	Both algebras are isomorphic to 
	$(\Pol(\orbext\lambda), \starhat)$ or
	$(\Hol(\orbext\lambda), \starhat)$.
\end{proof}

\begin{example}[\boldmath$\mathbb{CP}^n$ and $\mathbb D^n$] \label{ex:wickrotation:coadjorb}
	We know from \autoref{ex:CPn} and \autoref{ex:Dn} that $\mathbb{CP}^n$ and 
	$\mathbb D^n$ are coadjoint orbits of the Lie groups $\group{SU}(1+n)$ and 
	$\group{SU}(1,n)$ through the same element, and that $\SL[1+n]$
	is a common complexification. So the previous proposition 
	implies that the star product algebras on $\mathbb{CP}^n$ and $\mathbb D^n$ are 
	isomorphic if we choose the same ordering in the construction of the star 
	products.
	
	The ordering that induces a K\"ahler complex structure on $\mathbb{CP}^n$,
	induces the complex structure on $\mathbb D^n$
	that is the opposite of the K\"ahler complex structure.
	Therefore the associated star product on $\mathbb D^n$ is of pseudo Wick type
	with respect to this opposite complex structure,
	and therefore of anti-Wick type for the K\"ahler complex structure.
	(A star product is of anti-Wick type 
	if the first argument is derived in antiholomorphic directions and the 
	second argument is derived in holomorphic ones.)
	Consequently, the algebra $\analytics(\mathbb{CP}^n)$ with the
	Wick type star product is isomorphic to the algebra $\analytics(\mathbb D^n)$ 
	with the anti-Wick type star product. 
	Similarly, the algebra $\analytics(\mathbb{CP}^n)$ with the anti-Wick type star product
	is isomorphic to the algebra $\analytics(\mathbb D^n)$ with the Wick type star product.
	
	One can also construct an isomorphism between the Wick type star product for $\hbar$
	and the anti-Wick type star product for $-\hbar$,
	both on the hyperbolic disc and the complex projective space.
	Composing with these isomorphisms shows that
	the Wick type star product for $\hbar$ on $\mathbb{CP}^n$ is isomorphic to
	the Wick type star product for $-\hbar$ on $\mathbb D^n$.
\end{example}
Note that \autoref{theo:wickrotation} only gives an algebra homomorphism 
between $\Pol(\orb\lambda^1)$ and $\Pol(\orb\lambda^2)$,
or between $\analytics(\orb\lambda^1)$ and $\analytics(\orb\lambda^2)$.
If we view these algebras as $^*$-algebras with the star involution
considered in \autoref{subsec:positiveLinearFunctionals}
then they are in general not $^*$-isomorphic!
One can see this for example by proving
that the point evaluation functionals on $\mathbb{CP}^n$ are not positive
for $\hbar \in (0, \infty) \setminus P_\lambda$. 	
\begin{appendices}
\section{\texorpdfstring{Proofs, \boldmath$G$-finite functions, and complex structures}{Proofs, G-finite functions, and complex structures}}
\label{appendices}
In \appendixref{appendix:diffops} we prove 
\autoref{theo:leftInvDiffOps:complex} and
\autoref{theo:diffops:bijection:real}.
In \appendixref{appendix:gfinites} we prove 
\autoref{theo:pullbackOfPolynomials} using the concept of $G$-finite functions.
Finally we recall some facts about complex structures
on coadjoint orbits in \appendixref{appendix:complexstructures}.

\subsection{Proofs of \autoref{theo:leftInvDiffOps:complex} and 
	\autoref{theo:diffops:bijection:real}} 
\label{appendix:diffops}

Let $M$ be a manifold.
For $f \in \Cinfty(M)$ we define 
$M_f \colon \Cinfty(M) \to \Cinfty(M)$, $f' \mapsto f f'$ and 
$\smash{M_f^{i}} = \id^{\times (i-1)} \times M_f \times \id^{\times(k-i)} \colon
\Cinfty(M)^k \to \Cinfty(M)^k$.

\begin{definition}\label{def:nDifferentialOperators}
	Let $M$ be a manifold.
	For a multiindex $K = (K_1, \dots, K_k) \in \mathbb Z^k$
	we define $\nDiffop_{K}(M) = \lbrace 0 \rbrace$ if some $K_i <0$
	and otherwise we define inductively
	\begin{multline}
	\nDiffop_{K}(M) = \lbrace D \colon \Cinfty(M)^{k} \to \Cinfty(M) 
	\mid M_f \circ D - D \circ M_f^{i} \in \nDiffop_{K - E_i}(M) \\
	\text{ for all $f \in \Cinfty(M)$ and $1 \leq i \leq k$}\rbrace 
	\punkt
	\end{multline}
	Here $(K-E_i)_j = K_j - \delta_{ij}$
	where $\delta_{ij}$ is $1$ if $i=j$ and $0$ otherwise.
	Elements of $\nDiffop_K(M)$ are called \emph{$k$-differential operators
	of degree $K$}.
	A map $D\colon \Cinfty(M)^k \to \Cinfty(M)$ is said to be a \emph{$k$-differential 
	operator} if it is a $k$-differential operator of some degree $K$.
	The space of $k$-differential operators is denoted by $\nDiffop(M)$.
\end{definition}
It follows that a $k$-differential operator is local in every argument,
so that it can be restricted to any open subset.
In a chart $U \subseteq M$ with local coordinates $(x^1, \dots, x^n)$,
a $k$-differential operator $D$ of degree $K$ can be written as
\begin{equation} \label{eq:nDifferentialOperators:charts}
D (f_1, \dots, f_k) =
{\sum}_{I_1, \dots, I_k \in \mathbb N_0^n}
c_{I_1, \dots,I_k}
\partial_x^{I_1} f_1 \cdot \ldots \cdot \partial_x^{I_k} f_k
\end{equation}
where $c_{I_1, \dots, I_k} \in \Cinfty(M)$ and $c_{I_1, \dots, I_k} = 0$ if 
$\abs{I_i} > K_i$ for some $1 \leq i \leq k$.
For a multiindex $J\in\mathbb N_0^n$ we used 
$\partial_x^J \coloneqq \partial_{x^1}^{J_1} \dots \partial_{x^n}^{J_n}$ and
$\partial_{x^i} \coloneqq \frac{\partial}{\partial x^i}$.
Conversely, an operator $D \colon \Cinfty(M)^k \to \Cinfty(M)$
that has this form in any chart is $k$-differential of order $K$.
A $k$-differential operator $D$ on a complex manifold $M$ is 
holomorphic if, in local holomorphic coordinates $(z^1, \dots, 
z^n)$, we have 
\begin{equation*}
D(f_1, \dots, f_k) = {\sum}_{I_1, \dots, I_k \in \mathbb N_0^n} c_{I_1, \dots, 
	I_k}
\partial_z^{I_1} f_1 \cdot \ldots \cdot \partial_z^{I_k} f_k
\end{equation*}
with all $c_{I_1, \dots, I_k}$ being holomorphic. 
Here $\partial_z^J = \partial_{z^1}^{J_1} \dots \partial_{z^n}^{J_n}$
and $\partial_{z^i} = \frac{\partial}{\partial z^i}$.
Equivalently, $D$ is holomorphic if 
$D$ maps $\Hol(U)^k$ into $\Hol(U)$ and 
$D \at U \circ M_f^{i} - M_f \circ D \at U = 0$ 
for all open subsets $U \subseteq M$ and all antiholomorphic functions $f$ on $U$.
We write $\nDiffop_\holo(M)$ for the space
of holomorphic $k$-differential operators.

We say a $k$-differential operator is of order
$K \in \mathbb Z^k$ at a point $p \in M$ if, when written in a local
chart $U$ around $p$ as in \eqref{eq:nDifferentialOperators:charts},
we have $c_{I_1,\dots, I_k}(p) = 0$ whenever $\abs{I_j} > K_j$ for some $1 \leq 
j \leq k$.

If $I_1 ,\dots, I_k, J, K \in \mathbb N_0^n$ are all multiindices, 
we write $J \leq K$
if $J_i \leq K_i$ for all $1 \leq i \leq n$.
If $X_1, \dots, X_n \in \lie g$,
then we use $X^J$ as a shorthand for $X_1^{J_1} \dots X_n^{J_n} \in \univ\lie g$
and $X^{I_1 \tensor \dots \tensor I_k}$ as a shorthand for 
$X^{I_1} \tensor \dots \tensor X^{I_k} \in (\univ \lie g)^{\tensor k}$.

\begin{proofof}[Proof of \autoref{theo:leftInvDiffOps:complex}:]
	Choose a basis $\lbrace X_1, \dots, X_n\rbrace$ of $\lie g$.
	It follows from the Poincar\'e--Birkhoff--Witt theorem that 
	$\lbrace X^{I_1 \tensor \dots \tensor I_k} \mid I_1, \dots, I_k \in \mathbb 
	N_0^n \rbrace$ is a basis of $(\univ \lie g)^{\tensor k}$.
	Moreover, $\lbrace \leftinvholo{X_1} \at e, \dots,\allowbreak \leftinvholo{X_n} \at e \rbrace$
	is a basis of the tangent space $\smash{\Tangent_e^{(1,0)} G}$ and we can choose
	a complex chart $U$ around $e$ with local coordinates $(z^1, \dots, z^n)$
	such that $\partial_{z^i} \at e = \leftinvholo{X_i} \at e$.
	
	Assume
	$\vec u = \sum_{I_1, \dots, I_k \in \mathbb N_0^n} c_{I_1, \dots, I_k} 
	X^{I_1 \tensor \dots \tensor I_k} \neq 0$
	with only finitely many $c_{I_1, \dots, I_k} \neq 0$.
	Choose $I_1, \dots, I_k$ in such a way that 
	$c_{I_1, \dots, I_k} \neq 0$ and 
	$c_{J_1, \dots, J_k} = 0$ whenever $I_i \leq J_i$ 
	and $(I_1, \dots, I_k) \neq (J_1, \dots J_k)$.
For $\vec f = (z^{I_1}, \dots, z^{I_k}) \in \Cinfty(U)^{\times k}$ we compute
	$\leftinvholo{\vec u} \vec f(e) = I_1! \dots I_k! c_{I_1, \dots, I_k} \neq 0$.
So $\leftinvholo{\vec u} \neq 0$ and $\leftinvholo{(\argumentdot)}$ is injective.
	
	Note that 
	$\leftinvholo{(X^{I_1})} f_1 \cdot \ldots \cdot \leftinvholo{(X^{I_k})} f_k
	= \partial_z^{I_1} f_1 \cdot \ldots \cdot \partial_z^{I_k} f_k 
	+ D'(f_1, \dots, f_k)$
	where $D'$ is a holomorphic $k$-differential operator whose order at $e$
	is strictly smaller than $(\abs{I_1}, \dots, \abs{I_k})$.
	For any holomorphic $k$-differential operator $D$ we can therefore, by induction,
	find coefficients $c_{I_1, \dots, I_k} \in \mathbb C$,
	only finitely many of which are non-zero,
	such that 
	\begin{equation*}
	D (f_1, \dots, f_k) (e) 
	= \sum_{I_1, \dots, I_k \in \mathbb N_0^n} c_{I_1, \dots, I_k} 
	\leftinvholo{(X^{I_1})} f_1(e) \cdot\ldots\cdot \leftinvholo{(X^{I_k})} f_k(e)
	\end{equation*}
	holds for all $f_1, \dots, f_k \in \Cinfty(G)$.
	In other words, $D$ and the left-invariant differential operator
	$\sum_{I_1, \dots, I_k \in \mathbb N_0^n}
	\leftinvholo{(c_{I_1, \dots, I_k} X^{I_1 \tensor \dots \tensor I_k})}$ 
	agree at $e$.
	So if $D$ is also left-invariant, then these operators agree everywhere on $G$,
	proving surjectivity.
\end{proofof}
The proof of \autoref{theo:diffops:bijection:real} is similar.
We need the following lemma to simplify the local calculations.

\begin{lemma}
	\label{theo:diffops:specialCoordinates}
	Let $G$ be a complex Lie group with Lie algebra $\lie g$, 
	and assume that $H$ is a closed complex Lie subgroup of $G$
	with Lie algebra $\lie h$. 
	Given a basis $B = \lbrace X_1, \dots, X_n \rbrace$ of $\lie g$ 
	such that $B' = \lbrace X_{n-r+1}, \dots, X_n \rbrace$
	is a basis of $\lie h$
	one can choose a neighbourhood $U$ of $e$ in $G$ and 
	complex coordinates $z=(z^1, \dots, z^n)$ on $U$ such that
	\begin{lemmalist}
		\item\label{theo:diffops:specialCoordinates:i}
		for any $g \in U$ its fiber $g H \cap U$ is given locally as 
		$(\lbrace z(g)\rbrace + \lbrace 0 \rbrace \times \mathbb C^{r}) \cap z(U)$,
\item\label{theo:diffops:specialCoordinates:ii}
		the left-invariant holomorphic vector fields agree with
		coordinate vector fields at $e \in G$,
		that is $\leftinvholo{X_i} \at e = \partial_{z^i}\at e$.
	\end{lemmalist}
\end{lemma}

\begin{proof}
	It is well known that $\pi \colon G \to G / H$ is a principal bundle. 
	Therefore we can choose a local trivialization $\chi\colon\pi^{-1}(V) \to V \times H$
	on a small neighbourhood $V$ of $e H$ in $G / H$.
Choosing coordinates on $V$
	(after possibly shrinking $V$ first)
	and on a neighbourhood $W$ of the identity in $H$, 
	we obtain coordinates $z'$ on $U \coloneqq \chi^{-1}(V \times W) \subseteq G$
	satisfying property \refitem{theo:diffops:specialCoordinates:i}.
Since all $\leftinvholo{X_i}$ are linearly independent we can write 
	$\leftinvholo{X_i} \at e = A_{ij} \partial_{(z')^j} \at e$
	for some invertible matrix $A$ and since $\leftinvholo{X_i}$
	is tangential to $H \subseteq G$ for $i > n-r$,
	it follows that $A_{ij} = 0$ for $i > n-r$, $j \leq n-r$.
Then the coordinates $z \coloneqq (A^{-1})^{\mathrm{T}}	z'$
	satisfy both properties of the lemma.
\end{proof}
Let $\pi \colon G \to G / H$.
Given coordinates as in the previous lemma we may identify $\pi(U)$ locally with 
$\lbrace (z^1(g), \dots, z^{n-r}(g), 0, \dots, 0) \mid g \in U \rbrace$.
Then $(z^1, \dots, z^{n-r})$ descend to coordinates on $\pi(U)$
and $\pi$ is, with respect to these coordinates, given by the projection to the first $n-r$ coordinates.

\begin{lemma} The map $\Psi$ from \autoref{theo:diffops:bijection:real} is 
	injective.
\end{lemma}
\begin{proof}
	Let $r = \dim \lie h$ and $n = \dim \lie g \geq r$.
	We can choose a basis $B  = \lbrace X_1, \dots, X_n \rbrace$
	of $\lie g$ such that $B' = \lbrace X_{n-r+1}, \dots, X_n \rbrace$
	is a basis of $\lie h$.
Recall from the proof of \autoref{theo:leftInvDiffOps:complex} 
	that 
	$\lbrace 
	X^{I_1 \tensor \dots \tensor I_k} \mid I_1, \dots, I_k \in \mathbb N_0^n 
	\rbrace$
	is a basis of $(\univ \lie g)^{\tensor k}$. Furthermore,
	\begin{equation*}
	\lbrace X^{I_1 \tensor \dots \tensor I_k} \mid I_1, \dots, I_k \in \mathbb 
	N_0^n, (I_i)_j > 0 \text{ for some $1 \leq i \leq k$ and some $j > n-r$} 
	\rbrace
	\end{equation*}
	is a basis of the ideal $I$ defined just before 
	\autoref{theo:diffops:Hinvariance} and
	\begin{multline*}
	\lbrace X^{I_1 \tensor \dots \tensor I_k} \mid I_1, \dots, I_k \in \mathbb 
	N_0^n, (I_i)_j =  0 \text{ for all $1 \leq i \leq k$, $j > n-r$}\rbrace 
	= \\ =
	\lbrace X^{I_1 \tensor \dots \tensor I_k} \mid I_1, \dots, I_k \in \mathbb 
	N_0^{n-r}\rbrace
	\end{multline*}
	is a basis of a complement $C$ of $I$ in $(\univ \lie g)^{\tensor k}$.
	Injectivity of $\Psi$ means that $0$ is the only element of $C$ on which 
	$\Psi$ vanishes.
	
	So to prove that $\Psi$ is injective, it suffices 
	to find, for any non-zero  
	\begin{equation*}
	\vec u = \sum_{I_1, \dots, I_k \in \mathbb N_0^{n-r}} c_{I_1, \dots, I_k}
	X^{I_1 \tensor \dots \tensor I_k} \in C \komma
	\end{equation*}
	some open subset $U \subseteq G / H$ and some $k$-tuple of functions $\vec f \in \Cinfty(U)^{k}$
	such that $\Psi([\vec u]) (\vec f) \neq 0$.
Fix $\vec u \in C \setminus \lbrace 0 \rbrace$ and
	assume that $I_1, \dots, I_k \in \mathbb N_0^{n-r}$ are chosen 
	such that $c_{I_1, \dots, I_k} \neq 0$	and 
	such that for any multiindices $J_1, \dots, J_k \in \mathbb N_0^{n-r}$
	satisfying $I_i \leq J_i$ and $(I_1, \dots, I_k) \neq (J_1, \dots, J_k)$
	we have $c_{J_1, \dots, J_k} = 0$.
Choose coordinates $z = (z^1, \dots, z^n)$ around $e$ on $G$ 
	as in the previous lemma,
	and note that, as described just after this lemma, $(z^1, \dots, z^{n-r})$
	descend to coordinates $(y^1, \dots, y^{n-r})$ on $G / H$.	
	Set $\vec f = (y^{I_1}, \dots, y^{I_k})$,
	so that $\pi^* \vec f = (z^{I_1}, \dots, z^{I_k})$.
	This implies that
	\begin{equation*}
	\Psi([\vec u])(\vec f)(eH) 
	= 
	\leftinvholo {\vec u} (\pi^* \vec f)(e) 
	= 
	I_1! \dots I_k! c_{I_1, \dots, I_k} \neq 0 \punkt
	\end{equation*}\end{proof}

\begin{lemma} The map $\Psi$ from \autoref{theo:diffops:bijection:real} is 
	surjective.
\end{lemma}
\begin{proof}
	We claim that for any holomorphic $k$-differential operator $D$ on $G/H$
	we can find $\vec u \in (\univ \lie g)^{\tensor k}$ such that 
	\begin{equation*}
	\leftinvholo{\vec u} (\pi^* \vec f)(e) = \pi^*(D \vec f)(e)
	\end{equation*}
	holds for all $\vec f\in \Cinfty(G/H)^k$.
We prove this claim by induction on the order $K \in \mathbb Z^k$ of $D$ at $e H$.
	If $K_i < 0$ for some $1 \leq i \leq k$, 
	then $D=0$ and we can use $\vec u = 0$. 
	For the induction step, assume that the claim is already proven 
	for every holomorphic $k$-differential operator
	of order strictly smaller than $K$ at $eH$.
	Choose coordinates $z = (z^1, \dots, z^n)$ around $e$ on $G$ 
	as in \autoref{theo:diffops:specialCoordinates}
	and denote the coordinates on $G / H$ induced by $(z^1, \dots, z^{n-r})$
	by $y \coloneqq (y^1, \dots, y^{n-r})$.
	Locally we can write 
	\begin{equation*}
	D(f_1, \dots, f_k) = {\sum}_{I_1, \dots, I_k \in \mathbb N_0^{n-r}} c_{I_1, 
		\dots, I_k} \cdot 
	\partial_{y}^{I_1} f_1 \cdot \ldots \cdot \partial_{y}^{I_k} f_k
	\end{equation*}
	with $c_{I_1, \dots, I_k} \in \Cinfty(G / H)$ satisfying
	$c_{I_1, \dots, I_k}(e H) = 0$ whenever $\abs{I_i} > K_i$\
	for some $1 \leq i \leq k$.
Define a holomorphic $k$-differential operator $D_G$ on $G$ by
	\begin{equation*}
	D_G(f'_1, \dots, f'_k) 
	= 
	{\sum}_{I_1, \dots, I_k \in \mathbb N_0^{n-r}} 
	(c_{I_1, \dots, I_k} \circ \pi) \cdot  
	\partial_z^{I_1} f'_1 \cdot \ldots \cdot \partial_z^{I_k} f'_k \punkt
	\end{equation*} 
	Then $D_G (\pi^* \vec f) (e) = \pi^*(D \vec f)(e)$.
Set
	$\vec u_1 \coloneqq 
	\sum_{I_1, \dots, I_k \in \mathbb N_0^{n-r}}
	c_{I_1,	\dots, I_k}(\pi(e)) X^{I_1} \tensor \dots \tensor X^{I_k} 
	\in (\univ\lie g)^{\tensor k}$.
Note that $D'_G \coloneqq D_G - \leftinvholo{\vec u_1}$ 
	has a strictly smaller order than $D_G$ at $e$ 
	since $\leftinvholo{X_i} \at e = \partial_{z^i} \at e$.
There are functions $c'_{I_1, \dots, I_k} \in \Cinfty(G)$ such that we can 
	express $D'_G$ in local coordinates as 
	\begin{equation*}
	D'_G(f'_1, \dots, f'_k) = {\sum}_{I_1, \dots, I_k \in \mathbb N_0^n} 
	{c'_{I_1, \dots, I_k}} 
	\cdot \partial_z^{I_1} f'_1
	\cdot \ldots \cdot 
	\partial_z^{I_k} f'_k \punkt
	\end{equation*} 
	We obtain a $k$-differential operator $D'$ on $G / H$ of strictly 
	smaller order than $D$ at $eH$ by letting
	\begin{equation*}
	D'(f_1, \dots, f_k) 
	= {\sum}_{I_1, \dots, I_k \in \mathbb N_0^{n-r}} 
	{c'_{I_1, \dots, I_k}}(\argumentdot, 0) 
	\partial_y^{I_1} f_1
	\cdot \ldots \cdot 
	\partial_y^{I_k} f_k \punkt
	\end{equation*} 
	It fulfils $D'_G (\pi^* \vec f) (e) = \pi^*(D' \vec f)(e)$.
Using the induction hypothesis we find
	$\vec u' \in (\univ \lie g)^{\tensor k}$ such that
	$\leftinvholo{\vec u'} (\pi^* \vec f) (e) = \pi^*(D' \vec f)(e)$.
Now 
	\begin{multline*}
	\leftinvholo{(\vec u_1 + \vec u')} (\pi^* \vec f) (e)
	= (D_G - D'_G) (\pi^* \vec f)(e) + \pi^*(D' \vec f)(e) 
	= \\
	= \pi^*(D \vec f)(e) - \pi^*(D' \vec f)(e) + \pi^*(D' \vec f)(e)
	= \pi^*(D \vec f)(e)\komma
	\end{multline*}
	proving the claim.
	
	Assume that $D$ is in addition left-invariant.
	Writing $\leftact_g \colon G / H \to G / H$ 
	also for the action of $g \in G$ on $G/H$ we compute 
	\begin{multline*}
	\leftinvholo{\vec u} (\pi^* \vec f)(g) 
	= \leftact_g^* \leftinvholo{\vec u} (\pi^* \vec f)(e) 
	= \leftinvholo{\vec u} (\leftact_g^{\times k})^* (\pi^* \vec f)(e) 
	= \\ 
	= \leftinvholo{\vec u} \pi^*((\leftact_g^{\times k})^* \vec f)(e) 
	= \pi^*(D (\leftact_g^{\times k})^* \vec f)(e) 
	= \\
	= \pi^*(\leftact_g^* D \vec f)(e) 
	= \leftact_g^* \pi^*(D \vec f)(e) 
	= \pi^*(D \vec f)(g) \punkt
	\end{multline*}
	Thus $\leftinvholo{\vec u} (\pi^* \vec f) = \pi^*(D \vec f)$ holds
	for all $\vec f \in \Cinfty(G/H)^k$.
Finally, we need to show that $\vec u$ has the correct invariance 
	properties under the adjoint action of $H$. 
	Define $\rightact_g \colon G \to G$, $\rightact_{g'}(g) \coloneqq g g'$. 
	Since 
	$\rightact_h^* \pi^*(D \vec f) = \pi^*(D \vec f)$ for all $h \in H$ we obtain 
	$\rightact_h^* \leftinvholo{\vec u} \pi^* \vec f
	= 
	\leftinvholo{\vec u} \pi^* \vec f$
	and therefore
	\begin{equation*}
	\leftinvholo{(\Ad_h \vec u)} (\pi^* \vec f) (g)
	= (\leftinvholo{\vec u} \pi^* \vec f)(gh) 
	= \rightact_h^* \leftinvholo{\vec u} \pi^* \vec f(g) 
	= \leftinvholo{\vec u} \pi^* \vec f(g) 
	\end{equation*}
	for all $\vec f \in \Cinfty(G/H)^k$ and all $g \in G$,
	where the first equality follows as in the proof of \autoref{theo:diffops:Hinvariance}.
	This means that $\leftinvholo{(\Ad_h \vec u - \vec u)} (\pi^* \vec f) = 0$ 
	for all $\vec f \in \Cinfty(G/H)^k$,
	and therefore the proof of injectivity implies $\Ad_h \vec u - \vec u \in I$,
	or in other words $\vec u \in U_{\mathrm{inv}}$.
\end{proof}
 
\subsection[\texorpdfstring{$G$-finite functions}{G-finite 
	functions}]{\texorpdfstring{\boldmath$G$-finite functions}{G-finite 
		functions}}
\label{appendix:gfinites}

In this subsection we introduce $G$-finite functions on a Lie group $G$
and use them to prove \autoref{theo:pullbackOfPolynomials}.
The definition of $G$-finite functions uses only abstract properties of the Lie group $G$,
and is therefore independent of whether $G$ is explicitly realized by matrices or not.
For complex semisimple connected Lie groups a function is $G$-finite if and
only if it is a polynomial, and therefore $G$-finite functions give a characterization
of polynomials that is independent of the representation.
\begin{definition}[\boldmath$G$-finite functions]
	Let $M$ be a manifold with an action of a Lie group $G$. 
	Then $f \in \Cinfty(M)$ is said to be $G$-finite 
	if the vector space $\spann \lbrace g \acts f \mid g \in G \rbrace$ 
	is finite dimensional.
	We denote the space of $G$-finite functions on $M$ by $\finites{G}(M)$
	or just by $\finites{}(M)$ if $G$ is clear from the context.
\end{definition}
Here $g \acts f$ denotes the smooth function on $M$ defined by 
$(g \acts f)(m) = f(g^{-1} \acts m)$. 
Below, we use this definition only for $M = G$ and the action $\leftact$ 
or for $M = \orbext\lambda$ and the coadjoint action,
and will therefore not mention these actions explicitly.

\begin{lemma}\label{theo:polynomials:polynomialsAreGFinite}
	Let $G$ be a real or complex matrix Lie group and let $\orb\lambda$ be a coadjoint orbit of $G$.
	Then polynomials on $G$ are $G$-finite, 
	and polynomials on $\orb\lambda$ are also $G$-finite.
\end{lemma}
\begin{proof}
	Let $P_{ij} \colon G \to \mathbb C$, $X \mapsto X_{ij}$,
	and call such polynomials elementary in this proof.
	We compute 
	$(g \acts P_{ij})(h) 
	= 
	P_{ij}(g^{-1} h) 
	= 
	\sum_k (g^{-1})_{ik} h_{kj} 
	=
	\sum_k (g^{-1})_{ik} P_{kj}(h)$ for $g \in G$,
	so $g \acts P_{ij}$ is a linear combination of some elementary polynomials.
	If $p = P_{i_1 j_1} \dots P_{i_n j_n} \in \Pol(G)$ is a product of $n$ 
	elementary polynomials,
	then $g \acts p$ is in the linear span of products of $n$ many elementary 
	polynomials, which is a finite dimensional space.
	The statement for arbitrary polynomials follows by taking linear combinations.
	
	The action of $G$ on $\Pol(\orb\lambda)$
	is obtained by restricting the adjoint action of $G$ on $\Sym\lie g \cong \Pol(\lie g^*)$.
	The adjoint action preserves the degree of a symmetric tensor,
	so $\spann \lbrace \Ad_g X \mid g \in G \rbrace$ is finite dimensional for any $X \in \Sym\lie g$,
	and therefore $\spann \lbrace g \acts p \mid g \in G \rbrace$ is finite dimensional 
	for any $p \in \Pol(\orb\lambda)$.
\end{proof}

\begin{proposition} \label{theo:Gfinites}
	Let $G$ be a complex semisimple connected Lie group with coadjoint orbit $\orb\lambda$.
	Then $G$-finite holomorphic functions on $\orb\lambda$ are polynomials.
\end{proposition}
\begin{proof}
	$\Hol(\orb\lambda)$ is isomorphic to $\Hol(G)^{G_\lambda}$
	as a $G$-module.
	The restriction to a maximal compact Lie subgroup $K \subseteq G$
	is an injective $K$-module homomorphism to $\Lzwei(K)$, the square-integrable 
	functions on $K$ with respect to the left-invariant Haar measure,
	so that we may view 
	$\Hol(\orb\lambda)$ as a $K$-submodule of $\Lzwei(K)$.
	In particular, it is completely reducible as a $K$-module and 
	therefore also as a $G$-module. Each irreducible module of highest weight 
	$\nu$ appears only finitely many times in $\Lzwei(K)$ and thus also in 
	$\Hol(\orb\lambda)$.
	
	The scalar product of $\Lzwei(K)$ is $K$-invariant and therefore any 
	irreducible modules of different highest weights are orthogonal.
	Restricting the scalar product to $\Hol(\orb\lambda)$ gives that 
	$\Hol(\orb\lambda)^\nu$ is orthogonal to $\Hol(\orb\lambda)^{\nu'}$ 
	if $\nu \neq \nu'$.
	
	Assume $f \in \finites{}(\orb\lambda)$ is holomorphic and not in 
	$\Pol(\orb\lambda)$. 
	We can without loss of generality assume that
	$f \in \finites{}(\orb\lambda)^\nu$ for some weight $\nu$.
	(Indeed, we can write $f = \smash{\sum_\mu f^\mu}$ with
	$f^\mu \in \finites{}(\orb\lambda)^\mu$ and
	only finitely many $f^\mu$ are non-zero because $f$ is $G$-finite.
	One of these $f^\mu$ is not in $\Pol(\orb\lambda)$.)
	We can choose $f$ orthogonal to 
	$\Pol(\orb\lambda)^\nu$ (which is finite dimensional) and therefore 
	orthogonal to $\Pol(\orb\lambda)$. 
	However, this space is dense in $\Hol(\orb\lambda)$ because polynomials 
	on $K$ are dense in $\Lzwei(K)$. So $f = 0$, a contradiction.
\end{proof}
\begin{corollary}
	Let $G$ be a complex semisimple connected Lie group.
	Then the pullback map
	$\pi^* \colon \Pol(\orb\lambda) \to \Pol(G)^{G_\lambda}$
	is an isomorphism.
\end{corollary}
\begin{proof}
	We have seen in the proof of  \autoref{theo:pullbackOfPolynomials} that $\pi^*$ 
	is well-defined and injective, 
	so it only remains to show that $\pi^*$ is surjective.
	Any element $f \in \Pol(G)^{G_\lambda}$ is $G$-finite by 
	\autoref{theo:polynomials:polynomialsAreGFinite}.
	Then its image under the $G$-equivariant isomorphism 
	$\pi_*\colon \Hol(G)^{G_\lambda} \to \Hol(\orb\lambda)$ 
	is also $G$-finite because finite dimensionality of 
	$\spann\lbrace g \acts f \mid g \in G\rbrace$
	implies finite dimensionality of
	$\spann\lbrace g \acts \pi_* f\mid g \in G\rbrace 
	= \spann\lbrace\pi_* (g \acts f)\mid g \in G\rbrace$.
	The previous proposition implies that the $G$-finite element
	$\pi_* f \in \Pol(\orb\lambda)$ is a polynomial.
	It is mapped to $f$ by $\pi^*$.
\end{proof}
With similar methods as in this subsection one can prove that 
$G$-finite functions on a complex semisimple connected Lie group $G$ 
coincide with polynomials on $G$.
Since the definition of $G$-finite functions does not depend on a representation 
of $G$ as a 
linear group, it follows that our definition of polynomials in 
\autoref{def:polynomials:lieGroup} is indeed 
independent of the representation.
The same result is true for a compact semisimple connected Lie group $K$. 
\subsection{Complex structures on real coadjoint orbits}
\label{appendix:complexstructures}

We have seen in \autoref{subsec:generalities} that a coadjoint orbit of a real Lie group $G$
always admits a $G$-invariant symplectic structure, in particular its dimension is even. 
In this subsection, we will see that a semisimple coadjoint orbit $\orb\lambda$ of a connected
semisimple real Lie group $G$ admits a $G$-invariant complex structure
if $G_\lambda$ is compact,
and that the set of such complex structures is in bijection to invariant orderings.
If $G$ is compact, then there is a unique $G$-invariant complex structure
that makes $\orb\lambda$ a Kähler manifold.
If $G$ is not compact, then $\orb\lambda$ might or might not admit a K\"ahler structure.
All results of this subsection are classical and well-known,
see for example \cite{bordemann.forger.roemer:1986a} for a summary.

Let $G$ be a real connected semisimple Lie group.
Assume that $\lambda \in \lie g^*$ is semisimple 
and that $G_\lambda$ is compact.
Then any Cartan subalgebra $\lie h \subseteq \lie g$
containing $\lambda^\sharp$ is contained in $\lie g_\lambda$ and therefore compact.
As usual, we denote the complexification of $\lie g$ by $\hat{\lie g}$ and 
let $\overline{\argumentdot\vphantom{\lambda}}$ be the complex conjugation 
of $\hat{\lie g}$ with respect to $\lie g$.
 
Recall that a root $\alpha \in \lie h^*$ is called \emph{compact} if the Killing form $B$
is negative definite on $\lie g \cap (\lie g^\alpha \oplus \lie g^{-\alpha})$,
and \emph{non-compact} if it is positive definite. (The root spaces 
$\lie g^\alpha$ are subspaces of the complexification $\hat{\lie g}$ of $\lie g$.)
We can always choose $X_\alpha \in \lie g^\alpha$ 
such that $B(X_\alpha, X_{-\alpha}) = 1$ 
and if $[X_\alpha, X_\beta] = N_{\alpha, \beta} X_{\alpha+\beta}$,
then $N_{-\alpha, -\beta} = - N_{\alpha, \beta}$ 
(see \cite[Section 3]{bordemann.forger.roemer:1986a}).
In this case, 
\begin{subequations}\label{eq:tangentVectorsToOrbit}
	\begin{align}\label{eq:tangentVectorsToOrbit:compact}
	-X_{-\alpha} &= \overline X_\alpha & &\text{and} & 
	\I(X_\alpha + X_{-\alpha}), X_\alpha - X_{-\alpha} &\in \lie g &
	&\text{if $\alpha$ is compact,} \\
	X_{-\alpha} &= \overline X_\alpha & &\text{and} &
	\I(X_\alpha - X_{-\alpha}), X_\alpha + X_{-\alpha} &\in \lie g &
	&\text{if $\alpha$ is 
	non-compact.}\label{eq:tangentVectorsToOrbit:nonCompact}
	\end{align}
\end{subequations}
Recall that $\hat\Delta$ is the set of roots that are not orthogonal to $\lambda$.

\begin{theorem}\label{theo:complexStructure:classification}
	Let $\orb\lambda$ be a coadjoint orbit
	of a real connected semisimple Lie group $G$.
	Assume that $G_\lambda$ is compact,
	and let $\lie h$ be a Cartan subalgebra of $\lie g$
	containing $\lambda^\sharp$.
Then $G$-invariant complex structures on $\orb\lambda$
	are in bijection with invariant orderings of $\hat\Delta$
	(i.e.\ choices of positive roots $\hat\Delta^+$ 
	that arise as $\hat\Delta^+ = \hat\Delta \cap \Delta^+$
	from an invariant ordering of $\Delta$ as defined in \autoref{def:invariantOrdering}).
\end{theorem}

\begin{proof}[Sketch]
	Introduce 
	$\lie m 
	=
	\bigoplus_{\alpha \in \hat \Delta} \lie g^\alpha 
	\cong
	\hat{\lie g}/\hat{\lie g}_\lambda$.
	Since taking fundamental vector fields (see \autoref{subsec:generalities})
	gives an isomorphism $\lie g / \lie g_\lambda\to \Tangent_\lambda \orb\lambda$,
	$\lie m$ is isomorphic to the complexified tangent space
	$\Tangent_\lambda^{\mathbb C} \orb\lambda$ and $\lie g \cap \lie m$
	is isomorphic to $\Tangent_\lambda \orb\lambda$.

	Given an invariant ordering of $\hat\Delta$, see \autoref{def:invariantOrdering}, 
	define $I \colon \lie m \to \lie m$ 
	by extending $X_\alpha \mapsto \I X_\alpha$ if $\alpha \in \hat\Delta^+$, 
	$X_\alpha \mapsto -\I X_\alpha$ if $\alpha \in \hat\Delta^-$ linearly. 
	Clearly $I^2 = - \id$.
	For both a compact and a non-compact root $\alpha$,
	$I$ restricts to an endomorphism of
	$\lie g \cap (\lie g^\alpha \oplus \lie g^{-\alpha})$,
	from which it follows that $I$ restricts to 
	a map $\lie g \cap \lie m \to \lie g \cap \lie m$, squaring to $- \id$.
	To prove that it extends to a $G$-invariant almost complex structure on $\orb\lambda$,
	it suffices to prove that $I$ is $G_\lambda$-invariant.
	By applying the analogue of \autoref{theo:stabilizerConnected} for compact connected
	semisimple Lie groups to a maximally compact subgroup of $G$ containing $G_\lambda$,
	it follows that $G_\lambda$ is connected, and it suffices to prove that
	$I$ is $\lie g_\lambda$-invariant, in the sense that $I([A, B]) = [A, I(B)]$ 
	holds for all $A \in \lie g_\lambda$ and $B \in \lie m$.
	This identity holds for $A \in \lie h$ since $I$ preserves the root spaces.
	So we only need to check it for $A = X_\alpha$ 
	and $B = X_\beta$ with $\alpha \in \Delta'$ and $\beta \in \hat\Delta$, 
	which is equivalent to the invariance of the ordering.
	Finally, one uses that $\alpha + \beta$ is positive if $\alpha, \beta \in \hat\Delta$
	are positive to compute that the Nijenhuis torsion of $I$ vanishes,
	so $I$ is indeed a complex structure.

	Vice versa, a $G$-invariant complex structure $I$ on $\orb\lambda$
	determines a $\lie g_\lambda$-invariant map $I \colon \lie m \to \lie m$ with $I^2 = -\id$
	by restricting to the tangent space at $\lambda$ and complexifying.
	In particular $I$ is $\lie h$-invariant, and therefore preserves
	the root spaces, so $X_\alpha \mapsto \I c_\alpha X_\alpha$ with $c_\alpha = \pm 1$.
	Since $I$ preserves the real tangent space, we must have $c_\alpha = - c_{-\alpha}$.
	The Nijenhuis torsion of the complex structure vanishes,
	which implies that 
	$\hat\Delta^+ = \lbrace \alpha \in \hat\Delta \mid c_\alpha = 1 \rbrace$
	defines an ordering.
	Finally invariance under the whole Lie algebra $\lie g_\lambda$ gives
	that this ordering is invariant.
\end{proof}

\begin{proposition}\label{theo:complexStructure:Kaehler}
	If $\orb\lambda$ is a coadjoint orbit of a compact connected semisimple Lie group $K$, 
	then $\orb\lambda$ has a unique $K$-invariant complex structure $I$
	that makes $(\orb\lambda, I, \omega_{\mathrm{KKS}})$ a K\"ahler manifold,
	and this complex structure corresponds to an ordering for which 
	$\alpha \in \hat\Delta$ is positive if and only if $(\alpha, \I\lambda) > 0$.
\end{proposition}
Note that $\alpha$ attains purely imaginary values on $\lie k$, whereas 
$\lambda$ attains real values, so $(\alpha, \I \lambda) \in \mathbb R$.
The ordering for which $\alpha \in \hat\Delta$ is positive if $(\alpha, \I \lambda) > 0$ 
is standard (see \autoref{subsec:twist:general}).

\begin{proof}
	Since $K$ is compact, it follows that any root is compact.
	Given a $K$-invariant complex structure $I$,
	we associate the (not necessarily positive definite) metric
	$g(v,w) = \omega_{\mathrm{KKS}}(v, I w)$ and $\orb\lambda$ is a Kähler manifold
	if $g$ is positive definite.
	Since $I$ and $\omega_{\mathrm{KKS}}$ are $K$-invariant, so is $g$ and we may
	check positive definiteness on $\Tangent_\lambda \orb\lambda$.
	Identifying $\Tangent^{\mathbb C}_\lambda \orb\lambda$ with $\lie m$
	as in the proof of the previous proposition and extending $g$ complex linearly,
	we compute that
	$g(X_\alpha, X_\beta) 
	= 
	\omega_{\mathrm{KKS}}(X_\alpha, I X_\beta) 
	=
	c_\beta \lambda([X_\alpha, X_\beta])$
	for all $\alpha, \beta \in \hat \Delta$.
	This expression is non-zero only if $\alpha = -\beta$,
	and in this case
	$g(X_\alpha, X_{-\alpha}) = -\I c_\alpha \lambda(\alpha^\sharp) = -\I c_\alpha 
	\cdot(\alpha, 
	\lambda)$. Then 
	\begin{align*}
	g(\I(X_\alpha+X_{-\alpha}), \I(X_\alpha+X_{-\alpha})) &= 
	2\I 
	c_\alpha \cdot(\alpha, \lambda) 
	\shortintertext{and}
	g(X_\alpha-X_{-\alpha}, 
	X_\alpha-X_{-\alpha}) &= 2\I c_\alpha \cdot(\alpha, \lambda) \punkt
	\end{align*}
	So $g$ is positive definite if and only if $c_\alpha = 1$
	for all $\alpha\in\hat\Delta$ with $(\alpha, \I\lambda) > 0$.
\end{proof}
Note that the situation is more complicated if $G$ is non-compact, but $G_\lambda$ is compact, 
since we may then have both compact and non-compact roots. 
The condition for $g$ being positive definite then becomes $c_\alpha =1$ 
if either $\alpha$ is a compact root and $(\alpha, \I \lambda) > 0$ or 
if $\alpha$ is a non-compact root and $(\alpha, \I\lambda) < 0$. 
If these conditions define an invariant ordering, 
then $\orb\lambda$ has a $G$-invariant K\"ahler structure (which is automatically unique).
One can give more explicit criteria for when the conditions above define an 
invariant ordering, see \cite{bordemann.forger.roemer:1986a}, 
but we only need the following easy case.

\begin{corollary}\label{theo:complexStructure:Kaehler:nonCompact}
	Let $\orb\lambda$ be a coadjoint orbit of a connected semisimple Lie group $G$.
	Assume that $G_\lambda$ is compact, and that $\lie h$ is a Cartan subalgebra 
	containing $\lambda^\sharp$.
	If all roots in $\hat \Delta$ are non-compact,
	then $(\orb\lambda, I, \omega_{\mathrm{KKS}})$ is a K\"ahler manifold,
	where $I$ is the complex structure corresponding to the ordering
	for which $\alpha \in \hat\Delta$ is positive
	if and only if $(\alpha, \I \lambda) < 0$.
\end{corollary}
 \end{appendices}

\section*{Acknowledgements}
The author would like to thank Matthias Sch\"otz for many valuable discussions 
and helpful comments on an earlier version of this article.
He is extremely grateful to his advisor Ryszard Nest for many helpful and inspiring 
discussions on the content of this paper and related topics. 

\section*{Funding}
The author was supported by the Danish National Research Foundation 
through the Centre of Symmetry and Deformation (DNRF92).

{\footnotesize
	\renewcommand{\arraystretch}{0.5}

} 

\begin{thebibliography}{10}
		
		\bibitem {alekseev.lachowska:2005a}
		\textsc{Alekseev, A., Lachowska, A.: }\newblock \emph{Invariant
			{$\ast$}-products on coadjoint orbits and the {S}hapovalov pairing}.
		\newblock Comment.\ Math.\ Helv.  \textbf{80} (2005), 795--810. \doi{10.4171/CMH/35}
		
		\bibitem {bayen.et.al:1978a}
		\textsc{Bayen, F., Flato, M., Fr{{\o}}nsdal, C., Lichnerowicz, A., 
			Sternheimer,
			D.: }\newblock \emph{Deformation Theory and Quantization}.
		\newblock Ann.\ Phys.  \textbf{111} (1978), 61--151. \doi{10.1016/0003-4916(78)90224-5}
		
		\bibitem {beiser.roemer.waldmann:2007a}
		\textsc{Beiser, S., R{\"o}mer, H., Waldmann, S.: }\newblock 
		\emph{Convergence
			of the {W}ick Star Product}.
		\newblock Commun.\ Math.\ Phys.  \textbf{272} (2007), 25--52.
		\doi{10.1007/s00220-007-0190-x}
		
		\bibitem {beiser.waldmann:2014a}
		\textsc{Beiser, S., Waldmann, S.: }\newblock \emph{Fr{\'{e}}chet 
			algebraic
			deformation quantization of the Poincar{\'{e}} disk}.
		\newblock J.\ Reine Angew.\ Math.  \textbf{688} (2014), 147--207.  
		\doi{10.1515/crelle-2012-0052}
		
		\bibitem {bertelson.cahen.gutt:1997a}
		\textsc{Bertelson, M., Cahen, M., Gutt, S.: }\newblock 
		\emph{Equivalence of
			Star Products}.
		\newblock Class.\ Quant.\ Grav.  \textbf{14} (1997), A93--A107.
		\doi{10.1088/0264-9381/14/1a/008}
		
		\bibitem {bieliavsky.gayral:2015a}
		\textsc{{Bieliavsky}, P., {Gayral}, V.: }\newblock \emph{{Deformation
				Quantization for Actions of Kählerian Lie Groups}}.
		\newblock Mem.\ Am.\ Math.\ Soc.
		\textbf{236}.1115 (2015).
		\doi{10.1090/memo/1115}
		


		\bibitem {bordemann.forger.roemer:1986a}
		\textsc{Bordemann, M., Forger, M., R{\"o}mer, H.: }\newblock 
		\emph{Homogeneous
			{K}\"ahler manifolds: paving the way towards new supersymmetric 
			sigma models}.
		\newblock Commun.\ Math.\ Phys.  \textbf{102} (1986), 605--617.
		\doi{10.1007/BF01221650}
				
		\bibitem {cahen.gutt.rawnsley:1990a}
		\textsc{Cahen, M., Gutt, S., Rawnsley, J.: }\newblock 
		\emph{Quantization of
			K{\"{a}}hler Manifolds I: Geometric Interpretation of Berezin's
			Quantization}.
		\newblock J.\ Geom.\ Phys.  \textbf{7} (1990), 45--62.
		\doi{10.1016/0393-0440(90)90019-Y}
		
		\bibitem {cahen.gutt.rawnsley:1993a}
		\textsc{Cahen, M., Gutt, S., Rawnsley, J.: }\newblock 
		\emph{Quantization of
			K{\"{a}}hler Manifolds. II}.
		\newblock Trans.\ Am.\ Math.\ Soc.  \textbf{337}.1 (1993), 73--98.
		\doi{10.2307/2154310}
		
		\bibitem {cahen.gutt.rawnsley:1994a}
		\textsc{Cahen, M., Gutt, S., Rawnsley, J.: }\newblock 
		\emph{Quantization of
			K{\"{a}}hler Manifolds. III}.
		\newblock Lett.\ Math.\ Phys.  \textbf{30} (1994), 291--305.
		\doi{10.1007/BF00751065}
		
		\bibitem {cahen.gutt.rawnsley:1995a}
		\textsc{Cahen, M., Gutt, S., Rawnsley, J.: }\newblock 
		\emph{Quantization of
			K{\"{a}}hler Manifolds. IV}.
		\newblock Lett.\ Math.\ Phys.  \textbf{34} (1995), 159--168.
		\doi{10.1007/BF00739094}
		
		\bibitem {chari.pressley:1994a}
		\textsc{Chari, V., Pressley, A.: }\newblock
		\emph{A Guide to Quantum Groups}.
		\newblock\emph{Cambridge University Press}.
		\newblock Cambridge, 1994.
		
		\bibitem {crooks:2018a}
		\textsc{Crooks, P.: }\newblock \emph{Complex adjoint orbits in Lie 
			theory and geometry}.
		\newblock Expo.\ Math. \textbf{37}.2 (2019), 104-144.
		\doi{10.1016/j.exmath.2017.12.001}
		
		\bibitem {dewilde.lecomte:1983b}
		\textsc{DeWilde, M., Lecomte, P. B.~A.: }\newblock \emph{Existence of
			Star-Products and of Formal Deformations of the Poisson Lie Algebra 
			of Arbitrary Symplectic Manifolds}.
		\newblock Lett.\ Math.\ Phys.  \textbf{7} (1983), 487--496.
		\doi{10.1007/BF00402248}
		
		\bibitem {dolgushev:2005a}
		\textsc{Dolgushev, V.~A.: }\newblock \emph{Covariant and equivariant 
			formality 
			theorems}.
		\newblock Adv.\ Math. \textbf{191} (2005), 147--177.
		\doi{10.1016/j.aim.2004.02.001}
		
		\bibitem {me}
		\textsc{Esposito, C., Schmitt, P., Waldmann, S.: }\newblock 
		\emph{Comparison and Continuity of Wick-type Star Products
			on certain Coadjoint Orbits}.
		\newblock Forum Math. (2019).
		\doi{10.1515/forum-2018-0302}
		
		\bibitem {esposito.stapor.waldmann:2017a}
		\textsc{Esposito, C., Stapor, P., Waldmann, S.: }
		\newblock \emph{Convergence of the Gutt Star Product}.
		\newblock J.\ Lie Theory \textbf{27} (2017), 579--622.
		
		\bibitem {fedosov:1994a}
		\textsc{Fedosov, B.~V.: }\newblock \emph{A Simple Geometrical 
			Construction of	Deformation Quantization}.
		\newblock J.\ Diff.\ Geom.  \textbf{40} (1994), 213--238.
		\doi{10.4310/jdg/1214455536}
		
		\bibitem {fioresi.lledo:2001a}
		\textsc{Fioresi, R., Lledo, M.~A.: }\newblock \emph{On the deformation 
			quantization of coadjoint orbits of semisimple groups}.
		\newblock Pac.\ J.\ Math. \textbf{198} (2001), 411--436.
		\doi{10.2140/pjm.2001.198.411}
		
		\bibitem{goto:1950a}
		\textsc{Gotô, M.: }\newblock \emph{Faithful representations of Lie 
			groups II}.
		\newblock Nagoya Math. J. \textbf{1} (1950), 91--107.
		\doi{10.1017/S002776300002287X}
		
		\bibitem{hall:2003a}
		\textsc{Hall, B.~C.: }\newblock \emph{Lie Groups, Lie Algebras, and 
			Representations}.
		\newblock \emph{Graduate Texts in Mathematics} no. \textbf{222}.
		\newblock Springer-Verlag, Berlin, Heidelberg, New York, 2003.
		\doi{10.1007/978-3-319-13467-3}
		


		\bibitem{hoermander:1990b}
		\textsc{H{\"{o}}rmander, L.: }\newblock \emph{An Introduction to 
			Complex Analysis in Several Variables}.
		\newblock \emph{North Holland Mathematical Library}.
		\newblock North-Holland, Amsterdam, London, New York, Tokio, 3. 
		edition, 1990.
		
		\bibitem{humphreys:1995a}
		\textsc{Humphreys J.~E.: }\newblock \emph{Linear algebraic groups}.
		\newblock{Graduate Texts in Mathematics} no. \textbf{21}.
		\newblock Springer-Verlag, New York, Heidelberg, Berlin, 1975.
		\doi{10.1007/978-1-4684-9443-3}
		
		\bibitem {karabegov:1996a}
		\textsc{Karabegov, A.~V.: }\newblock \emph{Deformation Quantization with
			Separation of Variables on a K{\"{a}}hler Manifold}.
		\newblock Commun. Math. Phys.  \textbf{180} (1996), 745--755.
		\doi{10.1007/BF02099631}
		
		\bibitem {karabegov:1998c}
		\textsc{Karabegov, A.~V.: }\newblock \emph{{B}erezin's quantization on 
			flag manifolds and spherical modules}.
		\newblock Trans. Amer. Math. Soc.  \textbf{350}.4 (1998), 1467--1479.
		\doi{10.1090/S0002-9947-98-02099-6}
		
		\bibitem {karabegov:1999a}
		\textsc{Karabegov, A.~V.: }\newblock \emph{Pseudo-K\"{a}hler 
			Quantization on
			Flag Manifolds}.
		\newblock Commun. Math. Phys.  \textbf{200} (1999), 355--379.
		\doi{10.1007/s002200050533}
		
		\bibitem {kirillov:1962a}
		\textsc{{Kirillov}, A.: }\newblock \emph{Unitary representations of 
			nilpotent Lie groups}.
		\newblock Russ.\ Math.\ Surv. \textbf{17}.4 (1962), 53.
		\doi{10.1070/rm1962v017n04abeh004118}
		


		\bibitem {kontsevich:2003a}
		\textsc{Kontsevich, M.: }\newblock \emph{Deformation Quantization of 
			{P}oisson
			manifolds}.
		\newblock Lett.\ Math.\ Phys.  \textbf{66} (2003), 157--216.
		\doi{10.1023/B:MATH.0000027508.00421.bf}
		
		\bibitem{kraus.roth.schoetz.waldmann:2019a}
		\textsc{Kraus, D., Roth, O., Schötz, M., Waldmann, S.: }\newblock 
		\emph{A convergent star product on the Poincaré disc}.
		\newblock J.\ Funct.\ Anal. (2019).
		\doi{10.1016/j.jfa.2019.02.011}
		
		\bibitem{landsman:1998a}
		\textsc{Landsman, N.~P.: }\newblock \emph{Mathematical Topics between 
			Classical 
			and Quantum Mechanics}.
		\newblock \emph{Springer Monographs in Mathematics}.
		\newblock Springer-Verlag, Berlin, Heidelberg, New York, 1998.
		\doi{10.1007/978-1-4612-1680-3}
		
		\bibitem{onishchik.vinberg:1994a}
		\textsc{Onishchik, A.~L., Vinberg, E.~B.: }\newblock \emph{Lie Groups 
			and Lie Algebras III}
		\newblock \emph{Encyclopaedia of Mathematical Sciences} no. \textbf{41}.
		\newblock Springer-Verlag, Berlin, Heidelberg, 1994. 
		
		\bibitem{ostapenko:1992a}
		\textsc{Ostapenko, P.: }\newblock \emph{Inverting the Shapovalov form}.
		\newblock J.\ Algebra \textbf{147} (1992), 90--95.
		\doi{10.1016/0021-8693(92)90254-J}
		


		\bibitem{mudrov:2014a}
		\textsc{Mudrov, A.: }\newblock \emph{R-matrix and inverse Shapovalov 
			form}.
		\newblock J.\ Math.\ Phys. \textbf{57}.5 (2014), 051706.
		\doi{10.1063/1.4950894}
		
		\bibitem {natsume.nest:1999a}
		\textsc{Natsume, T., Nest, R.: }\newblock \emph{Topological Approach to 
			Quantum	Surfaces}.
		\newblock Commun.\ Math.\ Phys.  \textbf{202} (1999), 65--87.
		\doi{10.1007/s002200050575}
		
		\bibitem {natsume.nest.peter:2003a}
		\textsc{Natsume, T., Nest, R., Peter, I.: }\newblock \emph{Strict 
			Quantizations
			of Symplectic Manifolds}.
		\newblock Lett.\ Math.\ Phys.  \textbf{66} (2003), 73--89.
		\doi{10.1023/B:MATH.0000017652.90999.d8}
		
		\bibitem {nest.tsygan:1995a}
		\textsc{Nest, R., Tsygan, B.: }\newblock \emph{Algebraic Index Theorem}.
		\newblock Commun.\ Math.\ Phys.  \textbf{172} (1995), 223--262.
		\doi{10.1007/BF02099427}
		
		\bibitem {rieffel:1993a}
		\textsc{Rieffel, M.~A.: }\newblock \emph{Deformation quantization for 
			actions	of $\mathbb{R}^d$}.
		\newblock Mem.\ Amer.\ Math.\ Soc.  \textbf{106}.506 (1993).
		\doi{10.1090/memo/0506}
		
		\bibitem{rieffel:1998a}
		\textsc{Rieffel, M. A.}\newblock \emph{Questions on quantization}.
		\newblock Contemp.\ Math. \textbf{228} (1998).
		\doi{10.1090/conm/228}
		
		\bibitem {schoetz.waldmann:2018a}
		\textsc{{Schötz}, M., {Waldmann}, S.: }\newblock \emph{Convergent star
			products for projective limits of Hilbert spaces}.
		\newblock J.\ Funct.\ Anal. \textbf{274}.5 (2018), 1381 
		-- 1423.
		\doi{10.1016/j.jfa.2017.09.012}
		
		\bibitem {waldmann:2014a}
		\textsc{Waldmann, S.: }\newblock \emph{A nuclear Weyl algebra}.
		\newblock J.\ Geom.\ Phys.  \textbf{81} (2014), 10--46.
		\doi{10.1016/j.geomphys.2014.03.003}
		
		\bibitem {waldmann:2019a}
		\textsc{Waldmann, S.: }\newblock \emph{Convergence of Star Product: 
			From Examples to a General Framework}.
		\newblock ArXiv e-prints (Jan 2019).
		\arxiv{1901.11327}
		
	\end{thebibliography}
\end{document}